\documentclass[11pt,reqno]{amsart}
\usepackage{amsmath,amssymb,amscd,amsthm,graphicx,enumerate,tikz-cd}
\usepackage{hyperref}
\hypersetup{colorlinks=false, pdfborder={1 1 1}}
\hypersetup{breaklinks=true}
\usepackage{enumitem} 
\newlist{condenum}{enumerate}{1} 
\setlist[condenum]{label=\bfseries Condition \arabic*.,  ref=\arabic*, wide}
\usepackage[utf8]{inputenc}
\usepackage{amssymb}
\usepackage{amsthm}
\usepackage{mathrsfs}
\usepackage{amsfonts}
\usepackage{amsmath}
\usepackage{mathtools}
\usepackage{makecell}
\usepackage{xcolor}
\usepackage[margin=1.25in]{geometry}
\usepackage{lipsum}
\usepackage{soul}
\usepackage{caption}
\usepackage{tikz-cd}

\numberwithin{equation}{section}
\theoremstyle{plain}



\usepackage[utf8]{inputenc}
\usepackage{amssymb}
\usepackage{amsthm}
\usepackage{mathrsfs}
\usepackage{amsfonts}
\usepackage{amsmath}
\usepackage{mathtools}
\usepackage{makecell}
\usepackage{xcolor}
\usepackage{soul}
\usepackage{cancel}
\usepackage[margin=1.25in]{geometry}

\usepackage{lipsum}
\makeatletter
\def\ps@pprintTitle{%
 \let\@oddhead\@empty
 \let\@evenhead\@empty
 \def\@oddfoot{}%
 \let\@evenfoot\@oddfoot}
\makeatother

\newcommand{\R}{\mathbb{R}}

\newcommand{\ess}{\textnormal{ess}}
\newcommand{\rest}{\textnormal{rest}}

\newcommand{\LL}{\mathcal{L}}

\newcommand{\Rm}{\textnormal{Rm}}

\newcommand{\Ric}{\textnormal{Ric}}

\newcommand{\pt}{\partial_t}

\newcommand{\id}{\mathrm{id}}
\newcommand{\Ima}{\textnormal{Im}}
\newcommand{\Spec}{\textnormal{Spec}}
\newcommand{\tr}{\textnormal{tr}}

\newcommand{\eps}{\varepsilon}

\newcommand{\Di}{\textnormal{div}}
\newcommand{\lie}{\textnormal{div}^*}
\newcommand{\Ker}{\textnormal{Ker}}

\newcommand{\ds}{\left.\frac{d}{ds}\right|_{s=0}}
\newcommand{\Sym}{\textnormal{Sym}}

\newcommand{\Lie}{\textnormal{Lie}}

\newcommand{\td}{\widetilde}
\newcommand{\PP}{\mathcal{P}}

\numberwithin{equation}{section}

\newtheorem{theorem}{Theorem}[section]
\newtheorem{lem}[theorem]{Lemma}
\newtheorem{lemma}[theorem]{Lemma}
\newtheorem{remark}[theorem]{Remark}
\newtheorem{prop}[theorem]{Proposition}
\newtheorem{proposition}[theorem]{Proposition}
\newtheorem{claim}[theorem]{Claim}
\newtheorem{cor}[theorem]{Corollary} 

\theoremstyle{definition}
\newtheorem{defn}[theorem]{Definition}

\newtheorem*{theorem*}{Theorem}

\usepackage{xpatch}
\makeatletter
\xpatchcmd{\tableofcontents}{\contentsname \@mkboth}{\small\contentsname \@mkboth}{}{}
\xpatchcmd{\listoffigures}{\chapter *{\listfigurename }}{\chapter *{\small\listfigurename }}{}{}
\makeatother

\makeatletter
\def\blfootnote{\xdef\@thefnmark{}\@footnotetext}
\makeatother

\begin{document}

\title[Ricci flow and long-time Harmonic map heat flow]{Convergence of Ricci flow and long-time existence of Harmonic map heat flow}

\begin{abstract}
For an ancient Ricci flow asymptotic to a compact integrable shrinker, or a Ricci flow developing a finite-time singularity modelled on the shrinker, we establish the long-time existence of a harmonic map heat flow between the Ricci flow and the shrinker for all times. This provides a global parabolic gauge for the Ricci flow and implies the uniqueness of the tangent flow without modulo any diffeomorphisms.

We present two main applications: First, we construct and classify all ancient Ricci flows asymptotic to any compact integrable shrinker, showing that they converge exponentially.
Second, 
we obtain the optimal convergence rate at singularities modelled on the shrinker, characterized {by the first negative eigenvalue} of the stability operator {for the entropy}. 
In particular, we show that any Ricci flow developing a round $\mathbb S^n$ singularity converges at least at the rate $(-t)^{\frac{n+1}{n-1}}$.
\end{abstract}

\author{Kyeongsu Choi}
\address{Kyeongsu Choi\hfill\break School of Mathematics, Korea Institute for Advanced Study, 85 Hoegiro, Dongdaemun-gu, Seoul, 02455, South Korea
}
\email{choiks@kias.re.kr}

\author{Yi Lai}
\address{Yi Lai\hfill\break Department of Mathematics, UC Irvine, CA 92697, USA
}
\email{ylai25@uci.edu}

\maketitle

\tableofcontents

\section{Introduction}

\subsection{Long-time existence of Harmonic map heat flow}
The Ricci flow, introduced by Hamilton in 1982, is a family of Riemannian metrics $g(t)$, $t\in [0,T]$, on a manifold $M$ evolving by the parabolic partial differential equation $$\frac{\partial}{\partial t}g(t)=-2\textnormal{Ric}(g(t)).$$ Singularity analysis plays a crucial role in the Ricci flow and its topological applications. For example, Perelman resolved the Poincar\'e Conjecture and Geometrization Conjecture \cite{Pel1,Pel2}, by performing surgeries at the singularities and continuing the flow.
For a blow-up sequence \textit{towards} the singularity, one can obtain an ancient solution (defined for all negative times) as a subsequential limit, called a singularity model.
Perelman's 3D surgery is based on a good understanding of the geometry of all singularity models.

Ricci solitons are self-similar Ricci flows, which shrink, remain steady, or expand up to diffeomorphisms.
In particular, a gradient shrinking soliton (shrinker) is a Riemannian manifold $(M,\bar g,f)$ together with a function $f$ on $M$, such that
\[\Ric+\nabla^2 f=\tfrac12\bar g.\]
By Bamler's recent breakthrough work \cite{bamler2020entropy, bamler2023compactness, bamler2020structure},
the blow-up limit \textit{at} the singularity is a shrinking Ricci soliton (possibly with a singular set), called the \textbf{tangent flow} at the singularity; and the blow-down limit of an ancient non-collapsed flow is also a shrinking Ricci soliton, called its \textbf{tangent flow at infinity}.
We also say that these Ricci flows are \textit{asymptotic to} the shrinker.
For compact smooth shrinkers, Chan-Ma-Zhang \cite{CMZ1} showed that the tangent flow and tangent flow at infinity on them are unique \textit{modulo diffeomorphisms}, see also Sun-Wang's work in the K\"ahler case \cite{SW15}.

However, the uniqueness of tangent flow \textit{without modulo any diffeomorphisms} is needed to apply PDE methods effectively, such as to classify ancient Ricci flows and estimate the convergence rate to singularities. It requires to find a global parabolic gauge for all times under which the solution remains a small perturbation of the tangent flow.
In mean curvature flow (MCF), one can employ the ambient space as the global parabolic gauge, so this suffices to show that the shrinker is the unique tangent flow \textit{without modulo any Euclidean motions}. Schulze \cite{schulze2014} showed that if a mean curvature flow rescaled at a singularity has a sequence of time-slices converging to a compact shrinker then the shrinker is the limit of the rescaled flow, namely the tangent flow is unique \textit{without modulo any Euclidean motion}. Then, Colding-Minicozzi \cite{colding2015uniqueness} proved the uniqueness of tangent flow at cylindrical singularities. In other words, the rescaled MCF at a cylindrical singularity converges to only one cylinder as $\tau\to \infty$.
{In minimal surface, Allard-Almgren \cite{Allard1981OnTR} proved the uniqueness of tangent cone at an isolated singularity satisfying 
integrability condition for the cross section. Simon \cite{Simon1983AsymptoticsFA} extended the result for smooth cross sections. See also White \cite{White1983TangentCT} for two-dimensional minimizing currents. Note that in this paper we employ the integrability condition of Allard-Almgren for the Ricci flow as in Sesum \cite{sesum2006convergence}. See also Colding-Minicozzi \cite{Colding2013OnUO} for the unique blow-down of Ricci-flat metrics with Euclidean volume growth.}

Colding-Minicozzi's uniqueness of tangent flow is an important ingredient in the classification of asymptotically cylindrical {ancient MCFs}, see works of Angenent-Daskalopoulos-Sesum \cite{ADS_JDG,ADS_Annals}, Brendle-Choi \cite{brendle2019uniqueness}, Choi-Haslhofer-Hershkovits and Choi-Haslhofer-Hershkovits-White \cite{choi2022ancient,CHH_nowing_GT, CHH_Camb, CHHW}, Du-Haslhofer \cite{du2021hearing}, Choi-Daskalopoulos-Du-Sesum-Haslhofer \cite{CDD2022}, and Choi-Haslhofer \cite{CH2024_bubble}.
There are many other classification results for ancient MCFs whose tangent flow at infinity is unique but not cylindrical; see for example Choi-Mantoulidis \cite{CM22} and Choi-Sun \cite{CS20} for compact tangent flows, and Choi-Chodosh-Mantoulidis-Schulze \cite{chodosh2024mean} and Choi-Chodosh-Schulze \cite{chodosh2023mean} for noncompact tangent flows. See also the uniqueness results for tangent flows (at infinity) which are compact \cite{schulze2014} by Schulze, asymptotically conical \cite{CS2021_uniqueness, LZ2024_uniqueness} by Chodosh-Schulze and Lee-Zhao, singular \cite{lotay2022neck, stolarski2023structure} by Lotay-Schulze-Sz{\'e}kelyhidi and Stolarski, and with multiplicity \cite{neves2007, CDSZ2025} by Neves and Choi-Seo-Su-Zhao. 
In Ricci flow, 
notable examples of asymptotically cylindrical flows are the Bryant soliton, Perelman's oval, and the flying wings by the second-named author \cite{Lai2020_flying_wing,Lai2022_angles,Lai2022_O(2)}. However, the uniqueness of tangent flows remains open even for the compact integrable shrinkers. 

A key challenge in studying the uniqueness of the tangent flow in Ricci flow, as opposed to MCF, is the lack of strong parabolicity. 
The presence of diffeomorphism invariance means that the flow is only weakly parabolic, requiring an additional choice of gauge to control the evolution.
The parabolic gauge can be given by a harmonic map heat flow by the DeTurck’s trick: If there is a \textbf{harmonic map heat flow} (HMHF) $\chi_t$ between the Ricci flow $g(t)$ and a background metric $\bar g$, which satisfies
\[\partial_t \chi_t=\Delta_{g(t),\bar g}\chi_t,\]
then the push-forward of the Ricci flow by the HMHF to a background metric becomes a parabolic  \textbf{Ricci DeTurck flow} (RDTF). For an ancient or immortal Ricci flow, we need a long-time HMHF that exists for all times to provide the global parabolic gauge.

The short-time existence of a HMHF is well-known by the standard parabolic theory.
However, the long-time existence of a HMHF is a geometric problem rather than purely an analytical one, as it depends on the closeness between the Ricci flow $g(t)$ and the background shrinker $\bar g$.
In this paper, for any ancient Ricci flow (or a Ricci flow at a finite-time singularity) asymptotic to a compact integrable shrinker, we establish the long-time existence of a HMHF between the Ricci flow and the shrinker.

\begin{theorem}\label{t:uniqueness_of_tangent_flow}
Let $(M,\bar g)$ be a compact integrable shrinker. Let $g_\tau$ be an ancient (resp. immortal) rescaled Ricci flow asymptotic to $\bar g$.
Then there exists an ancient (resp. immortal) HMHF between $g_{\tau}$ and $\bar{g}$ that converges exponentially to $\id$ as $\tau$ goes to negative (resp. positive) infinity.
\end{theorem}

Note that for a Ricci flow $g(t)$, the
\textbf{rescaled Ricci flow} $\hat g(\tau)=e^{\tau}g({-e^{-\tau}})$ obtained by the change of time variable $t=-e^{-\tau}$ satisfies:
\[\frac{\partial}{\partial \tau}\hat g(\tau)=-2\textnormal{Ric}(\hat g(\tau))+\hat g(\tau).\] 
So the rescaled flow of a Ricci flow that forms a singularity at $t=0$ is an immortal flow.

We prove Theorem \ref{t:uniqueness_of_tangent_flow} based on the following convergence rate estimate of the rescaled Ricci flow to the shrinker. 
\begin{theorem}\label{thm:RF_exp}
    Let $(M,\bar g)$ be a compact integrable shrinker, and $g_\tau$ an ancient or immortal rescaled Ricci flow asymptotic to $\bar g$.
Then $g_{\tau}$ converges exponentially fast to $\bar{g}$ as $\tau$ goes to infinity.
\end{theorem}
We outline both theorems at the end of this section.
We remark that
Theorem \ref{t:uniqueness_of_tangent_flow} and \ref{thm:RF_exp} imply the {uniqueness of tangent flow}: Since the HMHF converges to $\id$ exponentially, and the rescaled Ricci flow converges to the shrinker exponentially, together they imply that the parabolic rescaled RDTF converges exponentially to the shrinker without modulo any diffeomorpshisms.

As applications of Theorem \ref{t:uniqueness_of_tangent_flow}, we use the ancient HMHF to classify ancient Ricci flows asymptotic to compact integrable shrinkers, and use the immortal HMHF to derive an optimal convergence rate estimate for Ricci flow singularities modelled on such shrinkers.

\subsection{Classification of ancient Ricci flows}
Let $(M,\bar g,f)$ be a compact shrinker. Let
$$L=\Delta_f+2\Rm*$$ 
be the elliptic operator in the linearized rescaled RDTF perturbation equation with background metric being the shrinker, which acts on symmetric 2-tensors.  
We say the eigentensors with positive, zero, and negative eigenvalue of $L$ are \textit{unstable, neutral, and stable}, respectively.
A shrinker is a critical point of Perelman's $\mu$-entropy.
By the $L^2_f$-orthogonal decomposition of $\textnormal{Sym}^2(M)=\Ima\,\Di^*_f\perp \Ker\,\Di_f$, we call the eigentensors in $\Ima\,\Di^*_f$ (Lie derivatives) to be \textit{generic}, and those in $\Ker\,\Di_f$ to be \textit{essential}.
Then the second variation of the $\mu$-entropy is positive at any unstable essential eigentensors, is zero at all generic eigentensors and neutral essential eigentensors, and is negative at any stable essential eigentensors.
Since the $\mu$-entropy is monotone non-decreasing under the Ricci flow, an ancient flow has larger $\mu$-entropy than its asymptotic shrinker. 

In fact, for any fixed compact shrinker, let $\textnormal{ess-index}(L)$ be the dimension of space of unstable essential eigentensors,
we construct an $\textnormal{ess-index}(L)$-parameter family of ancient Ricci flows. These flows converge exponentially to the shrinker in the rescaled flow. The construction is to first find ancient RDTFs by a fixed-point argument, then convert them to ancient Ricci flows. See works in MCF by Choi-Mantoulidis \cite{CM22} and Chodosh-Choi-Mantoulidis-Schulze \cite{chodosh2024mean}, and also \cite{CJL2024,mramor2021ancient,han2023ancient}.

\begin{theorem}[Existence of ancient Ricci flows]\label{t:existence_theorem_intro}
    Let $(M,\bar g,f)$ be a compact shrinker, and $\lambda>0$ be the smallest positive essential eigenvalue.
    There exists an $\textnormal{ess-index}(L)$-parameter family of pairwise different ancient rescaled Ricci flows asymptotic to $\bar g$, such that for each flow $g_\tau$, there exists a family of diffeomorphisms $\chi_\tau$ such that each $k\in \mathbb{N}$ satisfies
    \[\|\chi_\tau^*g_\tau-\bar g\|_{C^{k}(\bar g)}\le C_ke^{\lambda\tau},\]
    for some $C_k>0$.
    In particular, $\chi_\tau\equiv\id$ if $(M,\bar g)$ is positive Einstein.
\end{theorem}

Note that a Ricci flow remains a Ricci flow under dilations and rescalings.
If the shrinker is linearly stable in the entropy sense, meaning that $\Ric$ is the only essential unstable eigentensor of $L$, the Ricci flows in Theorem \ref{t:existence_theorem_intro} only differ from the Ricci flow of the shrinker by time-shiftings. On the other hand, if the shrinker is linearly unstable, then the essential index is at least two, i.e., $I=\textnormal{ess-index}(L)\ge2$. In this case, the $I$-parameter family of rescaled Ricci flows in Theorem \ref{t:existence_theorem_intro} gives a $(I-2)$-parameter family of pairwise non-isometric Ricci flows modulo dilations and time-shiftings.

By Theorem \ref{t:uniqueness_of_tangent_flow}, we classify all ancient flows coming out of a compact integrable shrinker: 
\begin{theorem}[Classification]\label{t:uniqueness}
   Let $(M,\bar g)$ be a compact integrable shrinker. Then any ancient rescaled Ricci flow asymptotic to $(M,\bar g)$ must be one of the flows in Theorem \ref{t:existence_theorem_intro}. 
\end{theorem}

The \textit{integrability} means that every essential neutral eigentensor arises as the derivative of a path of shrinking solitons. 
Note that integrability is a necessary condition in Theorem \ref{t:uniqueness}; for example, there exists an ancient Ricci flow asymptotic to the non-integrable shrinker $\mathbb{CP}^n$ with a slower-than-exponential convergence rate \cite{KS19,Kroencke2013StabilityOE}, thus it is not in Theorem \ref{t:existence_theorem_intro}.
The class of integrable shrinkers includes many well-known examples, such as $\mathbb S^n, \mathbb{HP}^n, Ca\mathbb{P}^2$.
{For ancient flows with a slower-than-exponential convergence rate, see the work of Choi–Hung \cite{choi2024thom} for equations with elliptic gauges.}

Ricci flows asymptotic to compact integrable \textit{Ricci-flat} manifolds are first considered by Sesum \cite{ses06}, under the assumption that the Ricci-flat manifold is linearly stable. Our method extends Sesum's result to shrinkers and removes the stability assumption. So first, this allows the presence of \textit{essential} unstable eigentensors beyond just the Ricci tensor.
See \cite{HM14,Has12,li2010stability,dai2005stability,dai2006stability,bamler2014stability,bamler2015stability,kroncke2015stability,guenther2002stability,schnurer2007stability,schnurer2010stability,Kroencke2013StabilityOE,CS05,hall2011linear} for discussions of stability for Ricci-flat manifolds and many more.

Moreover, while Ricci-flat manifolds have no generic unstable eigentensors (see e.g. \cite{ses06}), this is not true for shrinkers—even as simple as 
the round sphere $(\mathbb S^n(\sqrt{2(n-1)}), g_{\mathbb S^n})\subset \mathbb R^{n+1}$. If $n\ge3$,
there is a unique positive generic eigenvalue $\lambda_n=1-\tfrac{n}{2(n-1)}>0$, and the corresponding unstable generic eigenspace is spanned by $\{x^{i} g_{\mathbb S^n}\}_{i=1}^{n+1}$, where $\{x^{i}\}_{i=1}^{n+1}$ are the coordinates functions in $\mathbb R^{n+1}$, and $x^ig_{\mathbb S^n}=\LL_{\nabla x^i}g_{\mathbb S^n}$. These vector fields $\{\nabla x^{i}\}_{i=1}^{n+1}$ form a basis of the eigenspace of the operator $\Delta+\Ric$ with eigenvalue $\lambda_n$ acting on vector fields.

In fact, we discover that the \textit{generic stability} of 
$L$ are closely related to the stability of
\[\mathfrak L=\Delta_f+\tfrac12,\]
which reduces to $\Delta+\Ric$ when the shrinker is positive Einstein.
Note $\mathfrak{L}$ is the linearized operator of the HMHF equation. 
We find that for any eigentensor $\LL_X \bar g$ of $L$ with a non-zero eigenvalue, the vector field $X$ can be chosen as an eigenvector of $\mathfrak L$ with the same eigenvalue. This suggests an analogy between the generic eigentensors with the Euclidean motions in MCF, whose actions do not affect the entropy of the MCF solution.

Moreover, 
if $(M,\bar g)$ is positive Einstein, we establish the equivalence between the generic stability of $L$ and the stability of $\mathfrak L$.

\begin{theorem}\label{thm:generic_index}
    If $(M,\bar g)$ is positive Einstein, then for any $\lambda\neq0$, it is the eigenvalue of a generic eigentensor of $L$ if and only if it is an eigenvalue of $\mathfrak L$, and 
    the generic eigenspace $E_{\lambda,L}\cap\Ima\,\Di^*$ is isomorphic to the eigenspace $E_{\lambda,\mathfrak L}$. In particular, we have
    \[\textnormal{gen-index}(L)=\textnormal{index}(\mathfrak L).\]
\end{theorem}

The existence of \textit{generic} unstable eigentensors creates  
a new challenge in Theorem \ref{t:uniqueness}, in contrast to Ricci-flat manifolds and MCF.
In MCF, if an unstable eigenfunction given by an Euclidean translation dominates the rescaled MCF, then one can eliminate this term by shifting to a new base point. See Sesum \cite{sesum2008rate} and \cite[Theorem 4.15]{choi2022ancient}. 
In Ricci flow, this phenomenon means an unstable generic eigentensor of $L$ dominates the rescaled RDTF perturbation.
Based on our observation of the stability of $L$ and $\mathfrak L$, the dominant generic unstable eigentensor is the Lie derivative of an unstable eigenvector $X$ of $\mathfrak L$.
Thus, we can eliminate the dominant generic unstable eigentensor in the RDTF perturbation by eliminating $X$ in the {HMHF perturbation}. The HMHF perturbation equation is a PDE for vector fields, which is satisfied by the \textit{difference} between two HMHFs between a fixed Ricci flow and the shrinker (see Section \ref{sec: uniqueness}).

In fact, for a fixed ancient rescaled Ricci flow that converges exponentially, by solving the HMHF perturbation equation,
we construct $\textnormal{index}(\mathfrak L)$-parameter family of ancient HMHFs. 
For example, there are $(n+1)$-parameter family of ancient HMHFs between the static flow by $\mathbb S^n$ and itself, which corresponds to one non-trivial ancient HMHF between the $\mathbb S^n$ and itself modulo time-shiftings and isometries.

\subsection{Optimal convergence rate for Ricci flow singularities}
In this subsection, we study Ricci flows that develop a singularity at 
$t=0$ modelled on a compact shrinker. Our goal is to analyze the convergence rate of these Ricci flows to the singularity.
First, we show that for any essential eigenvalue $-\lambda<0$, there exists a Ricci flow that converges to the singularity at the rate $(-t)^{1+\lambda}$. 
\begin{theorem}\label{t:different_convergence}
    Let $(M,\bar g,f)$ be a compact shrinker. Let {$-\lambda<0$ be any negative essential eigenvalue of $L$}. Then there exists a Ricci flow $g_t$ developing a singularity modelled on $\bar g$ as $t\to0$ such that there are diffemorphisms $\chi_{t}$ such that
    \begin{equation*}
       C_k^{-1}(-t)^{1+\lambda}\le \|\chi_{t}^*g_t-(-t)\bar g\|_{C^{k }(\bar g)}\le C_k(-t)^{1+\lambda},
    \end{equation*}
    for each $k\in\mathbb N$ and some $C_k>0$.
    In particular, $\chi_t\equiv\id$ if $(M,\bar g)$ is positive Einstein.
\end{theorem}

In the rescaled flow, this means the convergence rate is $e^{-\lambda\tau}$. Since the absolute value of a negative essential eigenvalue can be arbitrarily large, the theorem implies that Ricci flow singularities can form at any prescribed exponential rate in the rescaled flow.

Then, a natural question is whether a flow can converge to the shrinker faster than $(-t)^{1+\lambda}$ for any $\lambda>0$. Kotschwar \cite{kotschwar2022maximal} showed that it is impossible unless it is self-similarly shrinking flow up to diffeomorphisms. We provide an alternative proof. See also \cite{sesum2008rate,strehlke2020unique,martin2023rate,DHH25rate} for analogous results in MCF and \cite{choi2024thom} for a very large class of elliptic and parabolic equations.

{Note that this result is related to the unique continuation in elliptic PDE theory}, which states that if a solution nearly vanishes in a region, it must vanish entirely. The unique continuation was also used to establish the uniqueness of MCF shrinking solitons by Wang \cite{Wang_JAMS} and also shrinking Ricci solitons by Kotschwar-Wang \cite{Kotschwar-Wang}.
\begin{theorem}[Kotschwar \cite{kotschwar2022maximal}]\label{thm:Carleman}
Let $(M,\bar g,f)$ be a compact shrinker, and $\alpha\in(0,1)$.
Let $g_t$, $t\in(-T,0)$ for some $T>0$,
    be a Ricci flow developing a singularity modelled on $\bar g$ as $t\to0$. If for any $\lambda>0$ there exists $C_{\lambda}>0$ such that
    \begin{equation*}
        \|(\xi_t^{-1})^*g_t-(-t)\bar g\|_{C^{2,\alpha}(\bar g)}\le C_{\lambda}(-t)^{1+\lambda}.
    \end{equation*}
Then there is a smooth family of diffeomorphisms $\chi_t$ such that $\chi_t^*g_t\equiv\bar g_t$. Moreover, if $\bar g$ is positive Einstein, then $g_t\equiv(-t)\bar g$.
\end{theorem}

Then, a natural question is whether the first {negative essential eigenvalue} determines the slowest possible convergence rate of the Ricci flow to the shrinker. Using the immortal HMHF gauge from Theorem \ref{t:uniqueness_of_tangent_flow}, we obtain the optimal convergence rate is indeed given by the eigenvalue of the first essential stable eigentensor for compact integrable shrinkers. Consequently, by Theorem \ref{t:different_convergence}, this provides an optimal lower bound for the convergence rate.
Note that the integrability cannot be removed, as the exponential convergence may fail at cylindrical singularities.
\begin{theorem}[{Rate of convergence at singularity}]\label{thm:gap}
Let $(M,\bar g,f)$ be a compact integrable shrinker.
Let $-\lambda<0$ be the largest {negative essential eigenvalue.}
Let $g_t$ be a Ricci flow developing a singularity modelled on $\bar g$ as $t\to0$, then there exist a smooth family of diffeomorphisms $\chi_t$ such that 
    \begin{equation}\label{eq:gap}
        \|(\chi_t^{-1})^*g_t-(-t)\bar g \|_{C^{k}(\bar g)}\le C_k(-t)^{1+\lambda},
    \end{equation}
for each $k\in\mathbb N$ and some $C_k>0$.
In particular, if $\bar g$ is positive Einstein, then $\chi_t=\id$ and
\begin{equation}\label{eq:gap_2}
        \|g_t-(-t)\bar g \|_{C^{k}(\bar g)}\le C_k(-t)^{1+\lambda}.
    \end{equation}
\end{theorem}

For the round sphere $\mathbb \mathbb S^n(\sqrt{2(n-1)})$, the {first negative essential eigenvalue} is $-\frac{2}{n-1}$, so our theorem implies that any spherical singularity of Ricci flow forms at least at the exponential rate of $\frac{2}{n-1}$ in the rescaled flow. Equivalently, for the unrescaled flow, we have

\begin{cor}
    Let $(M,g_t)$ be an n-dimensional Ricci flow which develops a singularity modelled on $\mathbb S^n$ as $t\to0$, then for each $k\in\mathbb N$ there is $C_k>0$ such that
\begin{equation*}
        \|g_t-(-t)\bar g \|_{C^{k}(\bar g)}\le C_k(-t)^{\frac{n+1}{n-1}}.
    \end{equation*}
\end{cor}

In MCF, the convergence rate to a singularity has been extensively studied. Sesum \cite{sesum2008rate} established optimal convergence rate estimates for MCF at spherical singularities by using the first essential eigenvalue so that she showed that the arrival time (the function whose graph is the level set flow of the MCF) has optimal regularity of $C^{2,\frac{2}{n}}$. Notably, the cylinder is not integrable, and MCF cylindrical singularities can converge slower than any exponential rate.
As a result, the arrival time is not necessarily $C^2$ as shown by Ilmanen \cite{ilmanen1992generalized}, but is still twice differentiable as shown by Colding and Minicozzi \cite{colding2016differentiability, colding2018regularity}.

\textit{Challenges of Theorem \ref{thm:RF_exp}}: 
We discuss the challenges of Theorem \ref{thm:RF_exp} for ancient flows. 

Since the standard PDE analysis fails for Ricci flow due to the lack of strong parabolicity, we want to work with RDTFs.
One approach to obtaining a global RDTF is to solve initial value problems for RDTFs at a sequence of times approaching infinity and then take a limit. However, even if the initial values tend to zero as the flow converges to the shrinker, it is unclear whether the RDTFs can be extended uniformly up to a fixed time \( T_0 < 0 \). 

In fact, the RDTFs can be dominated by unstable generic {modes}, leading to exponential growth forward in time.  
Without precise control over the rate at which the initial values decay to zero as the initial times tend to \( -\infty \), the RDTFs may quickly become large and cease to exist before reaching \( T_0 \). So constructing a global RDTF at this stage is not feasible.  

Instead, our strategy is to work with a sequence of short-time RDTFs. Starting at an arbitrary time, we evolve a RDTF for a short time, adjust the diffeomorphism, and then start the next RDTF, ensuring that unstable generic {modes} remain small throughout the process. {Then, these RDTFs increases together with unstable essential modes, which has exponential growth rate. Moreover, this process continues until time $T_0$ where the $\mu$-entropy gap from the shrinker reaches some $\eps_0>0$. }

\textit{Outline of Theorem \ref{thm:RF_exp} $($ancient flow$)$}:
For an ancient rescaled flow $g(\tau)$ asymptotic to the shrinker $\bar g$, choose a sequence of times $\tau_k\to-\infty$. At each $\tau_k$, we find the \textit{best} gauge by diffemorphism $\phi_k$, such that the perturbation $\phi_k^*g(\tau_k)-\bar g$ consist of only essential modes. So the unstable essential modes dominate
by the entropy monotonicity. 
We start a short-time HMHF from $\phi_k$, and a rescaled RDTF from  $\phi_k^*g(\tau_k)-\bar g$ for some time $T>0$.
At $\tau_k+T$, we eliminate any unstable generic modes by changing a diffeomorphism.
Then, we start a new rescaled RDTF from $\tau_k+T$.

Iterating this by induction, we obtain a sequence of rescaled RDTFs on $[\tau_k+iT,\tau_k+(i+1)T]$, until we reach $\tau_k+iT>T_0$.
Crucially, the unstable essential modes remain dominant at all times on $[\tau_k,T_0]$.
Furthermore, the entropy upper bound put on uniform upper bounds on the unstable essential modes at all times.
Therefore, the exponential growth of the unstable essential modes implies the backward exponential decay from $T_0$ all the way back to $\tau_k$ under the gauge by $\phi_k$. This implies the desired backward exponential decay from $T_0$ to $-\infty$ as $\tau_k\to-\infty$ and $\phi_k$ converges to a fixed gauge.

\textit{Outline of Theorem \ref{thm:RF_exp} $($immortal flow$)$}: For an immortal rescaled flow $g(\tau)$ asymptotic to the shrinker $\bar g$, 
we first fix some negative eigenvalue $\lambda<0$, and find a diffeomorphism $\phi_0$ at $\tau=0$ such that two initial conditions hold in the perturbation $\phi_0^*g(0)-\bar g$: 
\begin{enumerate}
    \item the unstable generic modes and neutral modes are sufficiently small;
    \item the modes with eigenvalues larger than $\lambda$ dominate the perturbation.
\end{enumerate}
To achieve this, we first find the \textit{best} diffeomorphism as in the ancient case, so that all generic and neutral modes vanish. Then we introduce a perturbation via a diffeomorphism corresponding to a stable generic eigentensor, ensuring that this mode dominates the perturbation, thereby satisfying the second condition.

Then, we start a short-time HMHF from $\phi_0$ for some time $T>0$, which gives a rescaled RDTF from  $\phi_0^*g(0)-\bar g$.
Condition (2) ensures that the perturbation does not decay too rapidly, so that the unstable generic modes and neutral modes (which can only increase at a bounded rate) remain relatively small on $[0,T]$. Consequently, we can apply the entropy estimate at 
$T$ and conclude that the {stable modes dominate} on $[0,T]$.
At $T$, we refine the perturbation by adjusting the diffeomorphism to eliminate unstable generic modes and modifying the background shrinker to eliminate neutral modes.
This guarantees that condition (1) remains valid at 
$T$, while condition (2) also holds at 
$T$ as a consequence of spectral analysis.

So we can iterate this by induction, and obtain a sequence of rescaled RDFT on $[iT,(i+1)T]$,
such that the two conditions hold at every $iT$.
The exponential decay of the dominating stable modes thus implies the same for the flow, finishing the proof.

\textit{Outline of Theorem \ref{t:uniqueness_of_tangent_flow}}: We observe that for a family of HMHFs between the flow $g_\tau$ and $\bar g$ by diffeomorphisms $\chi_\tau$ that are close to $\id$, the vector field $X_\tau:=\exp^{-1}_{\id}\chi_\tau$ satisfies the \textit{harmonic map heat flow perturbation equation} of the form
\[\partial_\tau X_\tau=(\Delta+\Ric) X_\tau+Q(X_\tau)+E(X_\tau)+G_\tau,\]
where $Q(X_\tau)$ denotes quadratic terms of $X_\tau$, $E(X_\tau)$ denotes products of $X_\tau$ and $g_\tau-\bar g$, and $G_\tau$ denotes the inhomogenous terms of $g_\tau-\bar g$. Thanks to Theorem \ref{thm:RF_exp}, we know $g_\tau-\bar g$ decays exponentially. So the existence of an ancient or immortal solution $X_\tau$ follows by a standard PDE fixed point argument.

\subsection{Acknowledgements}  
K.C.~is supported by the KIAS Individual Grant MG078902, an Asian Young Scientist Fellowship, and the National Research Foundation (NRF) grants RS-2023-00219980 and RS-2024-00345403 funded by the Korea government (MSIT).
Y.L.~is supported by NSF grant DMS-2203310. Part of this work was completed during Y.L.'s visit at UC Berkeley and SLmath, and she is grateful for the hospitality of Richard Bamler, Catherine Searle, and Guofang Wei.
The authors also thank Pak-Yeung Chan, Pei-Ken Hung, Klaus Kr\"oncke, and John Lott for comments.

\section{Commutators and eigentensors} \label{sec:commutator}
Let $(M,g,f)$ be a compact shrinker $\Ric+\nabla^2 f=\tfrac12 g$. Consider the two elliptic operators $L$ on $ \Gamma(S_2(T^*M))$ and $\mathfrak L$ on $ \Gamma(TM)$ given by
\begin{align*}
    L&=\Delta_f+2\Rm*,\\
    \mathfrak L&=\Delta_f+\tfrac12=\Delta+\Ric-[\nabla f,\cdot],
\end{align*}
which are self-adjoint with respect to the $L^2$-norm measured by $e^{-f}\,dg$.
Note that $L$ and $\mathfrak L$ are the linearization of the modified rescaled Ricci DeTurck flow Perturbation equation and the modified HMHF equation (see Appendix \ref{sec:RRF} \ref{sec:HMHF}), and $L$ is also equal to the operator of the second variation of the $\mu$-entropy restricted on $\Ker\,\Di_f$ (see Appendix \ref{subsec:stability}).
The main result in this section (Lemma \ref{t: unstable Lie derivative implies unstable vector_shrinker}) shows that any generic eigentensor of $L$ with non-zero eigenvalue $\lambda$, is the Lie derivative of an eigenvector of $\mathfrak L$ with the same eigenvalue $\lambda$.

Furthermore, for positive Einstein manifolds, we show the converse is also true (Theorem \ref{t: unstable Lie derivative implies unstable vector}). In particular, $\Spec(L)\setminus\{0\}=\Spec(\mathfrak L)\setminus\{0\}$, and there is a one-to-one correspondence between the eigenvectors.

First, we recall the notions of several differential operators.
The divergence operator is $\Di:\Gamma(S_2(T^*M))\to \Gamma(T^*M)$ defined as $$\Di \,h=C_1^1\nabla h.$$
The weighted divergence is defined as
\[\Di_f \,h=\Di \,h-h(\nabla f,\cdot).\]
The Lie operator $\lie:\Gamma(T^*M)\to \Gamma(S_2(T^*M))$ is defined as 
$$\Di^*W=-\tfrac12\mathcal{L}_{W^{\#}}g.$$ 
We also define $\Di_0:\Gamma(S_2(T^*M))\to \Gamma(T^*M)$ by
$$\Di_0(h)=\textnormal{div}(h)-\frac{1}{2}\nabla(\textnormal{tr}(h)).$$
Note $\Di,\Di^*,\Di_0$ are all first-order differential operators.
By identifying a vector field to its dual 1-form, these operators act between $\Gamma(S_2(T^*M))$ and $ \Gamma(TM)$.

\begin{lemma}\label{l: Delta Z = div div^* Z}
   Let $(M,g)$ be a Riemannian manifold. For any smooth 1-form $Z$, we have
    $$\Di_{0}(-2\Di^*)Z=\Di_{0}(\mathcal{L}_Zg)=(\Delta+\Ric) Z.$$
    In particular, for any Killing field $Y$, we have $(\Delta+\Ric)Y=0$.
\end{lemma}

\begin{proof}
    First, we have $$\mathcal{L}_Zg=\nabla_iZ_j+\nabla_jZ_i,$$
    and thus
    \begin{equation*}
        \begin{split}
            (\Di_0(\mathcal{L}_Zg))_k&=g^{ij}(\nabla_j\nabla_iZ_k+\nabla_j\nabla_kZ_i-\tfrac12\nabla_k\nabla_iZ_j-\tfrac12\nabla_k\nabla_jZ_i)\\
            &=g^{ij}(\nabla_j\nabla_iZ_k+\nabla_j\nabla_kZ_i-\nabla_k\nabla_jZ_i)\\
            &=\Delta Z_k+\Ric(Z)_k,
        \end{split}
    \end{equation*}
    where in the last equality we used the Ricci identity 
    $$g^{ij}(Z_{i,kj}-Z_{i,jk})=g^{ij}Z_lR^l_{kji}=Z_lR^l_k=\Ric(Z)_k.$$
\end{proof}

Recall that all compact shrinking solitons are necessarily \textit{gradient} solitons \cite[Prop 12.18]{KL}, which means there is a function $f$ on $M$ such that $X=\nabla f$.

\begin{lem}\label{lem:compact_Killing}
    Let $(M,g,f)$ be a compact shrinking soliton. Then $f\circ\phi=f$ for any isometry $\phi$. In particular, we have $[\nabla f,Y]=0$ for all Killing field $Y$.
\end{lem}

\begin{proof}
    By the Ricci soliton equation $\Ric+\nabla^2 f=\tfrac{1}{2}g$, we have $\nabla^2 f\circ\phi=\nabla^2 f$. Since $M$ is compact, $f$ has a maximal point $p\in M$. Composing $\phi$ with the soliton identity $R+|\nabla f|^2-f=C$ for some $C\in\R$, and note $0=|\nabla f|^2(p)\le |\nabla f|^2(\phi(p))$ and $-f(p)\le -f(\phi(p))$, it follows $0=|\nabla f|^2(p)=|\nabla f|^2(\phi(p))$ and $f(p)=f(\phi(p))$, and hence $f\circ \phi=f$.
\end{proof}

\begin{lem}\label{t: unstable Lie derivative implies unstable vector_shrinker}
   Let $(M,g,f)$ be a compact shrinking soliton. If $h\in \Ker(L-\lambda)\cap \Ima\,\Di^*$ for some $0\neq\lambda\in\R$, then there exists a vector field $X$ such that 
   \[X\in \Ker(\mathfrak L-\lambda)\quad\textnormal{and}\quad h=\mathcal L_Xg.\]
\end{lem}

\begin{proof}
   First, by \cite[Theorem 1.32]{colding2021singularities}, for any vector field, there is the commutator
\begin{equation*}
    L\circ\Di^*_f=\Di^*_f\circ\mathfrak L.
\end{equation*}
If $h\in \Ker(L-\lambda)\cap \Ima\,\Di^*$, assume $h=\mathcal L_Xg$. Then by the commutator we have
\[\lambda\mathcal L_X g=\lambda h=Lh=\mathcal L_{\mathfrak L X}g,\]
and hence $Y:=\mathfrak L X-\lambda X$ is a Killing field. Using Lemma \ref{lem:compact_Killing} and \ref{l: Delta Z = div div^* Z}, we have $[\nabla f,Y]=0$ and $(\Delta+\Ric)Y=0$, and thus $\mathfrak L Y=0$.
It follows that
\[\mathfrak L (X+\tfrac{1}{\lambda}Y)=\mathfrak L X=\lambda X+Y,\]
and thus $X+\tfrac{1}{\lambda}Y$ satisfies the assertions of the lemma.
\end{proof}

In the following, we assume $(M,g)$ is a positive Einstein manifold.

\begin{lem}\label{l: -L=2div*div}
    Let $(M,g)$ be positive Einstein with $\Ric=\frac12 g$. Then for all $h\in\Ima\,\Di^*$ that 
\begin{equation*}
    Lh=\Delta h +2\Rm*h=-2\Di^*\circ\Di_0 h.
\end{equation*}
\end{lem}

\begin{proof}
    Using the linearization of $-2\Ric$ and $\Ric=\frac12 g$ we obtain
\begin{equation}\label{e: decompose Ric}
\begin{split}
    D(-2\Ric)[h]&=\Delta h +2\Rm*h-2\Ric(h)+2\Di^*\circ\Di_{0}h,\\
    &=L h-h+2\Di^*\circ\Di_{0}h.
\end{split}
\end{equation}
Since $-2\Ric$ is invariant under diffeomorphisms, assuming $h=\LL_X g$ we have
\begin{equation*}
    D(-2\Ric)[h]=D(-2\Ric)[\mathcal{L}_Xg]=\mathcal{L}_X(-2\Ric)=\mathcal{L}_X(-g)=-h.
\end{equation*}
So by \eqref{e: decompose Ric} we obtain the assertion of the lemma. 
\end{proof}

\begin{cor}\label{l: commute}
 Let $(M,g)$ be positive Einstein, then we have the commutators
\begin{equation}\label{e: commute}
\begin{split}
    \Di_0\circ L&=\mathfrak L\circ\Di_0,\quad\quad\quad\textnormal{on }\Ima\,\Di^*,\\
    -2\Di^*\circ \mathfrak L&= L\circ(-2\Di^*),\quad\textnormal{on }\Gamma(TM).
\end{split}
\end{equation} 
\end{cor}

\begin{proof}
    These follow from Lemma \ref{l: Delta Z = div div^* Z} and \ref{l: -L=2div*div}.
\end{proof}

\begin{theorem}\label{t: unstable Lie derivative implies unstable vector}
   Let $(M,g)$ be positive Einstein, then for any $\lambda\in\R$, we have $\Di_0\big(\Ker(L-\lambda)\cap\Ima\,\Di^*\big)\subset \Ker(\mathfrak L-\lambda)$ and $-2\Di^*\big(\Ker(\mathfrak L-\lambda)\big)\subset \Ker(L-\lambda)$, and
   \begin{equation}\label{e: invertible}
    \begin{split}
        (-2\Di^*)\circ\Di_0&=\lambda\cdot\id \quad\textnormal{ on }\quad \Ker(L-\lambda)\cap\Ima\,\Di^*;\\
        \Di_0\circ(-2\Di^*)&=\lambda\cdot\id \quad\textnormal{ on }\quad \Ker(\mathfrak L-\lambda).
    \end{split}
\end{equation}
In particular, for any $\lambda\neq0$, we have $\lambda\in \Spec(-L)$ if and only if $\lambda\in \Spec(-\mathfrak L)$. In this case, $\Di_0$ and $-2\Di^*$ are isomorphisms between $\Ker(L-\lambda)\cap \Ima\,\Di^*$ and  $\Ker(\mathfrak L-\lambda)$.
\end{theorem}

\begin{proof}
First, \eqref{e: invertible} holds by Lemma \ref{l: -L=2div*div} and \ref{l: Delta Z = div div^* Z}. Next, for any $h\in \Ker(L-\lambda)\cap\Ima\,\Di^*$ and $X\in \Ker(\mathfrak L-\lambda)$, using \eqref{e: invertible} we have $L(\Di_0h)=\lambda\, \Di_0h$ and $L(-2\Di^*X)=\lambda\cdot(-2\Di^*X)$. This shows all assertions.
\end{proof}

\begin{proof}[Proof of Theorem \ref{thm:generic_index}]
The theorem follows directly from Theorem \ref{t: unstable Lie derivative implies unstable vector}. In particular, the isomorphisms between $E_{\lambda,L}\cap\Ima\,\Di^*$ and $E_{\lambda,\mathfrak L}$ are given by $\Di_0$ and $-2\Di^*$.
\end{proof}

\section{Existence of ancient and immortal Ricci flows}\label{sec: existence}
In this section, we construct ancient and immortal Ricci flows asymptotic to any given compact shrinker as stated in Theorem \ref{t:existence_theorem_intro} and \ref{t:different_convergence} with prescribed asymptotic behaviors.  

First, we define a Ricci flow asymptotic to a compact shrinker.

\begin{defn}[Ricci flows asymptotic to a compact shrinker]\label{def:asymptotic}
We say an ancient (resp. immortal) Ricci flow $(M,g_t)$, $t\in(-\infty,-T]$, is \textit{asymptotic to} a compact shrinker $(M,\bar g)$,
if one of the following equivalent conditions holds \cite{CMZ1}:
\begin{enumerate}
    \item For some $\lambda_i\to0$, the flows $\lambda_i g_{\lambda_i^{-1}t}$ converge to the metric flow of $\bar g$ under the weak notion of $\mathbb F$-limits introduced by Bamler \cite{bamler2023compactness};
    \item There
are $t_i\to-\infty$, $Q_i>0$, and diffeomorphisms $\phi_{i}$ such that the metrics $Q_i\phi^*_{i}g_{t_i}$ converge smoothly to $\bar g$.
\item Let $\tilde g_\tau$ be the rescaled Ricci flow of $g_t$, then there exist diffeomorphisms $\phi_\tau$ such that $\phi^*_\tau\tilde g_\tau\to\bar g$ smoothly as $\tau\to-\infty$.
\end{enumerate}
Similarly, we can define a Ricci flow $(M,g_t)$, $t\in[-T,0)$, which is \textit{asymptotic to} $(M,\bar g)$ at $t=0$. We also say $(M,\bar g)$ is the singularity model of the Ricci flow at $t=0$.
\end{defn}

For a given compact shrinker $(M,\bar g,f)$, the stability operator $L=\Delta_{f}+2\Rm*:\textnormal{Sym}^2(M)\to \textnormal{Sym}^2(M)$ is symmetric and uniformly elliptic, and the eigenvalues are
\begin{equation}\label{eq:eigenvalues of L}
\lambda_1\ge\lambda_2\ge\cdots\ge\lambda_I>0=\lambda_{I+1}=\cdots=\lambda_{I+K}>\lambda_{I+K+1}\ge\cdots,
\end{equation}
repeated according to their multiplicity; note that $\lim_{j\to\infty}\lambda_j=-\infty$.
Also, the eigentensors with the eigenvalues $\lambda_1,...,\lambda_I$ are called \textit{unstable}, $\lambda_{I+1},...,\lambda_{I+K}$ are called \textit{neutral}, $\lambda_{I+K+1},...$ are called \textit{stable}.
We also fix once and for all an $L^2_f$-orthonormal sequence of corresponding smooth eigentensors $\{h_j\}_{j=1}^{\infty}$. Note $L\Ric=\Ric$, $\Ric$ is an unstable eigentensor of eigenvalue $1$.
Let  
\begin{equation}\label{eq:eigenvalues of L min}
    \lambda_+=\lambda_I>0\quad\textnormal{and}\quad \lambda_-=-\lambda_{I+K+1}>0.
\end{equation}

By the $L^2_f$-decomposition of the symmetric 2-tensors $\textnormal{Sym}^2(M)=\Ima\,\Di^*_f\perp \Ker\,\Di_f$, the subspaces $\Ima\,\Di^*_f, \Ker\,\Di_{f}$ are $L$-invariant.
We call eigentensors in 
$\Ima\,\Di^*_f$ (Lie derivatives) \textit{generic}, and eigentensors in $\Ker\,\Di_f$ \textit{essential}.
We write $V^{+},V^-,V^0,V_{\Lie},V_{\ess}$ to be the subspace spanned by all unstable, stable, neutral, generic, essential eigentensors, respectively; and let $\PP^{+},\PP^-,\PP^0,\PP_{\Lie},\PP_{\ess}$ be the projections.
Moreover, we denote $V_{\ess}^+=V_{\ess}\cap V^+$ and $P_{\Lie}=P_{\Lie}^+ +P_{\Lie}^0+P_{\Lie}^-$ and similarly for other subspace intersections.

For easy of notation, we denote by $\{ \bar h_1,\cdots, \bar h_{I_{\ess}}\}\subset \textnormal{Sym}^2(M)$ the essential unstable eigentensors which span the space $V_{\ess}^+$ and satisfy
\begin{align}
   & \langle  \bar  h_i,\bar  h_j\rangle_{L^2_f}=\delta_{ij}, && L\bar h_i= \bar\lambda_i \bar h_i, && \bar\lambda_1 \geq \bar\lambda_2\geq \cdots \geq \bar\lambda_{I_{\ess}}>0,
\end{align}
where ${I_{\ess}}=\dim(V_{\ess}^+)$ is the \textit{essential index} of $L$.
Now, we let $m_1$ denote the multiplicity of $\bar\lambda_1$, namely $\bar\lambda_1=\cdots=\bar\lambda_{m_1}>\bar\lambda_{m_1+1}$.
Then, for $k\in \mathbb{N}$, we denote by $m_{k+1}$ the multiplicity of $\bar\lambda_{m_k+1}$. Let $I_{\ess}=m_1+\cdots+m_{I'}$ and $\lambda_i'=\bar  \lambda_{\bar m_i}$ for $i \leq I'$, where $\bar m_i=m_1+\cdots+m_i$. Then there are $I'$ distinct positive eigenvalues $\lambda'_1,\cdots,\lambda'_{I'}$.

Given $\mathbf{a}=(a_1,\cdots,a_{I_{\ess}}) \in \mathbb{R}^{I_{\ess}}$, we denote $\mathbf{a}=(\vec{a}_1,\cdots,\vec{a}_{I'})\in \mathbb{R}^{m_1}\times \cdots \times \mathbb{R}^{m_{I'}}$ with abuse of notation. Similarly, we denote $\vec{ h}_i=( h_{\bar m_{i-1}+1},\cdots, h_{\bar m_i})$.  Moreover, we define the projection $\pi_i:\mathbb{R}^{I_{\ess}}\to \mathbb{R}^{I_{\ess}}$ by 
\begin{equation}
  \pi_i((\vec{a}_1,\cdots,\vec{a}_{I'}))=(0,\cdots,0,\vec{a}_i,\cdots,\vec{a}_{I'}).  
\end{equation}

\begin{theorem}\label{t: existence theorem}
Let $(M,\bar g,f)$ be a compact shrinking soliton. Then, there exist a map $S$ and some constant $\delta>0$ such that for each $\mathbf{a}\in \mathbb R^{I_{\ess}}$, $S(\mathbf{a})$ is an ancient rescaled modified RDTF with background $(M,\bar g,f)$ and
\begin{enumerate}
    \item $S(\mathbf{0})$ is the shrinking soliton $\bar{g}$.
    \item Suppose $\mathbf{a},\mathbf{b}\in \mathbb{R}^{I_{\ess}}$ satisfy $\vec{a}_j\neq \vec{b}_j$ for some $j\leq I'$ and $\vec{a}_i = \vec{b}_i$ for all $i > j$. Then,
    \begin{equation}\label{eq:construction strata}
 e^{\lambda_j'\tau}   [S(\mathbf{a})-S(\mathbf{b})]= (\vec{a}_j-\vec{b}_j)\cdot\vec{ h}_j+ O(e^{\delta \tau}),
    \end{equation}
  as $\tau\to -\infty$.
\end{enumerate}
\end{theorem}

\begin{proof}
Note the rescaled modified RDTF Perturbation equation \eqref{eq:most important eq} satisfies \eqref{eq:both_RDTF_HMHF}.
So for $\mathbf{a}=(\vec{a}_1,\cdots,\vec{a}_{I'})$, we first apply Theorem \ref{l: contraction mapping RDTF} with $X^0_\tau=S(\mathbf{0})-\bar g\equiv 0$, $\lambda=\lambda_{I'}$, and $aY_i=\vec{a}_{I'}\cdot \vec{ h}_{I'}$, then the solution to \eqref{eq:both_RDTF_HMHF} gives us $S(\pi_{I'}(\mathbf{a}))-\bar g$. Next, we apply Theorem \ref{l: contraction mapping RDTF} again with $X^0_\tau=S(\pi_{I'}(\mathbf{a}))-\bar g$, $\lambda=\lambda_{I'-1}$, and $aY_i=\vec{a}_{I'-1}\cdot \vec{ h}_{I'-1}$, then the solution to \eqref{eq:both_RDTF_HMHF} gives us $S(\pi_{I'-1}(\mathbf{a}))-\bar g$. By repeating this process $I'$ times, we obtain $S(\pi_1(\mathbf{a})):=S(\mathbf{a})$, and \eqref{eq:construction strata} follows from \eqref{eq:contraction strata on base}.
\end{proof}

\begin{remark}
We can see $S(0,\cdots,0,\vec{a}_{I'})(\cdot,t)=S(0,\cdots,0,e^{\lambda_{I'}s}\vec{a}_{I'})(\cdot,t-s)$ from the proof of Theorem \ref{l: contraction mapping RDTF}. Then, we can inductively show 
\begin{equation}
    S(\vec{a}_1,\cdots,\vec{a}_{I'})(\cdot,t)=S(e^{\lambda_1's}\vec{a}_1,\cdots,e^{\lambda_{I'}s}\vec{a}_{I'})(\cdot,t-s).
\end{equation}
\end{remark}

Now we prove Theorem \ref{t:existence_theorem_intro}:

\begin{proof}[Proof of Theorem \ref{t:existence_theorem_intro}]
   Let $\xi_\tau$ be the flow of diffeomorphisms such that $\partial_\tau\xi_\tau=\nabla f\circ\xi_\tau$ with $\xi_0=\id$. For each $S(\mathbf a)$, the flow $\widetilde g_\tau(\mathbf a):=\xi_{\tau}^*(\bar g+ S(\mathbf a))$ is a rescaled RDTF with background metric being the rescaled Ricci flow $\bar g_\tau$ of $\bar g$ (see Appendix \ref{sec:HMHF}). Then by solving the ODE $\partial_\tau\chi_\tau=\Delta_{\widetilde g_\tau(\mathbf a),\bar g_\tau}\id\circ\chi_\tau$ with $\chi_0=\id$ on $\tau\in(-\infty,0]$, we obtain a family of smooth diffeomorphisms $\chi_{\tau}$ such that $\chi_\tau^*(\xi_{\tau}^*(\bar g+ S(\mathbf a)))$ is a rescaled Ricci flow. 
In particular, for a positive Einstein manifold, we have $\xi_\tau=\id$, and we can choose $\chi_0$ so that $\chi_\tau\to\id$ exponentially as $\tau\to-\infty$, so the rescaled Ricci flow 
$\chi_\tau^*(\bar g+ S(\mathbf a))\to\bar g$ exponentially.
So this proves Theorem \ref{t:existence_theorem_intro} by Theorem \ref{t: existence theorem}. 
\end{proof}

Now, let us establish its analogue for immortal flows. To this end, we consider the notion of $\vec{ h}_j,\vec{a}_j$ even for neutral and stable eigentensors. We recall the \textit{space of eventually zero sequences}
\[c_{00} = \{ \mathbf a \in \mathbb{R}^{\infty} \mid \text{there is } N\in\mathbb N \text{ such that } \vec{a}_i = 0 \text{ for all } i \geq N \}.
\]
We let $k'=1$ if $V_{\ess}^0\neq 0$ and let $k'=0$ otherwise. Then, $\lambda_j'>0$ for $j> I'+k'$. 
Let 
\[\widehat c_{00} = \{ \mathbf a \in c_{00} \mid \text{{ $\vec{a}_i=0$ for all $i\leq I'+k'$}}\}.
\]

\begin{theorem}\label{t: existence theorem forward}
Let $(M,\bar g,f)$ be a compact shrinking soliton. Then there exist a map $S$ and a sequence $\delta_j>0$ such that for any {$\mathbf a\in \widehat c_{00}$}, each $S(\mathbf{a})$ is an immortal rescaled modified RDTF with background $(M,\bar g,f)$ with the following significance:
\begin{enumerate}
    \item $S(\mathbf{0})$ is the shrinking soliton $\bar{g}$.
    \item If $\mathbf{a},\mathbf{b}\in \widehat c_{00}$ satisfy $\vec{a}_j\neq \vec{b}_j$ for some $j> I'+k'$, and $\vec{a}_i = \vec{b}_i$ for all $i < j$, then
    \begin{equation}
 e^{\lambda_j'\tau}   [S(\mathbf{a})-S(\mathbf{b})]= (\vec{a}_j-\vec{b}_j)\cdot\vec{ h}_j+ O(e^{-\delta_j \tau}),
    \end{equation}
  as $\tau\to \infty$.
\end{enumerate}
\end{theorem}

\begin{proof}
The proof is the same to Theorem \ref{t: existence theorem}, using the forward case of Theorem \ref{l: contraction mapping RDTF}.
\end{proof}

\begin{proof}[Proof of Theorem \ref{t:different_convergence}]
   The rescaled immortal flows can be obtained from Theorem \ref{t: existence theorem forward}, in the same way as the rescaled ancient flows in
Theorem \ref{t:existence_theorem_intro}. 
\end{proof}

\begin{defn}[Fast convergence of harmonic map heat flow]
  Let $g_\tau$ be a rescaled modified ancient (resp. immortal) Ricci flow asymptotic to $\bar g$. We say an ancient (resp. immortal) modified HMHF $\chi_\tau$ between $g_\tau$ and $\bar{g}$ \textit{converges exponentially to $\id$} if 
there exist $\delta>0$ and $\tau_0\in(-\infty,0]$ (resp. $[0,\infty)$)
such that for all $\tau\le\tau_0$ (resp. $\tau\ge\tau_0$), the vector field $X_\tau:=\exp_{\id}^{-1}(\chi_\tau)$ satisfies
\begin{equation*}
    \|X_\tau\|_{C^k}\le  C_ke^{-\delta|\tau-\tau_0|},
\end{equation*} 
for each $k\in\mathbb N$ with some $C_k>0$, and the modified RDTFP $h_\tau:=(\chi^{-1}_\tau)^{*}g_\tau-{\bar{g}}$ satisfies
\begin{equation*}
    \|h_\tau\|_{C^k}\le C_ke^{-\delta|\tau-\tau_0|}.
\end{equation*}
\end{defn}

For any rescaled modified Ricci flow converging exponentially to the shrinker,
the following theorem shows the existence of modified HMHFs between the flow and the shrinker, which converges exponentially to $\id$.

\begin{theorem}
\label{t:existence_theorem_HMHF}
    Let $(M,\bar g,f)$ be a compact shrinker, and $(M,g_\tau)$ be a rescaled modified ancient $($resp. immortal$)$ Ricci flow converging exponentially to $(M,\bar g)$ as $\tau\to-\infty$. Then there exists an ancient $($resp. immortal$)$ modified HMHF between $g_\tau$ and $\bar g$, which converge exponentially to $\id$ as $\tau\to\mp\infty$.

    In particular, if $g_\tau$ is ancient, then there is an $\textnormal{index}(\mathfrak L)$-parameter family of pairwise different such ancient modified HMHFs, where $\mathfrak L=\Delta_f+\tfrac12$.
\end{theorem}

\begin{proof}
Since $g_\tau$ satisfies the assumptions in Lemma \ref{lem:existence_modified_HMHF}, applying which we obtain a solution $X_\tau=\exp_{\id}^{-1}\chi_\tau$ to the modified HMHF Perturbation at $\id$ equation \eqref{eq:modified_HMHF_perturbation}, and $\chi_\tau$ is a modified HMHF between $g_\tau$ and $\bar g$.
If $g_\tau$ is ancient, then by Theorem \ref{l: contraction mapping RDTF}, we can obtain $\textnormal{index}(\mathfrak L)$-parameter family of solutions to the modified HMHF Perturbation equation at $\chi_\tau$ \eqref{eq:HMHF perturbation sec 7_2}. Any such solution $Y_\tau=\exp_{\chi_\tau}^{-1}\mathcal Y_\tau$ gives a modified HMHF $\mathcal Y_\tau$ between $g_\tau$ and $\bar g$. 
\end{proof}

\section{A generalized slice theorem}
\label{sec:slice}

The main result in this section is Theorem \ref{t: slice theorem}. For any fixed compact integrable shrinker $\bar g$, and any nearby metric $g$, the theorem finds a nearby shrinker $\bar g'$ such that $g$ is divergence-free, and has zero neutral projections with respect to $\bar g'$.
If $\bar g$ has no essential neutral mode, then this follows from Ebin's slice theorem.
So Theorem \ref{t: slice theorem} is a generalized slice theorem for compact integrable shrinkers.

We first prove two lemmas that hold for all Riemannian manifolds.
For any small vector field $X$, the following lemma gives a Taylor expansion of the diffeomorphism by the exponential map of $X$, up to the optimal regularity.  

\begin{lemma}\label{l: key lemma for regularity}
    Let $(M,g)$ be a compact Riemannian manifold. Then for any $m\in\mathbb N$, $\alpha\in[0,1)$, there exist $\delta(m,\alpha,\|(M,g)\|_{C^{m+2,\alpha}}),C(m,\alpha,\|(M,g)\|_{C^{m+2,\alpha}})>0$ such that the following holds: For any vector field $X$, we define a smooth map $\phi_X$ by $\phi_X(p):=\exp_p(X(p))$ for all $p\in M$. Suppose $|X|,|\nabla X|\le\delta$, then $\phi_X$ is a diffeomorphism, and under any normal coordinates $\{x^i\}_{i=1}^n$, we have
\begin{align}
    \|\phi^i_X-x^i\|_{C^{m,\alpha}}&\le C \|X\|_{C^{m,\alpha}}\label{eq:first}\\
    \|\phi^i_X-x^i-X^i\|_{C^{m,\alpha}}&\le C \|X\|_{C^{m,\alpha}}^2\label{eq:second}\\
    \|\phi_X^*g-g-\LL_X g\|_{C^{m-1,\alpha}}&\le C\|X\|_{C^{m,\alpha}}^2.\label{eq:third}
\end{align}
Moreover, for any smooth function $F:M\to M$, we have
\begin{align}
    \|F\circ\phi_X-F-X(F)\|_{C^{m,\alpha}}&\le C\|F\|_{C^{m+2,\alpha}}\|X\|_{C^{m,\alpha}}^2.\label{eq:F function}
\end{align}
\end{lemma}

\begin{proof}
By compactness of the manifold there exists $\delta>0$ such that $\textnormal{inj}_p\ge10\delta$ for any $p\in M$.
Fix some $p\in M$ and
let $x=(x^1,...,x^n)\in U\subset M$ be a normal coordinates centered at $p$ in $B_g(p,10\delta)$. By abuse of notation, we identify a vector field $v=v^i\partial_i$ with its components 
$v^i$ under the coordinates. So the components of the exponential map  $\exp^i(x,v)$ are well-defined maps on $B_g(p,\delta)\times B_{\mathbb R^n}(\Vec 0,\delta)$ and maps into $ B_g(p,10\delta)$.
Suppose $|X|\le\delta$, then $\phi_X$ maps $B_g(p,\delta)$ into $B_g(p,2\delta)$ and thus the components $\phi_X^i(x)=\exp^i(x,X(x))$ are well-defined.
Then by the fundamental theorem of calculus, we have
\begin{equation}\label{eq:phi-x}
    \phi_X^i(x)-x^i=\int_0^1 \frac{\partial \exp^{i}}{\partial v^j}(x,sX(x))X^j(x)\,ds,
\end{equation}
and also
\begin{equation}\label{eq:phi_X diffeo}
    \phi_X^i(x)-x^i- X^i(x)=\int_0^1 \frac{\partial^2 \exp^{i}}{\partial v^j \partial v^k}(x,sX(x))X^j(x)X^k(x)\,ds.
\end{equation}
Note that for any $k\in\mathbb N$, there are constants $C_k>0$ uniform for all $p\in M$, such that $\|\exp^i\|_{C^k}\le C_k$ on $ B_g(p,\delta)\times B_{\mathbb R^n}(\Vec 0,\delta)$.
So we obtain the assertions \eqref{eq:first} and \eqref{eq:second}. 

Next, we show that $\phi_X$ is a diffeomorphism.
First, it is easy to see by \eqref{eq:phi-x} that $\phi_X$ has non-degenerate Jacobi at all points. It remains to show it is injective. Suppose there are $p_1,p_2\in M$ such that $\phi_X(p_1)=\phi_X(p_2)$.
Since $|X|\le\delta$, we have $\phi_X(p_j)\in B_g(p_j,2\delta)$ for $j=1,2$, and thus
$p_2\in B_g(p_1,4\delta)$ and $p_2$ is in the normal coordinate $\{U,x=(x^1,...,x^n)\}$ at $p_1$. Denote $x(p_j)=x_j$ and $x^i(p_j)=x^i_j$, $j=1,2$, then using \eqref{eq:phi_X diffeo} and $\phi_X^i(x_1)=\phi_X^i(x_2)$ we obtain
\begin{align*}
    &x_2^i-x_1^i+X^i(x_2)-X^i(x_1)\\
    =&\int_0^1 \left(\frac{\partial^2 \exp^{i}}{\partial v^j \partial v^k}(x_1,sX(x_1))X^j(x_1)X^k(x_1)-\frac{\partial^2 \exp^{i}}{\partial v^j \partial v^k}(x_2,sX(x_2))X^j(x_2)X^k(x_2)\right)\,ds.
\end{align*}
Since $|X|,|\nabla X|\le\delta$, we can see that the left hand side is $x_2^i-x_1^i+\delta \cdot O(|x_2-x_1|)$, and the right hand side is $\delta^2 \cdot O(|x_2-x_1|)$. So the equality of them for all $i$ when $\delta$ is sufficiently small implies $x_1=x_2$ and hence $p_1=p_2$. So $\phi_X$ is a diffeomorphism.

Next, we show the assertion \eqref{eq:F function} holds: First, note that
   \begin{equation}\label{eq:F_general}
   \begin{split}
       F\circ\phi_X
       &=F+ X(F)+\int_0^1\nabla^2 F(X_s,X_s)(\exp_x(sX(x)))\,ds
\end{split}
   \end{equation}
   where $X_s(\exp_x(sX(x)))=\exp_{x*sX(x)}(X(x))$. Since under local coordinates we have
\[\big\|X^i_s(\exp_x(sX(x)))\big\|_{C^{m,\alpha}}=\bigg\|\frac{\partial \exp^{i}}{\partial v^j}(x,sX(x))X^j(x)\bigg\|_{C^{m,\alpha}}\le C\|X^j\|_{C^{m,\alpha}},\] 
for some $C>0$,
the assertion \eqref{eq:F function} follows immediately from this and \eqref{eq:F_general}.

It remains to verify \eqref{eq:third}. 
Note that \eqref{eq:second} implies $\left\|\frac{\partial \phi_X^k}{\partial x^i}-\frac{\partial X^k}{\partial x^i}-\frac{\partial x^k}{\partial x^i}\right\|_{C^{m-1,\alpha}}\le C\|X\|_{C^{m,\alpha}}^2$,
and by choosing $F=g_{kh}$ in \eqref{eq:F function} we have $\left\|g_{kh}\circ\phi_X-g_{kh}-\frac{\partial g_{kh}}{\partial x^i}X^i\right\|_{C^{m,\alpha}}\le C\|X\|_{C^{m,\alpha}}^2$.
Thus, seeing the following local coordinate expressions we obtain \eqref{eq:third}:
\begin{align*}
(\phi_X^*g-g)_{ij}&=\frac{\partial \phi_X^k}{\partial x^i}\frac{\partial \phi_X^h}{\partial x^j}g_{kh}\circ\phi_X-g_{ij},\\    (\LL_Xg)_{ij}&=\frac{\partial X^p}{\partial x^i}g_{pj}+\frac{\partial X^p}{\partial x^j}g_{pi}-X^p\frac{\partial g_{ij}}{\partial x^p}.
\end{align*}
\end{proof}

Recall that for a compact Riemannian manifold $(M,g)$ and a smooth function $f$, we have the $L^2_f$-orthogonal decomposition 
$L^2_f(\textnormal{Sym}^2(M))=\Ker\,\Di_{f}\perp \Ima\,\Di^*_f$. 
Let $\PP_{\Lie}$ be the projection onto $\Ima\,\Di^*_f$, and denote by $S^2_{k,\alpha}(M)$ the space of all $C^{k,\alpha}$-smooth symmetric 2-tensors. The following lemma shows $\|\PP_{\Lie}h\|_{C^{k,\alpha}}$ is controlled by $\|h\|_{C^{k,\alpha}}$.

\begin{lem}
\label{l: standard elliptic estimates} 
Let $1\le k\in\mathbb N$, $\alpha\in(0,1)$.
Let $(M,g,f)$ be a compact Riemannian metric with a smooth function $f$.
Then there exists $C(k,\alpha,\|(M,g,f)\|_{C^{k+2,\alpha}})>0$
such that
    \begin{align*}
        \|\mathcal{P}_{\Lie}h\|_{C^{k,\alpha}}\le C\|\Di_{f}h\|_{C^{k-1,\alpha}}\le C\|h\|_{C^{k,\alpha}}.
    \end{align*}
In particular, we have 
\[S^2_{k,\alpha}(M)=\big(\Ker\,\Di_{f}\cap S^2_{k,\alpha}(M)\big)\oplus \big(\Ima\,\Di^*_f\cap S^2_{k,\alpha}(M)\big),\]
and the map $\Phi:h\longmapsto (h-\mathcal{P}_{\Lie}h,\mathcal{P}_{\Lie}h)$ is continuous.
\end{lem}

\begin{proof}
Assume $\mathcal{P}_{\Lie}h=\LL_X g=-2\,\Di^*_fX$ for some vector field $X$,
then we have
\[\Di_{f}\Di^*_fX=-\tfrac12\Di_{f}\LL_X g=-\tfrac12\Di_{f} h.\]
Since the space of Killing fields $\Ker\,\Di^*_f$ is finite-dimensional, we can assume $\langle X,Y\rangle_{L^2_f}=0$ for all vector field $Y\in \Ker\,\Di^*_f$.
Since $\Di_{f}$ is adjoint to $\Di^*_f$ with respect to $\langle\cdot,\cdot\rangle_{L^2_f}$, we have for any $Y\in \Di_{f}\Di^*_f$ that $0=\langle \Di_{f}\Di^*_fY,Y\rangle_{L^2_f}=\|\Di^*_fY\|^2_{L^2_f}$. So we have $\Ker\, (\Di_{f}\Di^*_f)=\Ker\,\Di^*_f$, and thus 
$\langle X,Y\rangle_{L^2_f}=0$ for all $Y\in \Ker\, (\Di_{f}\Di^*_f)$.
Since $\Di_{f}\Di^*_f$ is an elliptic operator (see for example \cite{eells1964harmonic}), and the $C^{k-1,\alpha}$-norm of the coefficients are bounded by $C$, it follows by the parabolic Schauder estimate \cite{krylov} and the above equation that
$$\|\PP_{\Lie}h\|_{C^{k,\alpha}}=\|\LL_Xg\|_{C^{k,\alpha}}\le C\|X\|_{C^{k+1,\alpha}}\le C\|\Di_{f} h\|_{C^{k-1,\alpha}}\le C\|h\|_{C^{k,\alpha}},$$
which implies the first assertion. The assertion of $S^2_{k,\alpha}(M)$ follows immediately.
\end{proof}

Next, we discuss the local moduli space and the integrability of a compact shrinker.
\begin{defn}[Integrability of shrinkers]\label{def:integrable}
    We say a compact shrinker $(M,\bar g,f)$ is integrable, if for any essential neutral mode $h$,
    there exists a family of shrinkers $\{g_s\}_{s\in[-\epsilon,\epsilon]}$ with $g_0=\bar g$ such that
    \[\lim_{s\to0}\Big\|\frac{g_s-\bar g}{s}-h\Big\|_{C^k}=0,\]
    for each $k\in \mathbb N$.
\end{defn}

Using
a modified version of Ebin’s slice theorem, Podest\`a and Spiro \cite{podesta2015moduli} constructed a $f$-twisted slice $\mathcal S_f$, which is a smooth manifold with tangent space $\Ker\,\Di_{\bar g,f}$, such that every metric near $\bar g$ is isometric to a unique metric in $\mathcal S_f$. 
Let 
$$\mathcal{S}ol_f=\{\widetilde g\in\mathcal S_f: \Ric_{\widetilde g}+\nabla^2 f_{\widetilde g}=\tfrac12 \widetilde g\}$$ 
be the set of all shrinking soliton metrics in $\mathcal S_f$.
If the shrinker $(M,\bar g,f)$ is integrable, $\mathcal{S}ol_f\subset \mathcal S_f$ is a finite-dimensional submanifold \cite[Theorem 3.4]{podesta2015moduli}, and the tangent space of $\mathcal{S}ol_f$ is $\Ker\, L\,\cap\,\Ker\,\Di_{\bar g,f}$, which consists of all essential neutral modes (equivalently infinitesimal solitonic deformations as in \cite{kroncke2016rigidity}).

In the proof of Theorem \ref{t: slice theorem}, we need the $C^{k,\alpha,*}_{\bar{g}}$-norm defined below. In Lemma \ref{lem:norms are comparable}-\ref{l: changing shrinker gauge}, we discuss several properties of this norm.

\begin{defn}
  Let $(M,\bar g,f)$ be a compact integrable shrinker.  For any $h\in L^2_{f}(\textnormal{Sym}^2(M))$, let
\[\|h\|_{C^{k,\alpha,*}_{\bar{g}}}:=\|\mathcal{P}_{\Lie}h\|_{C^{k,\alpha}_{\bar{g}}}+\|\mathcal{P}_{\ess}^0h\|_{C^{k,\alpha}_{\bar{g}}}+\|\mathcal{P}_{\textnormal{rest}}h\|_{C^{k,\alpha}_{\bar{g}}},\]
with respect to the $L^2_{f}$-orthogonal decomposition $L^2_{f}(\textnormal{Sym}^2(M))=V_{\Lie}\perp V^0_{\ess}\perp V_{\textnormal{rest}}$. 
\end{defn}

It is clear the $C^{k,\alpha,*}_{\bar g}$-norm is positive definite and satisfies the triangle inequality. The following lemma implies that $C^{k,\alpha,*}_{\bar g}$-norm is comparable to the standard $C^{k,\alpha}_{\bar g}$-norm:
\begin{lem}\label{lem:norms are comparable}
Let $1\le k\in\mathbb N,\alpha\in(0,1)$, and $(M,\bar g,f)$ be a compact integrable shrinker. There is a constant $C(k,\alpha,|\Rm_{\bar g}|)>0$ such that
\begin{align}\label{eq:norms are comparable}
    \|h\|_{C^{k,\alpha}_{\bar{g}}}\le\|h\|_{C^{k,\alpha,*}_{\bar{g}}}\le C\|h\|_{C^{k,\alpha}_{\bar{g}}}.
\end{align}
\end{lem}

\begin{proof}
    The first inequality is clear by definition. Since $V^0_{\ess}$ is finite-dimensional, we have \[\|P_{\ess}^0h\|_{C^{k,\alpha}_{\bar{g}}}\le C\|P_{\ess}^0h\|_{L^2_{f}}\le C\|h\|_{L^2_{f}} \le C\|h\|_{C^{k,\alpha}_{\bar{g}}},\]
    and by Lemma \ref{l: standard elliptic estimates} we have  $\|\mathcal{P}_{\Lie}h\|_{C^{k,\alpha}_{\bar{g}}}\le C\|h\|_{C^{k,\alpha}_{\bar{g}}}$. So the second inequality follows.
\end{proof}

The following lemma shows that the $C^{k,\alpha,*}$-norm with respect to a sequence of shrinkers converge if the shrinkers converge.

\begin{lem}\label{c:converge}
Let $M$ be a compact manifold.
Suppose there is a sequence of shrinkers $(M,g_i,f_i)\to(M,g_{\infty},f_{\infty})$ as $i\to\infty$ in the $C^{k,\alpha}_{g_\infty}$-sense, then for any $h\in \textnormal{Sym}^2(M)$ we have 
$$\lim_{i\to\infty}\|\PP_{\Lie,g_i}h- \PP_{\Lie,g_{\infty}}h\|_{C^{k,\alpha}_{g_\infty}}=0,$$ 
Moreover, if $\dim(V^0_{\ess,g_i})=\dim(V^0_{\ess,g_\infty})$, then
\[\lim_{i\to\infty}\|\PP^0_{\ess,g_i}h-\PP^0_{\ess,g_{\infty}}h\|_{C^{k,\alpha}_{g_\infty}}=0,\]
and also $\|h\|_{C^{k,\alpha,*}_{g_i}}\to \|h\|_{C^{k,\alpha,*}_{g_\infty}}$ as $i\to\infty$.
\end{lem}

\begin{proof}
    The second assertion is clear seeing $\dim(V^0_{\ess,g_i})<\infty$. For the first assertion, assume $h=\LL_{X_i}g_i+h_i$ with respect to the $L^2_{g_i,f_i}$-decomposition  $\textnormal{Sym}^2(M)=\Ima\,\Di^*_{g_i,f_i}\perp \Ker\,\Di_{g_i,f_i}$ for each $i$, where $X_i$ is orthogonal to all $g_i$-Killing fields and $\Di_{g_i,f_i}h_i=0$. 
    Then $\PP_{\Lie,g_i}h=\LL_{X_i}g_i$.
    By Lemma \ref{l: standard elliptic estimates} that $\|X_i\|_{C^{k+1,\alpha}}\le C\|h\|_{C^{k,\alpha}}$, so after passing to a subsequence we may assume $X_i\to X_{\infty}$ in the $C^{k+1,\alpha}$-sense, and $h_i\to h_{\infty}$ in the $C^{k,\alpha}$-sense. This implies $h=\LL_{X_{\infty}}g_{\infty}+h_{\infty}$ and $\Di_{g_\infty,f_\infty}h_\infty=0$. So $\PP_{\Lie,g_\infty}h=\LL_{X_{\infty}}g_{\infty}$, and $\PP_{\Lie,g_i}h\to \PP_{\Lie,g_\infty}h$ in the $C^{k,\alpha}$-sense, which proves the first assertion.
\end{proof}

The following lemma gives a quantitative comparison of the $C^{k,\alpha,*}$-norms with respect to two nearby shrinkers.

\begin{lem}\label{l: changing shrinker gauge}
Let $2\le k\in\mathbb N,\alpha\in(0,1),C_0>0$. 
Let $(M,g_i,f_i)$, $i=1,2$, be two compact shrinkers.
There exist $C(k,\alpha,C_0,|\Rm_{g_1}|),\eps(k,\alpha,C_0,|\Rm_{g_1}|)>0$ such that if $\dim (V^0_{\ess,g_1})=\dim (V^0_{\ess,g_2})$ and $\|f_2-f_1\|_{C^{k,\alpha}_{g_1}}\le C_0\|g_2-g_1\|_{C^{k,\alpha}_{g_1}}\le\eps$, then
for any $h\in \textnormal{Sym}^2(M)$, we have
    \begin{equation}\label{eq:neu}
    \begin{split}
        \bigg|\|\mathcal{P}^0_{\ess,g_1}h\|_{C^{k,\alpha,*}_{g_1}}-\|\mathcal{P}^0_{\ess,g_2}h\|_{C^{k,\alpha,*}_{g_2}}\bigg|
        &\le C\|h\|_{C^{k,\alpha}_{g_1}}\|g_1-g_2\|_{C^{k,\alpha}_{g_1}}.
    \end{split}
\end{equation}
Moreover, if $h\in\Ker\,\Di_{g_1,f_1}$, then
\begin{equation}\label{l: changing shrinker gauge and divergence free}
    \begin{split}
        \Big|\|h\|_{C^{k,\alpha,*}_{g_1}}-\|h\|_{C^{k,\alpha,*}_{g_2}}\Big|
        &\le C\|h\|_{C^{k,\alpha}_{g_1}}\|g_1-g_2\|_{C^{k,\alpha}_{g_1}}.
    \end{split}
\end{equation}

\end{lem}

\begin{proof}
For simplicity, we abbreviate $C^{k,\alpha}_{g_1}$ as $C^{k,\alpha}$, $L_{g_i,f_i}$ as $L_i$ for $i=1,2$. We first claim
\begin{claim}\label{claim:h-neu}
  For any $h\in V^0_{\ess,g_2}$ with $\|h\|_{L^2_{f_2,g_2}}=1$, we have
    \begin{align}\label{eq:h-neu of h}
        \|h-\mathcal{P}^0_{\ess,g_1}h\|_{C^{k,\alpha}}\le C\|h\|_{C^{k,\alpha}}\|g_1-g_2\|_{C^{k,\alpha}}.
    \end{align}
\end{claim}

\begin{proof}[Proof of Claim \ref{claim:h-neu}]
By the assumption $\|f_2-f_1\|_{C^{k,\alpha}_{g_1}}\le C_0\|g_2-g_1\|_{C^{k,\alpha}_{g_1}}$, $\Di_{g_2,f_2}h=0$, and Lemma \ref{l: standard elliptic estimates} we obtain
    \begin{align}\label{eq:lie of h}
        \big\|L_1(\mathcal{P}_{\Lie,g_1}h)\big\|_{C^{k-2,\alpha}}\le\|\mathcal{P}_{\Lie,g_1}h\|_{C^{k,\alpha}}\le C\|\Di_{g_1,f_1}h\|_{C^{k-1,\alpha}}\le C\|h\|_{C^{k,\alpha}}\|g_1-g_2\|_{C^{k,\alpha}}.
    \end{align}
    This implies by  $L_1(\mathcal{P}^0_{\ess,g_1}h)=0$ that
    \begin{align*}
        \big\|L_1(\mathcal{P}_{\rest,g_1}h)-L_1h\big\|_{C^{k-2,\alpha}}=\big\|L_1(\mathcal{P}_{\Lie,g_1}h)\big\|_{C^{k-2,\alpha}}\le C\|h\|_{C^{k,\alpha}}\|g_1-g_2\|_{C^{k,\alpha}}.
    \end{align*}
    So by $L_2h=0$ and thus $\|L_1h\big\|_{C^{k-2,\alpha}}\le C\|h\|_{C^{k,\alpha}}\|g_1-g_2\|_{C^{k,\alpha}}$, we obtain
    \[\big\|L_1(\mathcal{P}_{\rest,g_1}h)\big\|_{C^{k-2,\alpha}}\le C\|h\|_{C^{k,\alpha}}\|g_1-g_2\|_{C^{k,\alpha}}.\]
    Since $\mathcal{P}_{\rest,g_1}h$ is $L^2_{g_1,f_1}$-orthogonal to $\Ker\,L_1=V^0_{\ess,g_1}\perp V^0_{\Lie,g_1}$, by the parabolic Schauder estimate \cite{krylov} we have  
    \[\|\mathcal{P}_{\rest,g_1}h\|_{C^{k,\alpha}}\le C\|h\|_{C^{k,\alpha}}\|g_1-g_2\|_{C^{k,\alpha}},\]
    which together with \eqref{eq:lie of h} implies the claim.
\end{proof}

Assume $\{\varphi_i\}_{i=1}^{m}$ form an $L^2_{g_2,f_2}$-orthonormal basis of $V^0_{\ess,g_2}$. By Claim \ref{claim:h-neu} we have
  \begin{align*}
        \|\mathcal{P}^0_{\ess,g_1}(\varphi_i)-\varphi_i\|_{C^{k,\alpha}}\le C\|g_1-g_2\|_{C^{k,\alpha}},
    \end{align*}
    and thus by conducting the Gram-Schmidt reduction to $\{\mathcal{P}^0_{\ess,g_1}(\varphi_i)\}_{i=1}^{m}$, and using  $\dim (V^0_{\ess,g_1})=\dim (V^0_{\ess,g_2})$, we obtain an $L^2_{f_1,g_1}$-orthonormal basis $\{\widetilde\varphi_i\}_{i=1}^{m}$ of $V^0_{\ess,g_1}$, satisfying $$\|\tilde{\varphi}_i-\varphi_i\|_{C^{k,\alpha}}\leq C\|g_1-g_2\|_{C^{k,\alpha}}.$$
    The first assertion \eqref{eq:neu} follows immediately from this.

    For the second assertion \eqref{l: changing shrinker gauge and divergence free}, using $\Di_{g_1,f_1}h=0$, $\|f_2-f_1\|_{C^{k,\alpha}_{g_1}}\le C\|g_2-g_1\|_{C^{k,\alpha}_{g_1}}$, and Lemma \ref{l: standard elliptic estimates}, we obtain
$$\|\mathcal{P}_{\Lie,g_2}h\|_{C^{k,\alpha}_{g_1}}\le C\|\Di_{g_2,f_2}h\|_{C^{k-1,\alpha}}\le C\|g_1-g_2\|_{C^{k,\alpha}_{g_1}}\|h\|_{C^{k,\alpha}_{g_1}}.$$
Thus, together with $\mathcal{P}_{\Lie,g_1}h=0$ and the first assertion, we can obtain \eqref{l: changing shrinker gauge and divergence free}.
\end{proof}

\begin{remark}\label{re:re}
Note the assumptions 
in Lemma \ref{l: changing shrinker gauge} hold 
in the following two cases:
\begin{enumerate}
    \item $g_2=\phi_{X}^*g_1$, where $X$ is a vector field with $|X|,|\nabla X|\le\delta(k,\alpha,|\Rm_{g_1}|,\textnormal{inj}_{g_1})$, and $\phi_X$ is the diffeomorphism of $X$ from Lemma \ref{l: key lemma for regularity}.
    \item $g_1$ is an integrable shrinker, and there is a $C^1$-curve of shrinkers $g_s\in \mathcal{S}ol_f$, $s\in[1,2]$ connecting $g_1$ to $g_2$ (see Definition \ref{def:integrable}). 
\end{enumerate}
First, it is clear $\dim(V^0_{\ess,g_1})=\dim(V^0_{\ess,g_2})$. Moreover, since there is a $C^1$-curve of shrinkers $g_s$, $s\in[1,2]$, connecting $g_1,g_2$ with $\|\partial_sg_s\|_{C^{k,\alpha}}\le C(k,\alpha,|\Rm_{g_1}|)\|g_1-g_2\|_{C^{k,\alpha}}$, it follows by \eqref{eq:v_h} that $\|f_1-f_2\|_{C^{k,\alpha}}\le C\|g_1-g_2\|_{C^{k,\alpha}}$.
\end{remark}

Now we prove the main result of this section.
For any metric near a compact integrable shrinker,
we find a shrinker, with respect to which the metric is divergence-free and has zero projection in the essential neutral eigenspace.
We use $C^{k,\alpha,*}_\eps(\bar g)$ (resp. $C^{k,\alpha}_\eps(\bar g)$) to denote the set of all Riemmanian metrics $g$ such that $\|g-\bar g\|_{C^{k,\alpha,*}_{\bar g}}\le\eps$ (resp. $\|g-\bar g\|_{C^{k,\alpha}_{\bar g}}\le\eps$).

\begin{theorem}[A generalized Slice-theorem]\label{t: slice theorem}
Let $1\le k\in\mathbb N, \alpha\in(0,1)$, and $(M,g_0,f_0)$ be an integrable compact shrinker. There exist $C_0,\eps>0$ such that for any metric $g\in C^{k,\alpha}_{\eps}(g_0)$,
there is a compact integrable shrinker metric $\bar g\in C^{k,\alpha}_{2C_0\eps}(g_0)$ 
such that
    \[\mathcal{P}_{\Lie,\bar g}(g-\bar g)=\mathcal{P}^0_{\bar g}(g-\bar g)=0.\]
\end{theorem}

\begin{proof}  
Let $C_0>0$ be from \eqref{eq:norms are comparable}, and $C=C(k,\alpha,g_0)>0$.
First, since $g_0$ is an integrable shrinker, we may choose $\eps$ small enough so that $\dim(V^0_{\ess})$ are equal for all shrinkers in $C^{k,\alpha}_{2C_0\eps}(g_0)$.
Let $g_i\in C^{k,\alpha}_{2C_0\eps}(g_0)$ be a sequence of shrinker metrics such that 
\[\lim_{i\to\infty}\|g-g_i \|_{C^{k,\alpha,*}_{g_i}}=\inf_{\overline g\in C^{k,\alpha}_{2C_0\eps}(g_0) \textnormal{ is a shrinker}}\|g-\overline g\|_{C^{k,\alpha,*}_{\overline g}}\le C_0\eps.\]
Then after passing to a subsequence and using Lemma \ref{c:converge}, we may assume $g_i\to \bar g$ in the $C^{k,\alpha}$-sense, and 
\begin{equation}
    \|g-\bar g \|_{C^{k,\alpha,*}_{\bar g}}=\inf_{\overline g\in C^{k,\alpha}_{2C_0\eps}(g_0) \textnormal{ is a shrinker}}\|g-\overline g\|_{C^{k,\alpha,*}_{\overline g}}\le C_0\eps.
\end{equation}
In the following, we will show the theorem holds for $\bar g$. 
The key idea is that if $\PP_{\Lie}h\neq0$ or $\PP^0_{\ess}h\neq0$, then we can perturb $\bar g$ by the corresponding diffeomorphism or deform 
$\bar g$ by the integrability to almost cancel these projections.
By the definition of the 
$C^{k,\alpha,*}$-norm, this results in a strictly smaller norm, thus a contradiction. 
For convenience, all norms, derivatives, and projections are taken with respect to $\bar g$ unless specified otherwise. 

Let $h=g-\bar g$.
Suppose by contradiction $\PP_{\Lie}h\neq0$. Let $h_1=h-\mathcal{P}_{\Lie}h$, then we have
    \begin{equation}\label{eq:h1 small0}
        \|h_1\|_{C^{k,\alpha,*}}=\|h\|_{C^{k,\alpha,*}}-\|\mathcal{P}_{\Lie}h\|_{C^{k,\alpha,*}}.
    \end{equation}
    Assume $\mathcal{P}_{\Lie}h=\LL_X\bar g$ for some vector field $X$ orthogonal to all $\bar g$-Killing field. Let $\phi_X$ be the diffeomorphism from Lemma \ref{l: key lemma for regularity} \eqref{eq:third}, then for $h':=g-\phi_X^*\bar g$ we have
    \begin{equation}\label{eq:h'}
    \begin{split}
        \|h'-h_1\|_{C^{k,\alpha,*}}
        =\|\bar g-\phi_X^*\bar g+\LL_X\bar g\|_{C^{k,\alpha,*}}
        \le C\|\mathcal{P}_{\Lie}h\|_{C^{k,\alpha,*}}^2.
    \end{split}
    \end{equation}
Therefore, by \eqref{eq:h1 small0} and taking $\eps$ sufficiently small we have
\begin{equation}\label{eq:in the old norm}
    \|h'\|_{C^{k,\alpha,*}}\le\|h\|_{C^{k,\alpha,*}}-0.9\|\mathcal{P}_{\Lie}h\|_{C^{k,\alpha,*}}. 
\end{equation}
To convert the $C^{k,\alpha,*}$-norm to $C^{k,\alpha,*}_{\phi^*_X\bar g}$-norm in this estimate, using $\Di_{g_1}h_1=0$ by Remark \ref{re:re}, we can apply Lemma \ref{l: changing shrinker gauge} to $h_1,\bar g,\phi^*_X\bar g$ so that we get
\[\bigg|\|h_1\|_{C^{k,\alpha,*}_{\phi^*_X\bar g}}-\|h_1\|_{C^{k,\alpha,*}}\bigg|
        \le C\|h_1\|_{C^{k,\alpha,*}}\|\bar g-\phi^*_X\bar g\|_{C^{k,\alpha}}\le 0.1\|\mathcal{P}_{\Lie}h\|_{C^{k,\alpha,*}}.\]
Together with \eqref{eq:h'} and taking $\eps$ small this implies
\begin{equation}\label{eq:compare g sigma X and g sigma norms}
    \begin{split}
        \bigg|\|h'\|_{C^{k,\alpha,*}_{\phi^*_X\bar g}}-\|h'\|_{C^{k,\alpha,*}}\bigg|
        &\le 0.2\|\mathcal{P}_{\Lie}h\|_{C^{k,\alpha,*}}.
    \end{split}
\end{equation}
Combining \eqref{eq:compare g sigma X and g sigma norms} and \eqref{eq:in the old norm}, we obtain
\[\|h'\|_{C^{k,\alpha,*}_{\phi^*_X\bar g}}\le\|h\|_{C^{k,\alpha,*}}-0.7\|\mathcal{P}_{\Lie}h\|_{C^{k,\alpha,*}}<\|h\|_{C^{k,\alpha,*}}. \]
In particular, this implies $\phi^*_X\bar g\in C^{k,\alpha}_{2C_0\eps}(g_0)$, which contradicts the infimum choice of $\bar g$. Thus, we proved $\mathcal{P}_{\Lie}h=0$.

Therefore, we can assume by contradiction that $\mathcal{P}_{\Lie}^0h=0$ and $\mathcal{P}_{\ess}^0h\neq0$. Assume $h=\mathcal{P}_{\ess}^0h+h_1$, then $\Di_{\bar g}h_1=0$, and
    \begin{equation}\label{eq:h1 small}
        \|h_1\|_{C^{k,\alpha,*}}=\|h\|_{C^{k,\alpha,*}}-\|\mathcal{P}_{\ess}^0h\|_{C^{k,\alpha,*}}.
    \end{equation}
By the integrability of $\bar g$, there exists a $C^1$-curve of shrinkers $\bar g_s\in \mathcal{S}ol_f$, $s\in [0,1]$, connecting $\bar g=\bar g_0$ to $\bar g':=\bar g_1$, with 
 $$\|\bar g-\bar g'+\mathcal{P}_{\ess}^0h\|_{C^{k,\alpha,*}}\le \eta\|\mathcal{P}_{\ess}^0h\|_{C^{k,\alpha,*}},$$
 where $\eta(\eps)\to0$ as $\eps\to0$.
    Hence, letting $h'=g-\bar g'$, we have 
    \begin{equation}\label{eq:h'-h_1}
    \begin{split}
        \|h'-h_1\|_{C^{k,\alpha,*}}
        \le \eta\|\mathcal{P}_{\ess}^0h\|_{C^{k,\alpha,*}}.
    \end{split}
    \end{equation}
Therefore, by  \eqref{eq:h1 small} we have
\begin{equation}\label{eq:h'C}
    \|h'\|_{C^{k,\alpha,*}}\le\|h\|_{C^{k,\alpha,*}}-(1-\eta)\|\mathcal{P}_{\ess}^0h\|_{C^{k,\alpha,*}}. 
\end{equation}
To convert the $C^{k,\alpha,*}$-norm to $C^{k,\alpha,*}_{\bar g'}$-norm in this estimate, using $\Di_{\bar g}h_1=0$ by Remark \ref{re:re}, we can apply \eqref{l: changing shrinker gauge and divergence free} to $h_1$ and obtain
\begin{equation*}
    \begin{split}
        \bigg|\|h_1\|_{C^{k,\alpha,*}_{\bar g'}}-\|h_1\|_{C^{k,\alpha,*}}\bigg|
        &\le C\|h_1\|_{C^{k,\alpha,*}}\|\bar g-\bar g'\|_{C^{k,\alpha}}\le 0.1\|\mathcal{P}_{\ess}^0h\|_{C^{k,\alpha,*}},
    \end{split}
\end{equation*}
where in the last inequality we used that $\eps$ is sufficiently small so that $CC_0\eps<0.05$ and $\|\bar g-\bar g'\|_{C^{k,\alpha}}\le 2\|\mathcal{P}_{\ess}^0h\|_{C^{k,\alpha,*}}$.
Combining this with \eqref{eq:h'-h_1} yields
\begin{equation*}
        \bigg|\|h'\|_{C^{k,\alpha,*}_{\bar g'}}-\|h'\|_{C^{k,\alpha,*}}\bigg| 
        \le (0.1+2C_0\eta) \|\mathcal{P}_{\ess}^0h\|_{C^{k,\alpha,*}}<0.2\|\mathcal{P}_{\ess}^0h\|_{C^{k,\alpha,*}}.
\end{equation*}
Hence, together with \eqref{eq:h'C} we get
\[\|h'\|_{C^{k,\alpha,*}_{\bar g'}}\le\|h\|_{C^{k,\alpha,*}}-0.7\|\mathcal{P}_{\ess}^0h\|_{C^{k,\alpha,*}}<\|h'\|_{C^{k,\alpha,*}},\]
which implies $\bar g'\in C^{k,\alpha}_{2C_0\eps}(g_0)$. This contradicts the infimum choice of $\bar g$. Thus, $\PP^0_{\ess}h=0$.
\end{proof}

\section{Fast convergence of Ricci flows}\label{sec:exp_decay}

In this section, we fix $(M,\bar{g},f)$ to be a compact integrable shrinker.
For any ancient (resp. immortal) flows converging to the shrinker at time $\infty$ (resp. $-\infty$), we show the rescaled flows converge exponentially to the shrinker. This allows us to find ancient (resp. immortal) HMHFs between these Ricci flows and the flow of the shrinker.  
\subsection{Spectral analysis using the Entropy}

In this subsection, we use the $\mu$-entropy to estimate the closeness of a nearby metric to a shrinker metric, and also analyze the projections of the difference to various eigenspaces.
We fix $(M,\bar{g},f)$ to be a compact shrinker, and abbreviate the weighted $L^2_f$-norm $\|\cdot\|_{L^2_f} $ and the corresponding inner production $\langle \;,\;\rangle_{L_f^2}$ as $\|\cdot\|$ and $\langle \;,\;\rangle$, respectively.

\begin{lem}\label{l: expanding mu to the third order}
Let $\alpha\in(0,1)$.
There exist positive $\eps(\bar{g},\alpha),C(\bar{g},\alpha)>0$ such that for any $h$ with $\|h\|_{C^{2,\alpha}}\le \eps$, we have
    \begin{equation*}
        \left|\mu(\bar{g}+h,1)-\mu(\bar{g},1)-D^2\mu(\bar{g},1)[h,h]\right|\le C\|h\|^2_{H^1}\|h\|_{C^{2,\alpha}}.
    \end{equation*}
\end{lem}

\begin{proof}
    Since for any smooth function $f:[0,1]\to\R$, we have
    \begin{equation*}
        f(1)=f(0)+f'(0)+\tfrac12f''(0)+\int_0^1\int_0^u\int_0^t f'''(s)\,ds\,dt\,du.
    \end{equation*}
    By the same argument in \cite[Proposition 2.2]{Has12} for the $\lambda$-functional, we can see the third variation of the $\mu$-entropy satisfies
    \begin{equation*}
        \left|\frac{d^3}{ds^3}\mu(\bar{g}+sh,1)[h,h,h]\right|\le C\|h\|^2_{H^1}\|h\|_{C^{2,\alpha}},
    \end{equation*}
    when $\eps$ is sufficiently small. 
    So the lemma follows by seeing $D\mu(\bar g,1)[h]=0$ and letting $f(s)=\mu(\bar{g}+sh,1)$ in the above expansion.
\end{proof}

\begin{lem}\label{l: second vari}
    There is $C(\bar g)>0$ such that any symmetric 2-tensor $h$ satisfies
\begin{equation}
    \|h\|_{H^1}^2\leq C(\|\mathcal{P}^0h\|^2 +\langle h,L\mathcal{P}^+h\rangle-\langle h,L\mathcal{P}^-h\rangle).
\end{equation}
\end{lem}

\begin{proof}
It follows by the observation $\|h\|_{H^1}^2\le -\langle h,Lh\rangle+C\|h\|^2 $ by integration by parts. 
\end{proof}

For a nearby metric $g$ with an entropy larger than $\bar g$,
the next lemma shows that under an almost divergence-free and neutral-mode-free gauge, the projection of the perturbation $g-\bar g$ onto the essential unstable eigenspace has a definite portion.

\begin{lem}[Dominance of essential unstable mode]\label{l: entropy controls unstable}
There exist $C(\bar{g}),\eps_0(\bar{g})>0$ such that the following holds for all $\eps\le\eps_0$: 
    Let $h$ be a symmetric 2-tensor with $\|h\|_{C^3}\le \eps$. Suppose
    \begin{enumerate}
        \item $\|\mathcal{P}_{\Lie}h\|_{H^1}^2+\|\mathcal{P}_{\ess}^0h\|^2 \le\eps \|h\|^2_{H^1}  $; 
        \item $\mu(\bar{g}+h,1)\ge \mu(\bar{g},1)$.
    \end{enumerate}
    Then we have 
\begin{align*}
&    \eps^{-1}(\|\mathcal{P}_{\Lie}h\|^2_{H^1}+\|\mathcal{P}^0_{\ess}h\|^2) +\|\mathcal{P}^-h\|^2_{H^1}\leq  C \|\mathcal{P}^+_\ess h\|^2 .
\end{align*}
\end{lem}

\begin{proof}
Let $C_1>0$ be from Lemma \ref{l: expanding mu to the third order} so that by $\|h\|_{C^{2,\alpha}}\le\eps$ and Assumption (2) we have
$$D^2\mu(\bar{g},1)[h,h]\ge -C_1\|h\|_{C^{2,\alpha}} \|h\|_{H^1}^2\ge -C_1\varepsilon\|h\|_{H^1}^2.$$
The second variation formula  \eqref{eq:second derivative} implies $D^2\mu(\bar{g},1)[h,h]=\tfrac12\langle \mathcal{P}_{\ess}h,L\mathcal{P}_{\ess}h\rangle$.
Hence, 
by Assumption (1) and Lemma \ref{l: second vari}, there is $C_2(C_1,C)>0$ such that
\begin{align*}
    D^2\mu(\bar{g},1)[h,h]&\le \tfrac12\langle h,Lh\rangle+C\eps\|h\|_{H^1}  ^2\\
    &\le \big(\tfrac{1}{2}+C_2\eps\big)\langle h,L\mathcal{P}^+h\rangle +\big(\tfrac{1}{2}-C_2\eps\big)\langle h,L\mathcal{P}^-h\rangle-C_1\eps  \|h\|_{H^1}^2.
\end{align*}
Thus, combining these estimates we obtain
\begin{equation*}
    \big(\tfrac{1}{2}+C_2\eps\big)\langle h,L\mathcal{P}^+h\rangle +\big(\tfrac{1}{2}-C_2\eps\big)\langle h,L\mathcal{P}^-h\rangle \ge0.
\end{equation*}
The lemma thus follows by combining this with Assumption (1) and taking $\eps_0$ small.
\end{proof}

\begin{lem}[Dominance of stable mode]\label{l:entropy_forward}
There are $C(\bar{g}),\eps_0(\bar{g})>0$ such that the following significance holds for all $\eps\le\eps_0$. Let $h$ be a symmetric 2-tensor with $\|h\|_{C^3}\le \eps$ satisfying
\begin{enumerate}
    \item $\|\PP^0 h\|^2 +\|\PP^+_\Lie h\|^2 \le \eps\|h\|^2_{H^1} $;
    \item $\mu(\bar{g}+h,1)\le\mu(\bar{g},1)$.
\end{enumerate}
Then we have 
$$\eps^{-1}\big(\|\PP^0 h\|^2 +\|\PP^+_\Lie h\|^2 \big)+\|\mathcal{P}_{\ess}^+h\|^2 \le C\big|\langle h,L\mathcal{P}^-h\rangle\big|.$$
\end{lem}

\begin{proof}
    This follows in a similar way as Lemma \ref{l: entropy controls unstable}.
\end{proof}
 
\begin{lem}[Upper bound by entropy]\label{l: Upper bound by entropy}
There are $\eps_0(\bar{g}),C(\bar{g})>0$ such that given any symmetric 2-tensor $h$ satisfying $\|h\|_{C^3}\le \eps_0$ and
\begin{equation}\label{eq:assume of dominance}
    \|\mathcal{P}_{\Lie}h\|^2 +\big|\langle h,L\mathcal{P}^-h\rangle\big|  +\|\mathcal{P}_{\ess}^0h\|  ^2\le\eps_0\|h\| ^2,
\end{equation}
the following holds 
$$\|h\| \le C(\mu(\bar{g}+h,1)-\mu(\bar{g},1)).$$
\end{lem}

\begin{proof}
Note that the assumption \eqref{eq:assume of dominance} implies $\|\mathcal{P}_{\Lie}h\|_{H^1}^2\le C\eps_0\|h\|^2$.
Thus, combining with  \eqref{eq:assume of dominance} and Lemma \ref{l: second vari} yields
 \begin{align*}
    D^2\mu(\bar{g},1)[h,h]&
    \ge \tfrac12\langle h,Lh\rangle-C\|\mathcal{P}_{\Lie}h\|_{H^1}^2\geq \tfrac12\langle h,L\mathcal{P}^+h\rangle-C\varepsilon_0 \|h\|^2\ge C^{-1}\|h\|^2_{H^1}.
\end{align*}
Moreover, since $\|h\|_{C^3}\le\eps_0$, by Lemma \ref{l: expanding mu to the third order} we have 
\[ D^2\mu(\bar{g},1)[h,h]-C\eps_0\|h\|_{H^1}^2\le \mu(\bar{g}+h,1)-\mu(\bar{g},1).\]
The lemma thus follows by combining these estimates.
\end{proof}

\subsection{Fast convergence of ancient Ricci flows}

In Theorem \ref{p: g_t converge to g_Sigma exponentially} we establish the exponential convergence of a rescaled ancient Ricci flow to a compact integrable shrinker.

The following lemma allows us to reduce finitely many generic eigentensors from any small $h\in\Sym^2(M)$ by changing a diffeomorphism. 
 
\begin{lem}[Lie derivative reduction by changing gauge]\label{l: cancel Lie by diffeomorphisms}
Let $(M,\bar{g},f)$ be a compact shrinker and fix $\lambda\in\mathbb R$.
   For each $k\in\mathbb N$, there exist $C,\eps_0>0$ such that the following holds:
    Let $h\in\Sym^2(M)$ with $\|h\|_{C^{k+1}}\le \eps_0$.
    There is a diffeomorphism $\phi : M \to M$,
    such that
    \begin{enumerate}
    \item $\big\|\exp_{\id}^{-1}\phi\big\|_{C^{k}}\le C\Big\|\mathcal P_{\Lie}^{\ge\lambda}h\Big\| $,\label{ass:exp-id}
        \item $\mathcal{P}_{\Lie}^{\ge\lambda} (\phi^* (\bar{g}+ h) - \bar{g})=0$,\label{ass:proj=0}
        \item $\Big\|\phi^* (\bar{g}+ h) - (\bar{g}+h)+{\mathcal{P}_{\Lie}^{\ge\lambda}h}\Big\|_{C^{k}}\le C\|h\|_{C^{k+1}}\Big\|\mathcal{P}_{\Lie}^{\ge\lambda} h\Big\|$,\label{eq:smaller conf}
    \end{enumerate}
    where $\PP^{\ge\lambda}_{\Lie}$ is the projection to generic eigentensors of eigenvalues not smaller than $\lambda$.
\end{lem}

\begin{proof}
Denote $\| h\|_{C^{k+1}} =\eta\le\eps_0$. 
By abuse of notation, we let $V_{\Lie}^{\ge\lambda}$ denote the space of vector fields $X$ such that $\LL_X \bar{g}\in V^{\ge\lambda}_{\Lie}$ and orthogonal to all Killing fields, and we identify $X$ with $\LL_X \bar{g}$.
For any $X\in V_{\Lie}^{\ge\lambda}$, let $\phi_X$ be the diffeomorphism from Lemma \ref{l: key lemma for regularity}, then 
\begin{equation}\label{eq:defn of X}
   h'_X :=\phi_X^*(\bar g+h)-\bar g= (\phi^*_X \bar g - \bar g) + \phi^*_X h=h+\LL_X\bar g+O(\|X\|^2)+ \eta\, O(\|X\|).
\end{equation}
In this proof, $h=O(A)$, with tensor $h$ and $A>0$, implies $\|h\|_{C^k} \leq CA$ for some $C>0$ depending on $\bar g,k$. Also, we note that $\|X\|_{C^k} \leq C\|X\|$ holds, because $V^{\ge\lambda}_{\Lie}$ has finite dimension. 

We define $X_0\in V_{\Lie}^{\ge\lambda}$ by $\mathcal{P}_{\Lie}^{\ge\lambda}(h)=\LL_{-X_0}\bar g$, and let $h_1:=h+\LL_{X_0}\bar g$. Then, 
\begin{equation}\label{eq:central_vector}
\|X_0\|\le C\eta,
\end{equation}
and $\mathcal{P}_{\Lie}^{\ge\lambda}(h_1)=0$. Thus, combining with \eqref{eq:defn of X} yields $h'_{X_0}=h_1+\eta\cdot O(\|X_0\|)$,
 and therefore we have $\mathcal{P}_{\Lie}^{\ge\lambda}(h'_{X_0})=\eta\cdot O(\|X_0\|)$.
Together with \eqref{eq:defn of X}  and \eqref{eq:central_vector}, this implies
\begin{equation}\label{eq:X perturb}
\begin{split}
    \mathcal{P}_{\Lie}^{\ge\lambda}(h_X')&
    =\mathcal{P}_{\Lie}^{\ge\lambda}(h'_X-h'_{X_0})+\eta\cdot O(\|X_0\|)\\
    &=X-X_0+O(\|X-X_0\|^2)+\eta\cdot O(\|X_0\|+\|X\|).    \end{split}
\end{equation}
We define $\Phi: (V_{\Lie}^{\ge\lambda},\|\cdot\|)\to (V_{\Lie}^{\ge\lambda},\|\cdot\|)$ by $\Phi(X)=-\mathcal{P}_{\Lie}^{\ge\lambda} (h'_{X})+X$. Then,  
\[\Phi(X)-X_0=O(\|X\|^2)+\eta\cdot O(\|X_0\|+\|X\|).\]
So choosing $\eps_0$ sufficiently small, we have $\Phi(B(X_0,\|X_0\|))\subset B(X_0,\|X_0\|)$ {in $V_{\Lie}^{\ge\lambda}$ with respect to the norm $\|\cdot\|$}. Since $\Phi$ is continuous, the Brouwer fixed point theorem implies there exists {$X_1\in V_{\Lie}^{\ge\lambda}$ such that $\|X_1-X_0\|\leq \|X_0\|$,} $\Phi(X_1)=X_1$, and $\mathcal{P}_{\Lie}^{\ge\lambda} (h'_{X_1})=0$. So the diffeomorphism $\phi_{X_1}$ satisfies Assertion \eqref{ass:proj=0}. Lemma \ref{l: key lemma for regularity} implies Assertion \eqref{ass:exp-id} and also
\begin{align*}
   \phi_{X_1}^*(\bar g+ h)-(\bar g+h)+\PP^{\ge\lambda}_\Lie h&=\LL_{X_1}(\bar g+ h)+\PP^{\ge\lambda}_\Lie h+ O(\|X_1\|^2)\\
   &=\LL_{X_1-X_0}\bar g+ O(\eta\|X_0\|),
\end{align*}
where in the second equation we used $\|X_1\|\le 2\|X_0\|\le C\eta$ and $\mathcal{P}_{\Lie}^{\ge\lambda}h=\LL_{-X_0}\bar g$.
{Moreover, by \eqref{eq:X perturb} and $\mathcal{P}_{\Lie}^{\ge\lambda} (h'_{X_1})=0$ we see $\|X_1-X_0\|\le C\eta\|X_0\|$, which implies $\|\LL_{X_1-X_0}\bar g\|\le C\eta\|X_0\|$.} This shows Assertion \eqref{eq:smaller conf}.
\end{proof}

In the next lemma, we recall the short-time existence of HMHF and local derivative estimates of the RDTFP from \cite[Appendix A]{BK}.

\begin{lemma}\label{l:shortHMHF}
For any $A,T>0$ and $k\in\mathbb N$, there exist $C,\bar\eps>0$ such that the following holds:
    Let $M$ be a compact manifold and $g_t,g'_t$ be two Ricci flows on $M$, $t\in[0,T]$, which satisfy $|\nabla^m\Rm_{g_t}|,|\nabla^m\Rm_{g'_t}|\le A$, for $m=0,\cdots,10$. Suppose $h_0=g'_0-g_0$ satisfies $\|h_0\|_{C^{k}}\le \eps\le \bar\eps$, then for all $t\in[0,T]$,
    \begin{enumerate}
        \item The HMHF $\chi_t$ with initial condition $\chi_0=\id$ exists and $\chi_t$ is a diffeomorphism;
        \item The RDTF perturbation $h_t=(\chi_t^{-1})^*g'_t-g_t$ satisfies $\|h_t\|_{C^k}\le C\eps$.
    \end{enumerate}
    
\end{lemma}

\begin{proof}

Let $0<T^*\le T$ be the largest time so that the assertions hold true.
First, seeing that 
the RDTF perturbation $h_t$ satisfies (see \cite[(A.7)]{BK})  \[
\partial_t |h_t|^2 \leq (g_t + h_t)^{ij} \nabla^2_{ij} |h_t|^2 + C_0 |\mathrm{Rm}_{g_t}| \cdot |h_t|^2,
\]
for some dimensional constant $C_0>0$,
a standard heat kernel estimate implies
\begin{equation}\label{e:keep}
    \|h_t\|_{C^0}\le C\|h_0\|_{C^0}\le C\eps.
\end{equation}
Next,
by \cite[Proposition A.24]{BK} (the short-time existence of HMHF), choosing $\bar\eps$ sufficiently small, then \eqref{e:keep} guarantees that  
assertion (1) holds up to $T$, and the heat kernel estimate on $[0,T]$ implies $\|h_t\|_{C^0}\le C\eps$. 
Lastly, by \cite[Lemma A.14]{BK} (the local derivative estimates for RDTF), we have $\|h_t\|_{C^{k}}\le  C\eps$.
\end{proof}

Next, we prove a lemma which shows that the initial smallness of a metric perturbation is preserved by Ricci flow up to a multiple.

\begin{lemma}\label{lem:pseudolocality}
Let $(M,\bar g_\tau)$, $\tau\in[0,T]$, be a rescaled Ricci flow.
 For any $T>0$ and $3\le k\in\mathbb N$, there exist $C,\bar\eps>0$ with the following significance. Given any metric $g$ on $M$ with $\|\bar g_0-g\|_{C^{k+2}}\le \bar\eps$, the smooth rescaled Ricci flow $g_\tau$  with  $g_0=g$ and the HMHF $\chi_\tau$ between $g_\tau$ and $\bar g_\tau$ with $\chi_0=\id$ exist and stay smooth for $\tau \in [0,T]$.  Moreover, for all $\tau\in[0,T]$
 \[\|\exp^{-1}_{\id}\chi_\tau\|_{C^{k}}+\|g_\tau-\bar g_\tau\|_{C^{k}}\le C\|g-\bar g_0\|_{C^{k+2}},\]
 where the norms and derivatives are with respect to $\bar g_0$.
\end{lemma}

\begin{proof}
Let $\eps:=\|h_0\|_{C^{k+2}}$.
Note $h_\tau=(\chi_{\tau})_*g_\tau-\bar{g}_\tau$ satisfies the rescaled RDTFP equation \eqref{eq:rescaled RDTF Perturbation equation}. {By Lemma \ref{l:shortHMHF}, we have $\|h_\tau\|_{C^{k+2}}\le  C\eps$ for some $C(T)>0$.}

For each fixed $\tau$, by \eqref{eq:phi_X diffeo} in Lemma \ref{l: key lemma for regularity}, the vector field $X_\tau:=\exp_{\id}^{-1}\chi_\tau$ satisfies $C^{-1}\|\chi_\tau^i-x^i\|\leq \|X^i_\tau\|\leq  C\|\chi^i_\tau-x^i\|$  under local normal coordinates $\{x^i\}$ for any $C^{k}$-norm.
Since 
by map Laplacian equation \eqref{eq:Laplacian of map} we have 
\begin{equation}
 \begin{split}
 (\Delta_{\bar{g}_\tau+h_\tau,\bar{g}_\tau}\id)^{\gamma}&=(\bar g_\tau+h_\tau)^{ij}\big(-\Gamma^k_{ij}+\hat\Gamma^k_{ij}\big),
 \end{split}
\end{equation}
where $\Gamma,\widehat\Gamma$ are the Christoffel symbols of $\bar g_\tau+h_\tau,\bar g_\tau$, and using 
$$\partial_\tau\chi_{\tau}=\Delta_{g_\tau,\bar{g}_\tau} \chi_{\tau}=(\Delta_{\bar{g}_\tau+h_\tau,\bar{g}_\tau}\id)\circ \chi_{\tau},$$
it follows that $\|\partial_\tau \chi_\tau\|_{C^{k+1}}\le C\|h_\tau\|_{C^{k+2}}$.
Note also $X^i_\tau=\int_0^\tau \partial_s \chi^i_s\,ds$, 
we have
\begin{align*}
    \|X_\tau\|_{C^{k+1}}
    &\le C\int_0^\tau \|h_s\|_{C^{k+2}}\,ds
    \le C\eps.
\end{align*}
This implies $\|\exp^{-1}_{\id}\chi_\tau\|_{C^{k+1}}\le C\eps$, and together with $\|h_\tau\|_{C^{k+2}}\le  C\eps$ this implies 
$$\|g_\tau-\bar g_\tau\|_{C^{k}}\le \|g_\tau-(\chi_\tau)_*g_\tau\|_{C^{k}}+\|h_\tau\|_{C^{k}}\le C\eps,$$ and finishes the proof.
\end{proof}

{We define the following $H_W$-norm on $\textnormal{Sym}^2(M)$ by
\begin{equation}\label{def:H_W}
    \|h\|_{H_W}^2=\|h\| ^2-\langle \PP^-h, L \PP^- h\rangle :=\sum_{i=1}^\infty (1+ \lambda_i^+)\|\PP_i h\| ^2,
\end{equation}
where $\PP_i$ is the projection to the space spanned by $ h_i$ and $\lambda_i^+=\max\{\lambda_i,0\}$, where $\lambda_1,\lambda_2,\cdots$ are eigenvalues of $L$, see Section \ref{sec: existence}.
It is clear that 
\begin{equation}\label{eq:H^1_H_W}
    C^{-1}\|h\|_{H^1}\le \|h\|_{H_W}\le C\|h\|_{H^1}.
\end{equation}}

Recall that for a shrinker $(M,\bar g,f)$, $\xi_\tau$ denotes the flow of $\nabla f$ with $\xi_{0}=\id$, and $\xi^*_\tau \bar g$ is the rescaled Ricci flow starting from $\bar g.$

\begin{prop}\label{prop: unstable mode has certain portion}
Let $(M,\bar g,f)$ be a compact shrinker.
There exists $\underline T>0$ such that for any 
$T>\underline T$, there exist $\eps,\eps',C>0$ with the following significance.
Suppose $\{g_\tau\}_{\tau\in [0,T]}$ is a rescaled Ricci flow such that $h_0=g_0-\bar g$ satisfies
\begin{enumerate}
    \item\label{assump:i} $\|h_0\|_{C^{13}}\le\eps'$;
    \item\label{assump:ii} $\|\mathcal{P}_{\Lie}h_0\| ^2_{H_W}+\|\mathcal{P}^0 h_0\|^2_{H_W}\le \eps \|h_0\|^2  $;
    \item\label{assump:iii} $\mu(g_0,1)\ge\mu(\bar{g},1)$.
\end{enumerate}
Then there exists a diffeomorphism $\phi:M\to M$, with $h':=\phi^*(\xi^{-1}_{T})^*g_T-\bar{g}$ such that
\begin{enumerate}
    \item\label{eq:final time h is bounded by the entropy} $\|h'\|_{C^{13}}\leq C\min\{\mu(g_T,1)-\mu (\bar{g},1),\|h'\| \}$;
    \item\label{asser2} $\|\mathcal{P}_{\Lie}h'\|^2 _{H_W}+\|\mathcal{P}^0 h'\|^2_{H_W}\le \eps \|h'\|^2  $;
    \item\label{asser3} $\|h'\|\ge 100\|h_0\|$;
    \item\label{asser4} {$\|\exp_{\id}^{-1}\phi\|_{C^{11}}\le C\|h_0\|_{C^{13}}$.}
\end{enumerate}
\end{prop}

\begin{proof}
Let $\chi_\tau$ be the HMHF starting from $\chi_0=\id$ between $g_\tau$ and $\xi^*_\tau\bar g$, then $h_\tau=(\xi_\tau^{-1})^*(\chi_\tau^{-1})^*g_\tau-\bar g$ satisfies the modified rescaled RDTFP equation \eqref{eq:most important eq} with background $\bar g$. 
Moreover, the norms and derivatives are with respect to $\bar g$ in this proof.
From now on, we let $C>0$ denote a general constant that only depends on $\bar g$, and we will choose the constants $C,T,\eps,\eps'$ in the way that each only depends on the preceding ones, namely $C=C(\bar g)$, $T=T(\bar g, C)$, $\eps=\eps(\bar g,C,T)$, and $\eps'=\eps'(\bar g, C, T,\eps)$. To be specific,
\begin{enumerate}
    \item Take $T>0$ so that {$Ce^{-\lambda_+T}<\eps_0$}, where $\eps_0$ is smaller than the $\eps_0$ from Lemma \ref{l: entropy controls unstable}, \ref{l: Upper bound by entropy}, and \ref{l: cancel Lie by diffeomorphisms} (with $k=12$) and $\eps_0<10^{-4}$. We recall $\lambda_\pm$ from \eqref{eq:eigenvalues of L min}.  
    \item Choose $\eps$ small enough to get $C\eps e^{3\lambda^+_1 T}<\eps_0$, where $\lambda^+_1:=\max\{0,\lambda_1\}$ and $\lambda_1$ is the first eigenvalue  from \eqref{eq:eigenvalues of L}. 
    \item Let $\chi_\tau$ be the modified HMHF between $g_\tau$ and $\bar g$ with $\chi_0=\id$ and let $h_\tau=(\chi_{\tau}^{-1})^*g_\tau-\bar{g}$. By Lemma \ref{l:shortHMHF}, there is $\eps'>0$ such that $\chi_\tau$ exists for $[0,T]$, $\chi_\tau$ is a diffeomorphism, and $\|h_\tau\|_{C^{13}}\le  \eps_1<\eps$, where $\eps_1>0$ is chosen so that Lemma \ref{lem:general_V} is applicable, and Lemma \ref{lem:B5} is applicable for any $A\in(C^{-1}\eps,C\eps^{-1})$ and various invariant subspaces which will be clear below.
\end{enumerate}

First, by Assumption \eqref{assump:ii}\eqref{assump:iii} and Lemma \ref{l: entropy controls unstable} we have
\begin{align*}
\|\mathcal{P}^-h_0\|^2_{H_W}+ \eps^{-1}(\|\mathcal{P}_{\Lie}h_0\|^2_{H_W}+\|\mathcal{P}^0h_0\|^2 )\leq  C \|\mathcal{P}^+h_0\|^2 .  
\end{align*}
Thus, we can apply Lemma \ref{lem:B5} with a suitable $\delta=\delta(\bar g)>0$ so that
\begin{align}
    &Ce^{-\lambda_+\tau} \|\mathcal{P}^+h_\tau\|^2 \geq  e^{ \lambda_- \tau}\|\mathcal{P}^-h_\tau\|^2_{H_W} \label{eq:combo1}\\
   &C \eps e^{-\lambda_+\tau} \|\mathcal{P}^+h_\tau\|^2 \geq e^{ \lambda_- \tau}\|\mathcal{P}_{\Lie}^-h_\tau\|^2_{H_W} +e^{- 3\lambda^+_1 \tau}\|\mathcal{P}^+_\Lie h_\tau\|^2 +\|\mathcal{P}^0h_\tau\|^2.\label{eq:combo2}
\end{align}
In particular, these imply $\|h_\tau\|_{H_W}\le C\|\PP^+h_\tau\|$ for $\tau\in [0,T]$. So taking $V=V^+$ in Lemma \ref{lem:general_V} implies $\frac{d}{d\tau}\|\PP^+h_\tau\|_{H_W}^2\geq 0$, namely $\|\PP^+h_\tau\|_{H_W}^2$ monotone increases. Thus,
\begin{equation*}
\sup_{s \in [0,\tau]}\|h_s\|_{H_W} \leq C \|\PP^+h_\tau\|.
\end{equation*}
Hence, by standard parabolic estimates, we have
\[\|h_T\| _{C^{k}}\le  C_k\max_{\tau\in[T-1,T]}\|h_\tau\|\le C_k\|\PP^+h_T\|,\]
where $C_k$ only depends on $\bar g, k$. Furthermore, by integrating \eqref{eq:first1} we get
\begin{align}\label{eq:h_T 101 larger than h_0}
     2e^{\lambda_+ T}\|h_0\|^2 \leq  \|h_T\|^2.
\end{align}

Moreover, since $Ce^{-\lambda_+T}<\eps_0$ and $C\eps e^{3\lambda^+_1 T}<\eps_0$, by \eqref{eq:combo1}, \eqref{eq:combo2} we have
\[\|\mathcal{P}_{\Lie}^+h_T\|^2 +\|\PP^- h_T\|_{H_W}^2  +\|\mathcal{P}^0h_T\|  ^2\le\eps_0\|h_T\| ^2.\]
This enables us to apply Lemma \ref{l: Upper bound by entropy} at $T$ and obtain
\begin{equation}\label{eq:Ck-norm up bdd by entropy}
    \|h_T\| _{C^{13}}\le  C\|\PP^+h_T\|  \le C(\mu(g_T,1)-\mu(\bar{g},1)).
\end{equation}

Note \eqref{eq:combo2} implies $\|\PP^{\ge0}_\Lie h_T\|\le C\varepsilon^{\frac{1}{2}}  e^{\frac32\lambda^+_1 T} \|h_T\|\le C^{-1}\|h_T\|$. 
By Lemma \ref{l: cancel Lie by diffeomorphisms} for $h_T$ (with $k=12$), we can find a diffeomorphism $\psi$ on $M$ such that
\begin{equation}\label{eq:verify four}
\left\|\exp_{\id}^{-1}\psi\right\|_{C^{12}}\le C\left\|\mathcal{P}_{\Lie}^{\ge0}h_T\right\|\le C^{-1}\|h_T\|,
\end{equation}
such that $h':=\psi^*(\bar{g}+h_T)-\bar{g}$ satisfies $\mathcal{P}_{\Lie}^{\ge0}h'=0$ and 
\begin{equation}\label{eq:new Lie}
 {\left\|h'-h_T+\mathcal{P}_{\Lie}^{\ge0}h_T\right\|_{C^{12}}}\le C\|h_T\|_{C^{13}} \left\|\PP^{\ge0}_{\Lie}h_T\right\|\le C\eps\left\|h_T\right\|.
\end{equation}
Let $\phi=\big(\xi_{T}\circ\chi_T^{-1}\circ\xi_{T}^{-1}\big)\circ\psi$, then $h'
=\phi^* (\xi_T^{-1})^*g_T-\bar g$.
We verify that $h'$ and $\phi$ satisfy all assertions: First, note $h'-h_T=\psi^*(\bar{g}+h_T)-(\bar{g}+h_T)$, it follows from \eqref{eq:verify four} that 
\begin{equation}\label{eq:h'-h_T}
    \|h'-h_T\|_{C^{11}}\le C\|\exp_{\id}^{-1}\psi\|_{C^{12}} \le C^{-1}\|h_T\|. 
\end{equation}
This together with \eqref{eq:Ck-norm up bdd by entropy} implies $\|h'\|_{C^{11}}\le C(\mu(g_T,1)-\mu(\bar{g},1))$, which verifies Assertion \eqref{eq:final time h is bounded by the entropy}.
Next, by Lemma \ref{lem:pseudolocality}, there is $C_1(T)>0$ such that {$\|\exp_{\id}^{-1}\chi_T\|_{C^{11}}\le C_1 \|h_0\|_{C^{13}}$}, which together with \eqref{eq:verify four} and $\|h_T\|^2\le e^{3\lambda_+ T}\|h_0\|^2$ implies
$$\|\exp_{\id}^{-1}\phi\|_{C^{11}}\le C_1 \|h_0\|_{C^{13}}+C^{-1}\|h_T\|\le C_1 \|h_0\|_{C^{13}},$$ 
which verifies Assertion \eqref{asser4}.
Since $e^{\lambda_+ T}\ge 100$ and \eqref{eq:h'-h_T} implies
\begin{equation}\label{eq:h'_compare_h_T}
0.9\|h_T\|\leq \|h'\|\leq 1.1\|h_T\|,    
\end{equation}
it follows by \eqref{eq:h_T 101 larger than h_0} that $\|h'\|\ge e^{\frac12 \lambda_+ T}\|h_0\|$, which implies Assertion \eqref{asser3} since $e^{\frac12 \lambda_+ T}>100$.
For Assertion \eqref{asser2}, first, by \eqref{eq:combo2} and $Ce^{-\lambda_+T}<\tfrac12$ we have
\begin{equation}\label{eq:initial_h_T_com}
    \|\mathcal P^0h_T\|^2+ \| \PP^-_\Lie h_T\|_{H_W}^2\le \tfrac12\varepsilon \|h'\|^2.
\end{equation}
Taking $V=V^0_{\ess},V^-_\Lie$ orthogonal to $V^{\ge0}_\Lie$,
by \eqref{eq:new Lie} we have
\begin{align*}
    \| \PP_V (h'-h_T)\|_{H_W}^2&=\left\|\PP_V(h'-h_T+\mathcal{P}_{\Lie}^{\ge0}h_T)\right\|_{H_W}^2
    \le \left\|h'-h_T+\mathcal{P}_{\Lie}^{\ge0}h_T\right\|_{H_W}^2\le  \tfrac12 \eps\|h'\|^2,
\end{align*} 
which together with \eqref{eq:initial_h_T_com} implies
$$\|\mathcal P^0_{\ess}h'\|^2+ \| \PP^-_\Lie h'\|_{H_W}^2\le \varepsilon \|h'\|^2.$$
This together with $\mathcal{P}_{\Lie}^{\ge0}h'=0$ implies Assertion \eqref{asser2}.
\end{proof}

The following is the main theorem in this subsection. It establishes the exponential convergence rate of any ancient rescaled Ricci flow to their asymptotic shrinker, assuming the shrinker is compact integrable.

\begin{theorem}[Theorem \ref{thm:RF_exp}, ancient flow]\label{p: g_t converge to g_Sigma exponentially}
Let $(M,\bar{g},f)$ be a compact integrable shrinker, and $(M,g_\tau)$, $\tau\in(-\infty,0]$, be a rescaled Ricci flow asymptotic to $\bar{g}$.
Then there exist $\delta,C>0$ and a diffeomorphism $\psi:M\to M$ such that 
\begin{equation*}
    \|(\xi_{\tau}^{-1})^*g_\tau-\psi^*\bar{g}\|_{C^{2,\alpha}(\bar g)}\le C\,e^{\delta \tau}.
\end{equation*}
\end{theorem}

\begin{proof}
In the proof we only work with the rescaled modified Ricci flow $(\xi_{\tau}^{-1})^*g_\tau$, so by abuse of notation we denote it by $g_\tau$.
Take $T$ with $e^{\frac12\lambda_+T}\ge100$, and let $C,\eps,\eps'$ be from Proposition \ref{prop: unstable mode has certain portion}, and we may further increase the value of $C$ or decrease $\eps'$ in the course of proof.
We fix a sufficiently negative time $\tau_0<0$ so that $\mu(g_{\tau_0+T},1)-\mu(\bar{g},1)\le C^{-1}\eps'$, and $\phi^*_\tau g_\tau\in C^{2,\alpha}_\eps(\bar{g})$ under a suitable diffeomorphism 
$\phi_\tau$ for all $\tau\le \tau_0$.
For any $\tau_1<\tau_0$, by Theorem \ref{t: slice theorem} there exists an integrable shrinker
$(M,\bar g_1)$ such that $h_0= g_{\tau_1}-\bar g_1$ satisfies $\|h_0\|_{C^{13}}\le\eps'$ and
\[\mathcal{P}_{\Lie,\bar g_1}( g_{\tau_1}-\bar g_1)=\mathcal{P}^0_{\ess,\bar g_1}( g_{\tau_1}-\bar g_1)=0.\]
So by applying Proposition \ref{prop: unstable mode has certain portion} inductively starting from $\tau_1$, we get a sequence of diffeomorphisms $\{\psi_k\}_{k=0}^N$ with $\psi_0=\id$ and $\{h_k:=\psi^*_kg_{\tau_1+kT}-\bar g_1\}_{k=0}^N$, such that $\tau_1+NT\ge \tau_0$ and for all $k\ge 1$, we have
  \begin{enumerate}
    \item $100\|h_{k-1}\|\le\|h_k\| $, and $\|h_k\|_{C^{11}}\le\|h_k\|$,
    \item $\|\exp^{-1}_{\psi_{k-1}}\psi_{k}\|_{C^{11}}\le C\|h_{k-1}\|_{C^{13}}\le \eps'$,
\end{enumerate}
where hereafter the norms and derivatives are taken with respect to $\bar g$.
These imply 
\begin{equation}
    \|\exp^{-1}_{\psi_1}\psi_{k}\|_{C^{11}}\le C\sum_{i=2}^k \|\exp^{-1}_{\psi_{i-1}}\psi_{i}\|_{C^{11}}\le C \sum_{i=1}^k\|h_{i}\|\le C\|h_{k}\|\le \tfrac{C}{100^{N-k}}\eps'.
\end{equation}
So by triangle inequality,
\begin{equation*}
    \| \psi_1^*g_{\tau_1+kT}-\bar g_1\|_{C^{10}}\le \| \psi_1^*g_{\tau_1+kT}-\psi_{k}^*g_{\tau_1+kT}\|_{C^{10}}+\|h_k\|_{C^{10}}\le \tfrac{C}{100^{N-k}}\eps'.
\end{equation*}
By Lemma \ref{lem:pseudolocality}, we may assume $\eps'$ sufficiently small such that this inequality persists to hold for $\|\psi_1^*g_\tau-\bar g_1\|_{C^8}$ on $[\tau_1+kT,\tau_1+(k+1)T]$ with a possibly larger $C$.
Let $\delta=\frac{\log(100)}{T}$ and replacing $\bar g_1$ by $(\psi_1^{-1})^*\bar g_1$, then for all $\tau\in[\tau_1+T,\tau_0-T]$ we have $$\|  g_{\tau}-\bar g_{1}\|_{C^{8}}\le C\eps'e^{\delta\tau}.$$

Choose decreasing $\tau_j\to-\infty$ and apply the same induction starting from $\tau_j$, we obtain a sequence of integrable shrinkers $\{(M,\bar g_j)\}_{j=1}^{\infty}$  with $\|g_{\tau}-\bar g_j\|_{C^{8}}\le C\eps'e^{\delta\tau}$ for all $\tau\in[\tau_j+T,\tau_0-T]$.
In particular, for any ${j_1}\le {j_2}$, this implies $$\|\bar g_{j_1}-\bar g_{j_2}\|_{C^{8}}\le C\eps'e^{\delta \tau_{j_1}}.$$ 
So $\bar g_j$ converge to an integrable shrinker $\bar g_{\infty}$, and
for all $\tau\in[\tau_j+T,\tau_0-T]$,
\[\|g_\tau-\bar{g}_\infty\|_{C^{8}}\le \|g_\tau-\bar g_{j}\|_{C^{8}}+\|\bar g_{j}-\bar g_{\infty}\|_{C^{8}}\le C\eps'e^{\delta\tau}+C\eps'e^{\delta \tau_{j}}\le C\eps'e^{\delta\tau}.\]
This implies $\bar g_\infty$ is isometric to $\bar g$, and the theorem holds by letting $\tau_j\to-\infty$.
\end{proof}

\subsection{Fast convergence of Ricci flows to singularities}
The main result in this subsection is Theorem \ref{p: g_t converge to g_Sigma exponentially_forward}, in which we show that the rescaled flow of a Ricci flow which develops a singularity at $t=0$ modelled on a compact integrable shrinker converges exponentially to the shrinker.
The key lemma is
Proposition \ref{prop: unstable mode has certain portion, forward}, which is the analog of Proposition \ref{prop: unstable mode has certain portion}. 

First, we need the following lemma that compares the $H_W$-norms.

\begin{lem}\label{lem:comparable_norms}
Let $(M,g_i,f_i)$, $i=1,2$, be two compact shrinkers satisfying the same assumptions in Lemma \ref{l: changing shrinker gauge}. Then for any symmetric 2-tensor $h$, we have
    \[0.9\|h\|_{H_{W,g_1}}\le\|h\|_{H_{W,g_2}}\le 1.1\|h\|_{H_{W,g_1}}.\]
\end{lem}
\begin{proof}
Since $L_i=\Delta_{g_i,f_i}+2\Rm_{g_i}*$ and $\|L_1-L_2\|_{C^{m-2}}\le C_m\|g_1-g_2\|_{C^m}$, we can compute by integration by parts that
    \begin{equation}
        \begin{split}
            \big|\langle h,L_1 h\rangle_{g_1}-\langle h,L_2 h\rangle_{g_2}\big|
            &=\left|\int_M \big((g_1-g_2)*h *L_1h+ h *(L_1-L_2)h\big)  \,dg_1\right|\\
            &=\bigg|\sum_{k=0,1,2,\; i,j=0,1}\int_M \nabla^k(g_1-g_2)* \nabla^i h *\nabla^j h \,dg_1\bigg|\\
            &\le C\|g_1-g_2\|_{C^2}\|h\|_{H^1_{g_1}}^2.
        \end{split}
    \end{equation}
  This implies the assertion, noting that $H^1_{g_1}$ is comparable to $H_{W,g_1}$-norm by \eqref{eq:H^1_H_W}.
\end{proof}

Moreover, for a fixed shrinker $\bar g$,
we denote by $V^*_{\Lie},\PP^*_{\Lie}$ the subspace and the projection of generic eigentensors with eigenvalues larger than $-\tfrac12\lambda_{-,\Lie}$, where $-\lambda_{-,\Lie}<0$ is the largest negative generic eigenvalue of $\bar g$. Then $V_{\Lie}^{\ge0}=V_{\Lie}^*$, and for any shrinker $\bar g'$ sufficiently close to $\bar g$, we have $\dim V^*_{\Lie}=\dim V'^*_{\Lie}$ (note $\dim V'^*_{\Lie}$ may be larger than $\dim V'^{\ge0}_{\Lie}$), and
\[0.9\|\PP'^*_{\Lie} h\|\le \|\PP^*_{\Lie} h\|\le 1.1\|\PP'^*_{\Lie} h\|,\]
for any $h\in\Sym^2(M)$, where $\PP'$ denotes the projections with respect to $\bar g'$.
Similarly, fix $\lambda\notin \Spec(L_{\bar g})$, then we have $\lambda\notin \Spec(L_{\bar g'})$, $\dim V_{>\lambda}=\dim V'_{>\lambda}$, and
\[0.9\|\PP'_{>\lambda} h\|\le \|\PP_{>\lambda} h\|\le 1.1\|\PP'_{>\lambda} h\|.\]

\begin{prop}\label{prop: unstable mode has certain portion, forward}
Let $(M,\bar{g},f)$ be a compact integrable shrinker, and let $\lambda\notin\Spec(L_{\bar g})$, $\lambda<0$. 
There is $C,T,\eps,\eps'>0$ such that for any shrinker $(M,\bar{g}_1,f_1)$ with $\|\bar g_1-\bar g\|_{C^{13}}\le \eps'$, the following holds:
Let $\{g_\tau\}_{\tau\in [0,T]}$ be a rescaled Ricci flow, with $h_0=g_0-\bar g_1$ satisfying
\begin{enumerate}
\item\label{assump:small_h} $\|h_0\|_{C^{13}}\le\eps'$,
\item\label{assump:first_k_stable_dominates} $\|\PP_{>\lambda}h_0\|^2_{H_W}\ge  \|\PP_{ <\lambda}h_0\|^2_{H_W}$,
    \item\label{assump:small_Lie} $\|\mathcal{P}_{\Lie}^*h_0\|^2 +\|\mathcal{P}^0h_0\|^2 \le \eps \|h_0\|^2$,
\item \label{assump:entropy}$\mu(g_T,1)\le\mu(\bar{g},1)$,
\end{enumerate}
where the projections and norms are with respect to $\bar g_1$.

Then there exist a diffeomorphism $\phi$ and a shrinker $(M,\bar g_2,f_2)$, with $h'=\phi^*(\xi_{1,T}^{-1})^*g_T-\bar g_2$, where $\xi_{1,\tau}$ is the flow of $\nabla f_1$ with $\xi_{1,0}=\id$, such that
\begin{enumerate}
\item\label{assert:decay} $\big(\|h'\|_{C^{11} }+\|h'\|_{H_W}\big)\le \frac{1}{100}\|h_0\|_{H_W} $,\item\label{assert:k_dominance} $\| \PP_{>\lambda}h'\|^2_{H_W}\ge  \| \PP_{<\lambda}h'\|^2_{H_W}$,
    \item\label{assert:Lie_0_small} $\|\mathcal{P}_{\Lie}^*h'\|^2+\|\mathcal{P}^0h'\|^2\le \eps \|h'\|^2$,
    \item\label{assert:phi-id} ${\|\exp_{\id}^{-1}\phi\|_{C^{11} }}+\|\bar g_2-\bar g_1\|_{C^{13} }\le C\|h_0\|_{C^{13}}$,
\end{enumerate}
where the projections and norms are with respect to $\bar g_2$.
\end{prop}

\begin{proof}
Assume $\lambda\in(\lambda_{k+1},\lambda_k)$ for some $k\in\mathbb N$.
Let $\chi_\tau$ be the HMHF starting from $\chi_0=\id$ between $g_\tau$ and $\xi_{1,\tau}^*\bar g_1$.
Let $h_\tau=(\xi_{1,\tau}^{-1})^*(\chi_\tau^{-1})^*g_\tau-\bar g_1$. 
Let $C>0$ denote a general constant that only depends on $\bar g$ and $\lambda$, and we choose $T,\eps,\eps'>0$ in the following way (and will adjust their values) so that they only depend on preceding constants:
\begin{enumerate}
\item Let $T$ be large so that $ e^{\Lambda T}>C$ and $e^{(\lambda_{k}-\lambda_{k+1})T}>C$, where $\Lambda(\lambda_1,\lambda_-)>0$ will be clear later and $\lambda_->0$ is from \eqref{eq:eigenvalues of L min}.
    \item Let $\varepsilon<\eps_0$ from Lemma \ref{l: cancel Lie by diffeomorphisms} (with $k=12$), and $2\eps e^{(\lambda_1^+-\lambda_{k+1})T}<\eps_0$ from Lemma \ref{l:entropy_forward}, and satisfies $C\eps<\Lambda$;
    \item 
    By Lemma \ref{l:shortHMHF}, there is $\eps'>0$ such that $\chi_\tau$ is diffeomorphism on $[0,T]$ and $\|h_\tau\|_{C^{13}}\le  \eps_1<\eps$, where $\eps_1$ is such that Lemma \ref{lem:general_V} is applicable, and Lemma \ref{lem:B5} is applicable for any $A\in(C^{-1}\eps,C\eps^{-1})$ on $[0,T]$, and for various invariant subspaces which will be clear below.
\end{enumerate}

 First, Assumption \eqref{assump:first_k_stable_dominates} and Lemma \ref{lem:B5} imply 
 \begin{equation}\label{eq:A_h}
     \| \PP_{>\lambda}h_\tau\|^2_{H_W}\ge   e^{(\lambda_{k}-\lambda_{k+1})\tau}\| \PP_{<\lambda}h_\tau\|^2_{H_W},
 \end{equation}
 for all $\tau\in[0,T]$. Together with Assumption \eqref{assump:small_Lie}, Lemma \ref{lem:B5} also implies
 \begin{equation}\label{claim:Lie_neutral_stay_small}
     \big\|\mathcal{P}_{\Lie}^*h_\tau\big\|^2 +\|\mathcal{P}^0h_\tau\|^2 \le 2\eps e^{(\lambda_1^+-\lambda_{k+1})\tau}\| \PP_{>\lambda}h_\tau\|^2,
\end{equation}
noting that these subspaces are all finite-dimensional, thus the $H_W$-norms are bounded by the $L^2_f$-norms.
Since $2\eps e^{(\lambda_1^+-\lambda_{k+1})T}<\eps_0$, seeing  $\mu(g_T,1)\le\mu(\bar{g},1)$ we can apply the Lemma \ref{l:entropy_forward} at $T$ to obtain
\[\|\PP^+h_T\|^2_{H_W} +\|\PP^0h_T\|^2_{H_W} \le C\|\mathcal{P}^-h_T\|^2_{H_W}.\]
Then Lemma \ref{lem:B5} implies that
there is $C_1(C,\lambda_-,\lambda_1)>0$ such that
for all $\tau\in[0,T-C_1]$, 
\begin{equation}\label{eq:stable_dom}
    \|\PP^+h_\tau\|^2_{H_W} +\|\PP^0h_\tau\|^2_{H_W} \le \tfrac{\lambda_-}{6\lambda_1+1}\|\mathcal{P}^-h_\tau\|^2_{H_W}.
\end{equation}
Since by Lemma \ref{lem:general_V} we have 
\begin{align*}
    \partial_\tau \|h_\tau\|^2_{H_W}
    &\le 3\lambda_1(\|\PP^0h_\tau\|^2_{H_W}+\|\PP^+h_\tau\|^2_{H_W})+(-{\lambda_-})\|\PP^-h_\tau\|^2_{H_W},
\end{align*}
it follows from \eqref{eq:stable_dom} that for $\tau\in[0,T-C_1]$ we have $\Lambda(\lambda_1,\lambda_-)>0$ such that
\[\partial_\tau \|h_\tau\|^2_{H_W}\le -\tfrac{\lambda_-}{2}\|\PP^-h_\tau\|^2_{H_W} \le -2\Lambda(\lambda_1,\lambda_-)\|h_\tau\|^2_{H_W},\]
integrating which we obtain $\|h_{T-C_1}\|^2_{H_W}\le e^{-2\Lambda (T-C_1)}\|h_0\|^2_{H_W}$.
Integrating $\partial_\tau \|h_\tau\|^2_{H_W}\le C\|h_\tau\|^2_{H_W}$ on $[T-C_1,T]$, we obtain $\|h_\tau\|_{H_W}\le C\|h_{T-C_1}\|_{H_W}$ for all $\tau\in [T-C_1,T]$. So by standard parabolic estimate we have
\begin{equation}\label{eq:decay_h}
 \|h_T\|^2_{C^{12}}\le C\|h_{T-C_1}\|^2_{H_W}\le C e^{-2\Lambda T}\|h_0\|^2_{H_W}.
\end{equation}

Next, by Lemma \ref{l: cancel Lie by diffeomorphisms} we can find a diffeomorphism $\psi$ with
\begin{equation}\label{eq:closeness_of_phi}
    \|\exp_{\id}^{-1}\psi\|_{C^{12}}\le C\big\|\mathcal P_{\Lie}^*h_T\big\| ,
\end{equation}
such that $h':=\psi^* (\bar{g}+ h_T) - \bar{g}$ satisfies $\|h'-h_T\|_{C^{11}}\le C\|\PP_{\Lie}^*h_T\|$, $\PP_{\Lie}^*h'=0$, 
and 
$$\big\|h'-(h_T-\mathcal{P}_{\Lie}^*h_T)\big\|_{C^{12}}\le C\eps\big\|\mathcal{P}_{\Lie}^* h_T\big\| ,$$
where we used $\|h_\tau\|_{C^{13}}\le\eps$.
In particular, this implies $\|\PP^0_{\ess} (h_T-h')\| \le C\eps \big\|\mathcal{P}_{\Lie}^* h_T\big\|$, which by
\eqref{claim:Lie_neutral_stay_small} implies
\begin{equation}\label{eq:small_neutral}
    \|\PP^0_{\ess} h'\|^2\le C\eps e^{(\lambda_1^+-\lambda_{k+1})T}\|h_T\|^2 <C^{-1}\|h_T\|^2.
\end{equation}
Since $\bar g_1$ is integrable, there exists an integrable shrinker $(M,\bar g_2,f_2)$ such that 
\begin{equation}\label{eq:gprime}
    \|\bar g_1-\bar g_2+\mathcal{P}_{\ess}^0h'\|_{C^{13}}\le \Psi(\mathcal{P}_{\ess}^0h')\|\mathcal{P}_{\ess}^0h'\|,
\end{equation}
where $\Psi(\mathcal{P}_{\ess}^0h')\to0$ as $\mathcal{P}_{\ess}^0h'\to0$, thus by assuming $\eps'$  sufficiently small, we can assume $\Psi(\mathcal{P}_{\ess}^0h')Ce^{(\lambda_1^+-\lambda_{k+1})T}<\tfrac{1}{4}C^{-1}$. So we obtain
\begin{equation}\label{eq:closeness_of_shrinkers}
    \|\bar g_1-\bar g_2+\mathcal{P}_{\ess}^0h'\|^2_{C^{13}}\le \tfrac{1}{4}C^{-1}\eps \| h_T\|^2.
\end{equation}

Let $h''=\psi^* (\bar{g}+ h_T) - \bar g_2=h'+\bar g_1-\bar g_2$. Note that for $V=V^0_{\ess},V^*_\Lie$ we have
\[\PP_V h''=\PP_V (\bar g_1-\bar g_2+\PP^0_\ess h').\]
So \eqref{eq:closeness_of_shrinkers} implies
\begin{equation}\label{eq:neutral_small}
    \|\PP_{\ess}^0h''\|^2+\big\|\PP^*_\Lie h''\big\|^2\le C\|\bar g_1-\bar g_2+\mathcal{P}_{\ess}^0h'\|_{C^{13}}^2\le \tfrac{1}{4}\eps\| h_T\|^2.
\end{equation} 
Moreover, \eqref{eq:gprime} implies 
\begin{equation}\label{eq:g-gprime}
    \|\bar g_1-\bar g_2\|_{C^{13}}=\|h''-h'\|_{C^{13}}\le C\|\PP^0_{\ess}h'\|.
\end{equation}
Combining this with $\|h'-h_T\|_{C^{11}}\le C\|\PP_{\Lie}^*h_T\|$ we obtain
\begin{equation}\label{eq:h''-hT}
\begin{split}
    \|h''-h_T\|_{C^{11}}\le C\big(\|\PP^0_{\ess}h'\|+\big\|\PP_{\Lie}^*h_T\big\|\big)\le C^{-1}\|h_T\|.
\end{split}
\end{equation}

Let $\phi=\big(\xi_{1,T}\circ\chi_T^{-1}\circ\xi_{1,T}^{-1}\big)\circ\psi$, then $h''
=\phi^* g_T-\bar g_2$. 
Now we verify that $\bar g_2,\phi,h''$ satisfy all assertions:
First, Assertion \eqref{assert:decay} follows by combining \eqref{eq:h''-hT}\eqref{eq:decay_h} and taking $T$ large enough. Next, we write $\PP',V'$ for the projections and invariant subspaces with respect to $\bar g_2$ to distinguish from those with respect to $\bar g_1$. 
By the observation before Proposition \ref{prop: unstable mode has certain portion, forward}, the smallness of $\PP^*_\Lie h'',\PP_{\ess}^{0}h''$ in \eqref{eq:neutral_small} implies the smallness of  ${\PP'}^*_\Lie h'',{\PP'}_{\ess}^{0}h''$, that is,
\[\big\|{\PP'}_{\ess}^{0}h''\big\|^2_{L^2_{f_2}}+\big\|{\PP'}^*_\Lie h''\big\|^2_{L^2_{f_2}}\le \eps\| h''\|^2_{L^2_{f_2}}.\]
So Assertion \eqref{assert:Lie_0_small} holds. 
Since by Lemma \ref{lem:pseudolocality}, there is $C_2(T)>0$ such that $\|\exp_{\id}^{-1}\chi_T\|_{C^{11}}\le C_1 \|h_0\|_{C^{13}}$, combining which with \eqref{eq:closeness_of_phi} implies 
$\|\exp_{\id}^{-1}\phi\|_{C^{11}}\le C_2 \|h_0\|_{C^{13}}$. This 
and \eqref{eq:g-gprime} imply Assertion \eqref{assert:phi-id}.
Lastly, for Assertion \eqref{assert:k_dominance},  
using \eqref{eq:h''-hT}, the finite dimension of $V_{>\lambda}$, and the triangle inequality we have
\[\| \PP_{>\lambda}(h''-h_T)\|_{H_W}+\| \PP_{<\lambda}(h''-h_T)\|_{H_W}\le C^{-1}\| h_T\|,\]
which together with \eqref{eq:A_h} implies
\begin{equation}\label{eq:h''A}
     \| \PP_{>\lambda}h''\|_{H_W}\ge 4 \| \PP_{<\lambda}h''\|_{H_W}.
 \end{equation}
By the observation before Proposition \ref{prop: unstable mode has certain portion, forward}, we have
\[\|\PP_{>\lambda}h''-\PP'_{>\lambda}h''\|\le 0.1\|h''\|,\]
and by triangle inequality this implies the same estimate holds also for $\|\PP_{<\lambda}h''-\PP'_{<\lambda}h''\|$.
Combining these with \eqref{eq:h''A} we get 
\[\| \PP'_{>\lambda}h''\|_{H_W}\ge 2 \| \PP'_{<\lambda}h''\|_{H_W}.\]
Finally, using Lemma \ref{lem:comparable_norms} we can convert this estimate to the $H_W$-norm with respect to $\bar g_2$, which implies
Assertion \eqref{assert:k_dominance}.
\end{proof}

\begin{theorem}[Theorem \ref{thm:RF_exp}, immortal flow]\label{p: g_t converge to g_Sigma exponentially_forward}
Let $(M,\bar{g},f)$ be a compact integrable shrinker, and 
$g_\tau$, $\tau\in[0,\infty)$, be a rescaled Ricci flow asymptotic to $\bar{g}$.
Then there exist $\delta>0$ and a diffeomorphism $\psi:M\to M$ such that 
\begin{equation*}
    \|(\xi^{-1}_\tau)^*g_\tau-\psi^*\bar g\|_{C^{2,\alpha}(\bar g)}\le C\,e^{-\delta \tau}.
\end{equation*}
\end{theorem}

\begin{proof} 
Let $\lambda_{\Lie}<0$ be the largest negative generic eigenvalue. Assume $\lambda_{\Lie}=\lambda_k$ for some  $k\ge 1$, and let $\lambda=\frac{\lambda_{k+1}+\lambda_k}{2}$, and let $C,T,\eps,\eps'$ be from Proposition \ref{prop: unstable mode has certain portion, forward}. 
By a time-shifting of the flow $g_\tau$, we may assume $\|g_0-\bar g\|_{C^{13}}\le \eps''\ll\eps'$, for some $\eps''>0$ to be chosen later.
\begin{claim}
  There exists a shrinker $\bar g_1$ with $\|\bar g_1-\bar g\|_{C^{13}}\le C^{-1}\eps'$ such that $h_1= g_0-\bar g_1$ satisfies all assumptions in Proposition \ref{prop: unstable mode has certain portion, forward} with $\lambda$.
\end{claim}

\begin{proof}
First, by the entropy monotonicity we have $\mu(g_T,1)\ge\mu(\bar g,1)$.
   Next, by Theorem \ref{t: slice theorem} there exists an integrable shrinker
$(M,\bar g_0)$ with $\|\bar g_0-\bar g\|_{C^{13}}\le C\eps''$ such that $h_0= g_{0}-\bar g_0$ satisfies $\|h_0\|_{C^{13}}\le C\eps''$, and
\[\mathcal{P}_{\Lie,\bar g_0}( g_{0}-\bar g_0)=\mathcal{P}^0_{\ess,\bar g_0}( g_{0}-\bar g_0)=0.\]
{Let $\LL_X\bar g_0$ be an eigentensor of $L_{\bar g_0}$ with eigenvalue $\lambda_{\Lie}$, where $X$ is orthogonal to all Killing fields, and $\|X\|=C_1\|h_0\|_{C^{13}}$ for some large $C_1(C)\gg C$ which will be determined later. 
Let $\phi_X$ be the diffeomorphism generated by $X$, then by Lemma \ref{l: key lemma for regularity} we have
\begin{equation}\label{eq:qua}
    \|\phi_X^*\bar g_0-\bar g_0-\LL_X \bar g_0\|_{C^{13}}\le C\|X\|^2.
\end{equation}
Let $h_1=g_{0}-\phi^*_X\bar g_0=h_0+(\bar g_0-\phi^*_X\bar g_0)$, this implies 
\begin{align*}
 \|\PP_{>\lambda}h_1+\LL_X \bar g_0\|_{H_W}&\le C\big(\|X\|^2+\|h_0\|_{C^{13}}\big)\le C\|h_0\|_{C^{13}},\\
 \|\PP_{<\lambda}h_1\|_{H_W}&\le C\big(\|X\|^2+\|h_0\|_{C^{13}}\big)\le C\|h_0\|_{C^{13}}.
\end{align*}
Seeing also $C\|h_0\|_{C^{13}}\le 0.1\|\LL_X\bar g_0\|_{H_W}$ by taking $C_1$ sufficiently large, this implies
$$\|\PP_{>\lambda}h_1\|_{H_W}\ge  \|\PP_{ <\lambda}h_1\|_{H_W}.$$
Note that \eqref{eq:qua} also implies $\|h_1\|_{C^{13}}\in (C^{-1}\|X\|,C\|X\|)$.
Let $\bar g_1=\phi_X^*\bar g_0$, then by the orthogonality of $V^*_{\Lie},V^0_\ess$ with $\LL_X \bar g_0$, \eqref{eq:qua} implies
\[\|\mathcal{P}^*_{\Lie,\bar g_1}h_1\|+\|\mathcal{P}^0_{\ess,\bar g_1}h_1\|\le C\|X\|^2.\]
Taking $\eps''$ sufficiently small, then $\|X\|\le (CC_1)^{-1}\eps$ and this implies
\[\|\mathcal{P}^*_{\Lie,\bar g_1}h_1\|+\|\mathcal{P}^0_{\ess,\bar g_1}h_1\|\le \eps \|h_1\|.\]
Lastly, we also have $\|h_1\|_{C^{13}}<\eps'$,
thus all assumptions in Proposition \ref{prop: unstable mode has certain portion, forward} hold.
} 
\end{proof}

By abuse of notation, we denote $(\xi_\tau^{-1})^*g_\tau$ by $g_\tau$.
Then, by the above claim, we can
apply Proposition \ref{prop: unstable mode has certain portion, forward} inductively at $\{kT\}_{k=0}^\infty$, and get a sequence of diffeomorphisms $\{\psi_k\}_{k=0}^\infty$ with $\psi_0=\id$, integrable shrinkers $\{\bar g_k\}_{k=1}^\infty$, and $\{h_{k+1}:=\psi^*_kg_{kT}-\bar g_{k+1}\}_{k=0}^\infty$, such that 
  \begin{enumerate}
    \item $(\|h_{k+1}\|_{C^{11}}+\|h_{k+1}\|_{H_W})\le\tfrac{1}{100}\|h_k\|_{H_W} $, and $\|h_{k+1}\|_{C^{11}}\le\|h_{k+1}\|$,
    \item $\|\exp^{-1}_{\psi_{k}}\psi_{k+1}\|_{C^{11}}+\|\bar g_{k+2}-\bar g_{k+1}\|_{C^{13}}\le C\|h_{k}\|_{C^{13}}$.
\end{enumerate}
Note that $\|\bar g_{k}-\bar g\|_{C^{13}}\le \eps'$ and thus Proposition \ref{prop: unstable mode has certain portion, forward} can be repeated for all $k$.
As in Theorem \ref{p: g_t converge to g_Sigma exponentially}, these imply $\|h_k\|_{H_W}$ decays exponentially, and $\psi_k$ and $\bar g_k$ are Cauchy sequences, thus converge to a diffeomorphism $\psi_\infty$ and an integrable shrinker $\bar g_\infty$. It follows that $\psi_\infty^*g_\tau$ converge exponentially to $\bar g_{\infty}$ as $\tau\to\infty$, and $\bar g_{\infty}$ is isometric to $\bar g$.
\end{proof}

\begin{remark}
    By Theorem \ref{p: g_t converge to g_Sigma exponentially} and \ref{p: g_t converge to g_Sigma exponentially_forward}, for any rescaled modified ancient or immortal Ricci flow $g_\tau$ asymptotic to the compact integrable shrinker $\bar g$, there exists a diffeomorphism $\psi$ such that $g_\tau\to(\psi^{-1})^*\bar g$. So by replacing $\bar g$ by $(\psi^{-1})^*\bar g$ we have $g_\tau\to\bar g$ as $\tau\to\mp\infty$.
\end{remark}

\begin{proof}[Proof of Theorem \ref{t:uniqueness_of_tangent_flow}]
  By Theorem \ref{thm:RF_exp}, the modified rescaled Ricci flow $g_\tau$ converges exponentially to $\bar g$. So the theorem follows from Theorem \ref{t:existence_theorem_HMHF}.
\end{proof}

\section{Classification and optimal convergence rate}\label{sec: uniqueness}

In this section, we apply the global gauge of HMHF to classify the ancient flows (Theorem \ref{t:uniqueness}) and obtain the sharp asymptotic behavior near singularities (Theorem \ref{thm:gap}) that are asymptotic to compact integrable shrinkers.

\subsection{Classification of ancient flows}
The main result in this subsection is Theorem \ref{t:existence_of_a}.
It states that any ancient Ricci flow exponentially converging to a compact shrinker (possibly non-integrable) is one of the flows in Theorem \ref{t:existence_theorem_intro}.

\begin{theorem}\label{t:existence_of_a}
    Let $(M,\bar g)$ be a compact shrinker, and $(M,g_\tau)$ be an ancient rescaled modified Ricci flow asymptotic to $(M,\bar{g})$. If there exists $\delta>0$ such that $\|g_\tau-\bar g\|_{C^{2,\alpha}}=O(e^{\delta\tau})$ as $\tau\to-\infty$, then there exist a modified HMHF $\chi_\tau:(M,g_\tau)\to(M,\bar{g})$, $\tau\in(-\infty,\tau_0]$ for some $\tau_0\le0$ which converges exponentially to $\id$, and a vector $\mathbf{a}\in \mathbb{R}^{I_{\ess}}$ such that 
\begin{equation}\label{eq:limit_phi_a}
(\chi_\tau^{-1})^*g_\tau-\bar{g}=S(\mathbf{a}).
\end{equation}
\end{theorem}

If $(M,\bar g)$ is integrable, then by Theorem \ref{p: g_t converge to g_Sigma exponentially} we see that any ancient rescaled modified Ricci flow asymptotic to $(M,\bar{g})$ converges exponentially to it. So Theorem \ref{t:uniqueness} is a corollary of Theorem \ref{t:existence_of_a}.

We fix the conditions in Theorem \ref{t:existence_of_a} throughout this subsection: Let $(M,\bar{g},f)$ be a compact shrinker, and $L=\Delta_f+2\Rm*$. Let $(M,g_\tau)$ be an ancient rescaled modified Ricci flow asymptotic to $(M,\bar{g})$ that converges exponentially to $(M,\bar g)$ as $\tau\to-\infty$.

\begin{proposition}\label{prop:diff.asymp}
Let $h^1_\tau$ and $h^2_\tau$ be ancient solutions to the modified RDTF Perturbation equation \eqref{eq:most important eq} with background $(M,\bar g,f)$. Suppose $\|h^1_\tau\| +\|h^2_\tau\|  =O(e^{\delta \tau})$ for some $
\delta>0$. Then there is an eigentensor $\varphi$ of $L$ with eigenvalue $\lambda>0$ such that 
\begin{equation}\label{eq:limit.eigenfunction}
    \lim_{\tau\to -\infty} e^{-\lambda \tau}(h^1_\tau-h^2_\tau)= \varphi,
\end{equation}
holds in $L^2_f$-norm, unless  $h^1_\tau \equiv h^2_\tau$.
\end{proposition}

\begin{proof}
Let $w_\tau=h^1_\tau-h^2_\tau$. Our theorem is analogue to \cite[Theorem 6.4]{CS20}, and we follow the proofs in \cite[Section 6.2]{CS20} to obtain the desired result. To this end, by using {the structure of $Q$ in \eqref{eq:Q(h)}} we have 
\begin{equation}
 |\langle Q(h^1_\tau)-Q(h^2_\tau),w_\tau\rangle | \leq Ce^{\delta\tau}\|w_\tau\|_{H_W}^2.
\end{equation}
Together with the following 
Lemma \ref{lem:diff.dominace}, we get
\begin{equation}\label{eq:error.exp}
 |\langle Q(h^1_\tau)-Q(h^2_\tau),w_\tau\rangle | \leq Ce^{\delta\tau}\|\PP^+w_\tau\|_{H_W}^2\leq Ce^{\delta\tau}\|\PP^+w_\tau\| ^2,
\end{equation}
which can substitute \cite[Lemma 6.5]{CS20}. Then, we can obtain \cite[Lemma 6.6]{CS20} by using the same proof. Also, given any symmetric 2-tensor $\varphi$, by using integration by parts we get
\begin{equation}\label{eq:IBP}
 |\langle Q(h^1_\tau)-Q(h^2_\tau),\varphi\rangle | \leq C(\|h^1_\tau\|_{C^2}+\|h^1_\tau\|_{C^2})\|\varphi\|_{H_W}\|w_\tau\|_{H_W}\leq C \|\varphi\|_{H_W}\|w_\tau\|_{H_W}.
\end{equation}
Thus, we can obtain \cite[Lemma 6.7]{CS20} as well. Now, by following the rest of arguments in \cite[Section 6.2]{CS20} with the same proofs, we can complete the proof.
\end{proof}

The following lemma shows the dominance of the unstable modes in $w_\tau$.

\begin{lemma}\label{lem:diff.dominace}
We recall $h^1_\tau,h^2_\tau$ in Proposition \ref{prop:diff.asymp}. Then $w_\tau=h^1_\tau-h^2_\tau$ satisfies
\begin{equation}
    \|\PP^{\le0}w_\tau\|_{H_W}=o(e^{\delta' \tau})\|\PP^+w_\tau\|_{H_W},
\end{equation}
for any $\delta' \in (0,\delta)$.
\end{lemma}

\begin{proof}
We suppose $w_\tau \not \equiv 0$. Remembering \eqref{eq:most important eq}, we can see that $w_\tau$ is a solution to a uniformly parabolic linear equation. Hence, we have
\begin{equation}\label{eq:analyticity_time_slice}
    w_{\tau'} \not \equiv 0,
\end{equation}
for every $\tau'\ll -1$. We recall $\lambda_I>0$ is the smallest positive eigenvalue. Using $\|h^i_\tau\|_{C^3}=O(e^{\delta\tau})$, we apply Lemma \ref{lem:general_V} so that we can get
\begin{align}\label{eq:mode_evol}
    &\tfrac{d}{d\tau}e^{-2\lambda_I\tau}\|\PP^+w_\tau\|_{H_W}^2\geq -Ce^{ (\delta-2\lambda_I) \tau}\|w_\tau\|_{H_W}^2,
    &\tfrac{d}{d\tau}\|\PP^{\le0}w_\tau\|_{H_W}^2\leq Ce^{\delta\tau}\|w_\tau\|_{H_W}^2.
\end{align}
Therefore, $a(\tau):=e^{\delta'\tau}\|\PP^+w_\tau\|_{H_W}^2-\|\PP^{\le0}w_\tau\|_{H_W}^2$ with $\delta'\in (0,\delta)$ satisfies
\begin{align*}
a'&\geq (\delta' +2\lambda_I-Ce^{(\delta -\delta')\tau})e^{\delta'\tau}\|\PP^+w_\tau\|_{H_W}^2 -Ce^{\delta\tau}\|\PP^{\leq 0}w_\tau\|_{H_W}^2.
\end{align*}
Hence, there is $\tau_0\ll -1$ depending on $\bar g$ such that $a'(\tau')> 0$ holds if $a(\tau')=0$ with $\tau'\leq \tau_0$. Here, we used the fact that 
$\|\PP^{\leq 0}w_{\tau'}\|_{H_W}=0$ and $a(\tau')=0$ imply \eqref{eq:analyticity_time_slice}. Towards contradiction, we assume there is $\tau_1\leq \tau_0$ such that $a(\tau_1)\leq 0$. Then, we have $a< 0$ for all $\tau<\tau_1$. Then, by \eqref{eq:analyticity_time_slice} we have $\|\PP^{\le0}w_\tau\|_{H_W}^2>0$ for all $\tau\leq \tau_1$. Now, feeding $a< 0$ back to \eqref{eq:mode_evol} yields
\begin{equation}
\tfrac{d}{d\tau}\log \|\PP^{\le0}w_\tau\|_{H_W}^2\leq Ce^{(\delta-\delta')\tau}.
\end{equation}
Integrating from $\tau_1$ to $-\infty$ gives us 
\begin{equation}
\lim_{\tau\to-\infty}\log \|\PP^{\le0}w_\tau\|_{H_W}^2\geq \log \|\PP^{\le0}w_{\tau_1}\|_{H_W}^2 -Ce^{(\delta-\delta')\tau_1}>-\infty.
\end{equation}
This contradicts $w_\tau\to 0$ as $\tau\to -\infty$. Therefore, we have $a(\tau)>0$ for all $\tau\leq \tau_0$. This completes the proof.
\end{proof}

Recall the map $S(\cdot)$ from Theorem \ref{t: existence theorem}, which maps any $\mathbf{a}\in \mathbb{R}^{I_{\ess}}$ to a rescaled modified RDFT $S(\mathbf{a})$, and we have $\|S(\mathbf{a})\|_{C^k}=O(e^{\lambda_+\tau})$ for each $k\in \mathbb N$.
Thus, 
for any $\mathbf{a}\in \mathbb{R}^{I_{\ess}}$ and an ancient modified HMHF $\chi_\tau$ between $g_\tau$ and $\bar{g}$ which converges exponentially to $\id$, we define 
\begin{equation}
w_{\chi,\mathbf{a}}:=(\chi^{-1}_\tau)^*g_\tau-\bar{g}-S(\mathbf{a}).
\end{equation}
By Proposition \ref{prop:diff.asymp}, if $w_{\chi,\mathbf{a}}\not \equiv 0$, there is an eigentensor $ h_{\chi,\mathbf{a}}$ with eigenvalue $\Lambda_{\chi,\mathbf{a}}>0$ such that
\begin{equation}\label{eq:definition of lambda a}
\lim_{\tau\to -\infty}e^{-\Lambda_{\chi,\mathbf{a}}\tau}w_{\chi,\mathbf{a}}= h_{\chi,\mathbf{a}} \not \equiv 0.
\end{equation}
 The following lemma shows that if the dominating eigentensors $ h_{\chi,\mathbf{a}}$ is not generic, then there is another $S(\mathbf{b})$ which is either closer to $h_\tau=(\chi^{-1}_\tau)^*g_\tau-\bar{g}$, or $ h_{\chi,\mathbf{b}}$ is generic. 

\begin{lemma}\label{l: limit is not only generic mode}
Suppose $w_{\chi,\mathbf{a}}\not \equiv 0$ and the limit eigentensor
    \begin{equation*}
    \lim_{\tau\to -\infty} e^{-\Lambda_{\chi,\mathbf{a}}\tau}w_{\chi,\mathbf{a}}= h_{\chi,\mathbf{a}}
    \end{equation*}
    is not generic. Then, there is $\mathbf{b}\in \mathbb{R}^I$ such that either $\Lambda_{\chi,\mathbf{b}} > \Lambda_{\chi,\mathbf{a}}$, or $\Lambda_{\chi,\mathbf{b}}= \Lambda_{\chi,\mathbf{a}}$ and
    \begin{equation}
    \lim_{\tau\to -\infty} e^{\Lambda_{\mathbf{a}}\tau}w_{\chi,\mathbf{b}}=\mathcal{P}_{\Lie}  h_{\chi,\mathbf{a}} \neq 0.
    \end{equation}
\end{lemma}

\begin{proof}
We recall $I',\lambda_j',\vec{ h}_j$ in Section \ref{sec: existence}. Since $\PP_{\ess} h_{\chi,\mathbf{a}}\neq 0$, we have $\lambda_j'=\Lambda_{\chi,\mathbf{a}}$ for some $j\leq I'$, and moreover $\mathcal{P}_{\ess}  h_{\chi,\mathbf{a}}=\vec a'_j \cdot\vec{ h}_j:=\mathbf{a}'\cdot \vec h$ for some $\vec a'_j\neq0$. Thus, by Theorem \ref{t: existence theorem} we have
\[\lim_{\tau\to -\infty} e^{-\Lambda_{\mathbf{a}}\tau}(S(\mathbf{a}+\mathbf{a}')-S(\mathbf{a}))=\mathbf{a}'\cdot\vec h.\]
So $\mathbf{b}=\mathbf{a}+\mathbf{a}'$ satisfies
$\lim_{\tau\to -\infty} e^{-\Lambda_{\mathbf{a}}\tau}w_{\chi,\mathbf{b}}=\mathcal{P}_{\Lie}  h_{\chi,\mathbf{a}}$. This finishes the proof if $\mathcal{P}_{\Lie}  h_{\chi,\mathbf{a}}\neq0$, otherwise, Proposition \ref{prop:diff.asymp} implies $\Lambda_{\chi,\mathbf b}>\lambda_{\chi,\mathbf a}$.
\end{proof}

If $ h_{\chi,\mathbf{a}}$ is generic, then the next lemma finds a new modified HMHF $\{\psi_\tau\}$ such that the new $w_{\psi,\mathbf a}$ has a faster exponential decay. 

\begin{lem}\label{l: change of base point}
    Let $\chi_\tau$, $\tau\in(-\infty,0]$, be a modified HMHF between $g_\tau$ and $\bar{g}$ {which converges exponentially to $\id$} and $\mathbf{a}\in\mathbb R^I$ such that $ \PP_{\ess} h_{\chi,\mathbf{a}}= 0$. Then there exists a modified HMHF $\psi_\tau$, $\tau\in(-\infty,0]$,  between $g_\tau$ and $\bar{g}$ {which converges exponentially to $\id$} such that
\begin{equation*}
        \lim_{t\to{-\infty}}e^{-\Lambda_{\chi,\mathbf{a}} \tau}w_{\psi,\mathbf{a}}=0.
    \end{equation*}
In particular, we have either $w_{\psi,\mathbf{a}}\equiv0$ or $\Lambda_{\psi,\mathbf a}>\lambda_{\chi,\mathbf a}$.
\end{lem}

\begin{proof}
By Lemma \ref{t: unstable Lie derivative implies unstable vector},
there exists an eigenvector $\bar X$ with eigenvalue $\lambda:=\Lambda_{\chi,\mathbf{a}}$ of the operator
$\mathfrak{L}_{\bar g}=\Delta_{\bar{g},f}+\tfrac12$ 
such that $ h_{\chi,\mathbf{a}}=\LL_{\bar X} \bar{g}$.
So by Theorem \ref{l: contraction mapping RDTF}, there exist $C>0$ and a solution $X _\tau$ for $t\in(-\infty,0]$ to the modified HMHF Perturbation equation \eqref{eq:HMHF perturbation sec 7_2} at $\chi_\tau$ (or equation \eqref{eq:HMHF perturbation sec 7} if $\bar g$ is positive Einstein), such that $\psi_\tau=\exp_{\id}(X_\tau+\exp_{\id}^{-1}\chi_\tau)$, $\tau\in(-\infty,0]$, is a modified HMHF between $(M,g_\tau)$ and $(M,\bar{g})$, and $X_\tau$ satisfies
\begin{equation}\label{eq:fast decay}
    X _\tau-e^{\lambda \tau}\bar X=o(e^{\lambda \tau}),
\end{equation}
where hereafter $o(F),O(F)$ denotes vectors or tensors whose $C^{2,\alpha}$-norms are arbitrarily smaller than or bounded by $F$ as $\tau\to-\infty$.
This implies $X_\tau=O(e^{\lambda \tau})$ and moreover $$(\psi_\tau^{-1})^*g_\tau-(\chi_\tau^{-1})^*g_\tau=-\LL_{X_\tau}g_\tau+o(e^{\lambda \tau})=-\LL_{X_\tau}\bar{g}+o(e^{\lambda \tau}),$$
which by \eqref{eq:fast decay} implies
\begin{equation*}
    \begin{split}
        (\psi_\tau^{-1})^*g_\tau-(\chi_\tau^{-1})^*g_\tau+e^{\lambda \tau }\LL_{\bar X} \bar{g}
        &= o(e^{\lambda \tau}).
    \end{split}
\end{equation*}
So using $w_{\chi,\mathbf{a}}=e^{\lambda \tau}\LL_{\bar X} \bar{g}+o(e^{\lambda \tau})$ we obtain
\begin{align*}
    \omega_{\psi,\mathbf{a}}&=(\psi_\tau^{-1})^*g_\tau-(\chi_\tau^{-1})^*g_\tau+\omega_{\chi,\mathbf{a}}=o(e^{\lambda \tau}).
\end{align*}
This implies the assertion
$\lim_{\tau\to{-\infty}}e^{\Lambda_{\chi,\mathbf{a}} \tau}w_{\psi,\mathbf{a}}=0$, and by Proposition \ref{prop:diff.asymp}
we have either $w_{\psi,\mathbf{a}}\equiv0$ or $\Lambda_{\psi,\mathbf a}>\lambda_{\chi,\mathbf a}$.
\end{proof}
 
Now we prove Theorem \ref{t:existence_of_a}:

\begin{proof}[Proof of Theorem \ref{t:existence_of_a}]
By Theorem \ref{t:existence_theorem_HMHF}, we have a modified HMHF $\chi_{0,\tau}$ between $g_\tau$ and $\bar g$ converging to $\id$ exponentially.
We consider $w_{\chi_0,\textbf{0}}=(\chi_{0,\tau}^{-1})^*g_\tau-\bar{g}-S(\textbf{0})$. If $w_{\chi_0,\textbf{0}} \neq 0$, then by Proposition \ref{prop:diff.asymp} we have $\Lambda_{\chi_0,\textbf{0}}$ from \eqref{eq:definition of lambda a}. By using Lemma \ref{l: limit is not only generic mode} and Lemma \ref{l: change of base point}, we can find a modified HMHF $\chi_{1,\tau}$ and a vector $\mathbf{a}_1 \in \mathbb{R}^{I_{\ess}}$ such that either $w_{\chi_1,\mathbf{a}_1}\equiv 0$, or $\Lambda_{\chi_1,\mathbf{a}_1}>\lambda_{\chi_0,\textbf{0}}$.
We can repeat this process and get an increasing sequence of eigenvalues $\{\Lambda_{\chi_i,\mathbf{a}_i}\}$. Since $\Lambda_{\chi_0,\mathbf{a}_0}\ge\lambda_+$ and there are at most $I$ distinct positive eigenvalues, after at most $I$-times we find $\chi_{k,\tau}$ and $\mathbf{a}_k$ such that $w_{\chi_k,\mathbf{a}_k}\equiv 0$. This proves the theorem.
\end{proof}

Now consider the space
\begin{align*}
    \mathcal{M}(\bar g)&:=\big\{\{g_\tau\}_{\tau\in(-\infty,0]}: g_\tau \textnormal{ is a rescaled Ricci flow exponentially converging to }\bar g\big\}.
\end{align*}
By Theorem \ref{t:existence_of_a} and
\eqref{eq:limit_phi_a}, we can define a map $\pi: \mathcal{M}(\bar g)\to \mathbb R^I$
which maps a flow $\{g_\tau\}\in \mathcal M(\bar g)$ to a vector $\mathbf a\in \mathbb R^I$.
In the rest of this subsection, we assume $\bar g$ is \textit{positive Einstein}, and establish that $\chi_\tau$ and $\mathbf a$ in \eqref{eq:limit_phi_a} are uniquely determined by $g_\tau$, thus the map $\pi$ is well-defined.
We also show $\pi$ is a bijection, so $\mathcal{M}(\bar g)$ can be parametrized by $\mathbb R^I$.

First, the following lemma shows that  $(\chi_{1,\tau}^{-1})^*g_\tau-(\chi_{2,\tau}^{-1})^*g_\tau$ is dominated by a generic unstable mode.

\begin{lem}\label{l: dominated by a generic unstable mode}
   Let $k\in \mathbb N$. For any two   HMHFs $\chi_{i,\tau}$, $i=1,2$, between $g_\tau$ and $\bar{g}$, which converge exponentially to $\id$ as $\tau \to -\infty$, there exist a generic unstable eigenvector $\bar Z$ of $\mathfrak L_{\bar g}$ with eigenvalue $\lambda>0$ 
   such that
    \begin{equation}
        \|(\chi_{1,\tau}^{-1})^*g_\tau-(\chi_{2,\tau}^{-1})^*g_\tau-e^{\lambda \tau}\LL_{\bar Z} \bar{g}\|_{C^k}\le o( e^{\lambda \tau}),
    \end{equation}
    unless $\chi_{1,\tau}\equiv \chi_{2,\tau}$.
\end{lem}

\begin{proof}
    Suppose $X_{i,\tau}=\exp_{\id}^{-1}\chi_{i,\tau}$, for $i=1,2$, and let $Z_\tau=X_{1,\tau}-X_{2,\tau}$. We have 
    \begin{equation}\label{eq:Lie dominates}
        \|(\chi_{1,\tau}^{-1})^*g_\tau-(\chi_{2,\tau}^{-1})^*g_\tau-\LL_{Z_\tau} g_\tau\|_{C^k}\le C\|Z_\tau\|^2_{C^{k+1}}.
    \end{equation}
Since each $X_{i,\tau}$ satisfies \eqref{eq:HMHF Perturbation}, it follows that 
    \begin{equation}
        \partial_\tau Z_\tau=\mathfrak L_\tau Z_\tau + Q(X_{1,\tau})-Q(X_{2,\tau}),
    \end{equation}
    where $\mathfrak L_\tau=\Delta_{\bar g}+\Ric_{\bar g}+E_\tau(\cdot,{g_\tau-\bar g})$.
    Since we have 
    \begin{equation}
        |Q_\tau(X_{1,\tau})-Q_\tau(X_{2,\tau})|\le C e^{\delta \tau}\|X_{1,\tau}-X_{2,\tau}\|_{C^2}=C e^{\delta \tau}\|Z_\tau\|_{C^2},
    \end{equation}
we can deduce
\begin{equation}
\|\partial_\tau Z_\tau-\mathfrak{L}_{\bar g}Z_\tau\|_{C^1}\leq Ce^{\delta \tau}\|Z_\tau\|_{C^3}
\end{equation}    
for some $\delta>0$ and $C\geq 1$. Since $\|Z_\tau\|_{C^3}\leq Ce^{\delta'\tau}$ for some $\delta'>0$, there exists an unstable eigenvector $\bar Z\neq 0$ of $\mathfrak L_{\bar g}$ with eigenvalue $\lambda>0$ such that
    \begin{equation}
    \lim_{\tau\to -\infty}   e^{-\lambda \tau} Z_\tau=\bar Z,
    \end{equation}
    smoothly as in the proof of Proposition \ref{prop:diff.asymp}. Hence, combining with \eqref{eq:Lie dominates} completes the proof of this lemma.
\end{proof}

\begin{lem}\label{lem:well-defined_map}
Let $\chi_{i,\tau}$, $i=1,2$, be two   HMHFs between $g_\tau$ and $\bar{g}$, which converge exponentially to $\id$ as $\tau \to -\infty$.
    Suppose for $\mathbf a_i\in \mathbb R^I$ that 
    \begin{equation}\label{eq:assumption of unique}
        (\chi_{i,\tau}^{-1})^*g_\tau-\bar{g}=S(\mathbf{a}_i).
    \end{equation}
    Then we have
\begin{equation}
    \chi_{1,\tau}=\chi_{2,\tau}\quad\textnormal{and}\quad \mathbf{a}_1=\mathbf{a}_2.
\end{equation}
\end{lem}

\begin{proof}
    By \eqref{eq:assumption of unique} we have 
    \begin{equation}\label{eq:difference}
        (\chi_{1,\tau}^{-1})^*g_\tau-(\chi_{2,\tau}^{-1})^*g_\tau=S(\mathbf{a}_1)-S(\mathbf{a}_2).
    \end{equation}
   Note that $\mathbf{a}_1=\mathbf{a}_2$ if and only if $S(\mathbf{a}_1)=S(\mathbf{a}_2)$ by Theorem \ref{t: existence theorem}. Moreover, if $\mathbf{a}_1\neq\mathbf{a}_2$, then by the property of $S$ there is an essential unstable eigentensor $h_1$ of eigenvalue $\mu_1>0$ such that
    \begin{equation}
        S(\mathbf{a}_1)-S(\mathbf{a}_2)=e^{\mu_1 \tau}h_1+o(e^{\mu_1 \tau}).
    \end{equation}
    By Lemma \ref{l: dominated by a generic unstable mode}, if $\chi_{1,\tau}\not\equiv\chi_{2,\tau}$, there is a generic unstable eigentensor $h_2$ of eigenvalue $\mu_2>0$ such that 
    \begin{equation}
        (\chi_{1,\tau}^{-1})^*g_\tau-(\chi_{2,\tau}^{-1})^*g_\tau=e^{\mu_2 \tau}h_2+o(e^{\mu_2 \tau}).
    \end{equation}
    Since $\langle h_1,h_2\rangle=0$, the equality \eqref{eq:difference} implies $\chi_{1,\tau}=\chi_{2,\tau}$ and $S(\mathbf{a}_1)=S(\mathbf{a}_2)$.
\end{proof}

So by Lemma \ref{lem:well-defined_map} we see $\pi: \mathcal{M}(\bar g)\to \mathbb R^I$ is well-defined.
The following lemma shows that $S(\mathbf a)$ uniquely determines the ancient   Ricci flow and the HMHF.

\begin{lemma}\label{lem:injective}
    Suppose that two rescaled   Ricci flows $g_{1,\tau}$, $g_{2,\tau}$ satisfy $((\chi_{i,\tau})^{-1})^*g_{i,\tau} = \bar{g} + S(\mathbf a)$ for some $\mathbf{a}\in \R^I$, where $\chi_{i,\tau}$ is a {HMHF} between $g_{i,\tau}$ and $\bar{g}$ satisfying $\chi_{i,\tau}\to\id$ as $\tau\to-\infty$, for each $i=1,2$. Then
    \begin{equation}
        g_{1,\tau}=g_{2,\tau}\quad\textnormal{and}\quad \chi_{1,\tau}=\chi_{2,\tau}.
    \end{equation}
\end{lemma}

\begin{proof}    
Note that
    \begin{equation}
{\partial_\tau\chi_{i,\tau}=\Delta_{g_{i,\tau},\bar{g}}\chi_{i,\tau}=\Delta_{\bar{g} + S(\mathbf a),\bar{g}}\id\circ \chi_{i,\tau}.}
    \end{equation}
By the uniqueness of the initial value problem of the ODE, we have $\chi_{2,\tau}=\chi_{1,\tau}\circ \chi_{1,0}^{-1}\circ \chi_{2,0}$. Then it follows by the assumption that $\chi_{i,\tau}\to\id$ that $\chi_{1,0}= \chi_{2,0}$ and $\chi_{1,\tau}= \chi_{2,\tau}$.
\end{proof}

It is clear that $\pi:\mathcal{M}(\bar g)\to \mathbb R^I$ is surjective, and by Lemma \ref{lem:injective} it is injective. Therefore, we can conclude:
\begin{cor}
Let $\bar g$ be a positive Einstein manifold. Then \eqref{eq:limit_phi_a} gives a bijection between $\mathbb R^I$ and the space of all rescaled Ricci flows exponentially converging to $\bar g$.
\end{cor}

\subsection{Rate of convergence to singularity}
In this subsection, we study immortal rescaled Ricci flows asymptotic to a compact shrinkers, and prove Theorem \ref{thm:gap} and \ref{thm:Carleman}.
These flows arise from rescalings at singularities.

We first prove Theorem \ref{thm:gap}.

\begin{proof}[Proof of Theorem \ref{thm:gap}]
Denote by $\Lambda_{\ess}<0$ the largest negative essential eigenvalue.
Let $C>0$ denote a general constant.
By abuse of notation, let $g_\tau$ be the rescaled modified Ricci flow of $g_t$. 
    By Theorem \ref{t:existence_theorem_HMHF}, there are $\delta>0$ and a modified HMHF $\chi_{0,\tau}$ between $g_\tau$ and $\bar g$ such that $ h_{0,\tau}=(\chi^{-1}_{0,\tau})^*g_\tau-\bar g$ satisfies 
    \[\|h_{0,\tau}\|_{C^k}\le  Ce^{-\delta\tau}.\]
    Then by the similar argument as in the ancient flow case in Proposition \ref{prop:diff.asymp}, if $h_{0,\tau}\not\equiv0$, we can find a eigentensor $\bar\varphi_0$ with eigenvalue $\Lambda_0<0$ such that 
    \begin{equation*}
    \lim_{\tau\to \infty} e^{-\Lambda_0\tau}h_{0,\tau}=\bar\varphi_0,
\end{equation*} 
If the eigentensor $\bar\varphi_0$ contains an essential component, then $\Lambda_0\le \Lambda_{\ess}$, which proves \eqref{eq:gap} in the theorem.
So we may assume $\bar\varphi_0$ is a generic eigentensor, then by the same argument in Lemma \ref{l: change of base point} using Theorem \ref{l: contraction mapping RDTF}, we can find another HMHF $\chi_{1,\tau}$ between $g_\tau$ and $\bar g$ such that $h_{1,\tau}=(\chi^{-1}_{1,\tau})^*g_\tau-\bar g$ satisfies 
    \[\lim_{\tau\to \infty} e^{-\Lambda_0\tau}h_{1,\tau}=0.\]
So there exist an eigenvalue $\Lambda_1<\Lambda_0$ and an eigentensor $\bar\varphi_1$ with eigenvalue $\Lambda_1$ such that 
\begin{equation*}
    \lim_{\tau\to \infty} e^{-\Lambda_1\tau}h_{1,\tau}=\bar\varphi_1.
\end{equation*}
Repeating this until $\Lambda_i\le \Lambda_{\ess}$. This proves \eqref{eq:gap} in the theorem.

For positive Einstein manifold, the assertion \eqref{eq:gap_2} follows from Lemma \ref{l:same_decay_rate} below.
\end{proof}

\begin{lem}\label{l:same_decay_rate}
    Let $(M,\bar g,f)$ be a positive Einstein manifold, and $\bar g_t=(-t)\bar g$ be the associated Ricci flow. Let $g_t$
    be a Ricci flow which develops a singularity modelled on $\bar g$ as $t\to0$.
    Suppose there is a   HMHF $\chi_\tau$ between the rescaled   Ricci flow $\widehat{g}_\tau:=e^{\tau}g_{-e^{-\tau}}$ and $\bar g$ with $\chi_\tau\to\id$ as $\tau\to\infty$, such that
    $$\|(\chi^{-1}_\tau)^*\widehat g_\tau-\bar g\|_{C^{2,\alpha}(\bar g)}\le  C_0 e^{-\lambda\tau},$$ 
    for some $\lambda>0$ and $C_0>0$.
    Then there exists $C>0$ such that 
    \[\|\widehat{g}_\tau-\bar g\|_{C^{2,\alpha }(\bar g)}\le Ce^{-\lambda\tau}.\]
    Equivalently, the unrescaled   Ricci flow $g_{t}$ satisfies 
    $$\|g_{t}-(-t)\bar g \|_{C^{2,\alpha }(\bar g)}\le C(-t)^{1+\lambda}.$$
\end{lem}

\begin{proof}
Let $C>0$ denote a general constant.
 Assume $\widetilde{g}_\tau=(\chi^{-1}_\tau)^*\widehat{g}_\tau$ where $\chi_\tau$ is a HMHF between the rescaled Ricci flows $\widehat{g}_\tau$ and $\bar g$ with $\chi_\tau\to\id$ as $\tau\to\infty$. Then 
\begin{equation}\label{eq:flow of diffeomorhisms under ODE}
\partial_\tau\chi_\tau=\Delta_{\widetilde{g}_\tau,\bar g}i.d.\circ\chi_\tau.
\end{equation}
So it follows by the assumption, $\bar g=\bar g$, and integrating from $-\infty$ to $\tau$ that 
$$\|\chi_\tau-\id\|_{C^{3,\alpha}(\bar g)}\le  C e^{-\lambda\tau}.$$
The assumption also implies $\|(\chi^{-1}_\tau)^*\widehat{g}_\tau-\bar g\|_{C^{2,\alpha}(\bar g)}\le  C e^{-\lambda\tau}$.
Therefore, we have
\[\|\widehat{g}_\tau -\bar g\|_{C^{2,\alpha}(\bar g)}\le\|\widehat{g}_\tau-(\chi^{-1}_\tau)^*\widehat{g}_\tau\|_{C^{2,\alpha}(\bar g)} +\|(\chi^{-1}_\tau)^*\widehat{g}_\tau-\bar g\|_{C^{2,\alpha}(\bar g)}\le  C e^{-\lambda\tau},\]
which proves the lemma by using $t=-e^{-\tau}$ and the definition of rescaled flows.
\end{proof}

Lastly, we prove Theorem \ref{thm:Carleman}, which works for all compact shrinkers and does not need the integrability condition.

\begin{proof}[Proof of Theorem \ref{thm:Carleman}]
By abuse of notation we denote the rescaled modified Ricci flows of $g_t$ by $g_\tau$.
  By Lemma \ref{l: existence to time-dependent, inhomogeneous equation-2}, we can find a modified HMHF $\chi_\tau$ between $g_\tau$ and $\bar g_\tau$ such that for any $\lambda>0$ there is $C>0$ such that $h_{\tau}=(\chi^{-1}_{\tau})^*g_\tau-\bar g$ satisfies
  \[\|h_{\tau}\|_{C^{2,\alpha}}\le C e^{-\lambda\tau}.\]
  Then by applying \cite[Proposition 4.1]{choi2024thom}, it follows that $h_\tau\equiv0$, and thus $(\chi^{-1}_{\tau})^*g_\tau=\bar g$.
  If $f=0$ and $\bar g$ is positive Einstein, then $\chi_\tau$ is a HMHF, and 
  by \eqref{eq:flow of diffeomorhisms under ODE} this implies $\partial_\tau\chi_\tau\equiv0$, $\chi_\tau\equiv\id$ and $g_\tau\equiv\bar g_\tau$.
\end{proof}

\appendix

\section{Second variation of the entropy at a shrinker}\label{subsec:stability}
In this section, we give the second variation formula of the $\mu$-entropy at a shrinker, and several equivalent definitions of linear stability.
Let $(M,g)$ be a compact Riemannian manifold.
In \cite{Pel1}, Perelman defined the $\mathcal W$-functional
\[\mathcal{W}(g,f,\tau)=\tfrac{1}{(4\pi\tau)^{n/2}}\int_M\big[\tau(|\nabla f|^2+R_g)+f-n\big]\,e^{-f}\textnormal{dvol},\]
and the $\mu$-entropy
\[\mu(g,\tau)=\inf\bigg\{\mathcal{W}(g,f,\tau)\bigg| f\in C^\infty(M),\tfrac{1}{(4\pi\tau)^{n/2}}\int_Me^{-f}\textnormal{dvol}=1\bigg\},\]
and the $\nu$-entropy
\[\nu(g)=\inf\{\mu(g,\tau)\,|\,\tau>0\}.\]
A shrinking soliton $\Ric+\nabla^2 f=\tfrac{1}{2\tau}g$ is a critical point of the $\nu$-entropy, and $\nu(g)=\mathcal W(g,f,\tau)$. 
Cao-Hamilton-Ilmanen \cite[Theorem 2.1]{cao2004gaussian} computed the second variation of the $\nu$-entropy at positive Einstein manifolds, and Cao-Zhu \cite{cao2012second} extended it to shrinking
solitons:
\begin{equation}
    D^2\nu(g)[h,h]=\frac{\tau}{(4\pi\tau)^{n/2}}\int_M\tfrac12 \langle h, N h\rangle\,e^{-f}\textnormal{dvol}_g,
\end{equation}
for all symmetric 2-tensor $h\in S^2M$, where
\[
N h := \Delta_f h + 2\Rm(h, \cdot) + 2\operatorname{div}_f^* \operatorname{div}_f h - 2\nabla^2 v_h -  \frac{\int_M \langle \Ric, h \rangle e^{-f}}{\int_M Re^{-f}}\Ric,
\]
and $v_h=Df-\tfrac12\tr h$ is the unique solution to \[
(\Delta_f+\tfrac{1}{2\tau}) v_h = -\tfrac12\operatorname{div}_f \operatorname{div}_f h,
\]
and the operator
$N$ satisfies
\begin{equation}\label{eq:N_second}
N h=
    \begin{cases}
        \Delta_f h+2\Rm*h-\frac{2\int_M\langle\Ric,h\rangle\,e^{-f}\,\textnormal{dvol}}{\int_M R\,e^{-f}\,\textnormal{dvol}}\,\Ric,\qquad &h\in\Ker\,\Di_f;\\
        0,\qquad & h\in \Ima\,\Di^*. 
    \end{cases}
\end{equation}

A shrinking soliton $\Ric+\nabla^2 f=\tfrac{1}{2}g$ is a critical point of the $\mu(\cdot,1)$-entropy with a fixed $\tau=1$, and $\mu(g)=\mathcal W(g,f,1)$. 
By computation, we obtain that the second variation of the $\mu(\cdot,1)$-entropy at a shrinking soliton is given by
\begin{align*}
    D^2\mu(g,1)[h,h]
    &=(4\pi)^{-n/2}\int_M\tfrac12\langle h,\widetilde N h\rangle e^{-f}\,\textnormal{dvol},
\end{align*}
for all symmetric 2-tensor $h\in S^2M$, where 
\[\widetilde N h=\Delta_f h+2\Rm*h+2\Di^*\circ\Di_f h-2\nabla^2v_h,\]
and $v_h=Df-\tfrac12\tr h$ is the unique solution to
\begin{equation}\label{eq:v_h}
    (\Delta_f +\tfrac12) v_h=-\tfrac{1}{2}\Di_f\circ \Di_fh,
\end{equation}
and the operator $\widetilde N$ satisfies
\begin{equation}\label{eq:second derivative}
\widetilde N h=
    \begin{cases}
        \Delta_f h+2\Rm*h=:L ,\qquad &h\in\Ker\,\Di_f;\\
        0,\qquad & h\in \Ima\,\Di^*. 
    \end{cases}
\end{equation}

\begin{defn}
    We say a compact Ricci shrinking soliton $(M,g,f)$ with $\Ric+\nabla^2 f=\tfrac{1}{2}g$ is \textit{linearly stable} if one of the following equivalent conditions is true:
\begin{enumerate}
\item $D^2\nu(g)[h,h]\le0$ for all $h\in S^2M$.
        \item $\langle N h,h\rangle\le0$ for all $h\in S^2M$;
        \item $\langle L h,h\rangle\le0$  for all $h\in S^2M$ such that $\Di_{f}(h)=0$ and $\int_{M}\langle 
    \Ric_g,h\rangle\,e^{-f}\textnormal{dvol}_g=0$.
    \end{enumerate}
The equivalence follows from \eqref{eq:N_second}\eqref{eq:second derivative}, and the observation $N (\Ric)=0$ and $L( \Ric)=\Ric$. 
\end{defn}

\section{Modified rescaled Ricci flows}\label{sec:RRF}

In this section, we recall some basic facts about Ricci-DeTurck flows, Ricci flows, and rescaled flows. Moreover, for a given shrinker, we define the modified Ricci flows and Ricci-DeTurck flows.

First, let $(M,g_t)$, $t\in(t_0,0)$, be a Ricci flow satisfying
\[\frac{\partial}{\partial t}g_t=-2\Ric(g_t).\]
We introduce a dilated time variable $\tau:=-\log(-t)\in(-\log(-t_0),\infty)$, equivalently $t=-e^{-\tau}$, 
Then the flow $\widetilde g_\tau=e^{\tau}g_{-e^{-\tau}}$
satisfies the \textbf{rescaled Ricci flow equation}:
\begin{equation}\label{eq:Einstein stationary solution}
    \frac{\partial}{\partial \tau}\widetilde g_\tau=-2\Ric(\widetilde g_\tau)+ \widetilde g_\tau.
\end{equation}
Note that positive Einstein manifolds $\Ric=\tfrac12 g$ are stationary solutions to \eqref{eq:Einstein stationary solution}.

Next, let $g_t $ and $\bar{g}_t$ be two Ricci flows.
Let $\chi_t$ be a HMHF by diffeomorphisms between $g_t $ and $\bar{g}_t$, 
then $\widetilde g_t =(\chi_t^{-1})^*(g_t )$ is \textbf{Ricci Deturck flow (RDTF) with background $\bar{g}_t$}, which satisfies
\begin{equation}\label{eq:RDTF equation}
    \partial_t \td g_t = - 2 \Ric_{\td g_t} -  \mathcal{L}_{\pt\chi_t\circ\chi_t^{-1} } \td g_t.
\end{equation}
Conversely, if $\tilde{g}_t$ is a RDTF with background $\bar{g}_t$, and $\{\chi_t\}$ is a family of diffeomorphisms $\chi_t$ by the following ODE
$$\partial_t\chi_t=\Delta_{\tilde{g}_t,\bar{g}_t}\id\circ\chi_t.$$
Then $g_t =\chi_t^*\tilde{g}_t$ satisfies the Ricci flow equation.

Then $h_t=(\chi_t^{-1})^*g_t -\bar{g}_t$
satisfies the \textbf{Ricci Deturck flow Perturbation equation} 
\begin{equation}\label{eq:RDTFP}
\partial_th=\Delta_{L,t}h+Q(h),
\end{equation}
where the Lichnerowizs Laplacian is 
$\Delta_{L,t}h=\Delta_{\bar g_t } h+2\Rm_{\bar g_t }*h-2\Ric_{\bar g_t }(h)$, 
and the quadratic term is, see e.g. \cite[Appendix A]{Bamler_Kleiner_2022},
\begin{equation}\label{eq:Q(h)}
    Q(h)=h*\nabla^2 h+\nabla h*\nabla h+h*\nabla h+h*h,
\end{equation}
where the derivatives are taken with respect to $\bar g_t $.
The linearization of \eqref{eq:RDTFP} is the \textbf{Linearized Ricci DeTurck flow equation}
\begin{equation}\label{eq:linear-unrescaled-perturb}
    \partial_th={\Delta_{L,t}} h.
\end{equation}
Let $\widehat{g}_{\tau}:=e^{\tau}(\chi_{-e^{-\tau}})_*g_{-e^{-\tau}}$ be the \textbf{rescaled Ricci DeTurck flow}, and let
\begin{align}
 \widehat{h}_{\tau}=\widehat{g}_{\tau}-e^{\tau}\bar{g}_{-e^{-\tau}}=e^{\tau}h_{-e^{-\tau}} .
\end{align}
Then $\widehat{h}_{\tau}$ satisfies the \textbf{rescaled RDTF perturbation equation} with background $\bar g_t$,
\begin{equation}\label{eq:rescaled nonlinear}
\begin{split}
\partial_{\tau}\widehat{h}
&=\widehat{h}+\Delta_{L,\tau}\widehat{h}+{Q}(\widehat{h}),
\end{split}
\end{equation}
where $\Delta_{L,\tau}\widehat{h}$ and $Q(\widehat{h})$ are with respect to the rescaled Ricci flow $e^{\tau}\bar{g}_{-e^{-\tau}}$.

Next, let $(M,\bar g,X)$ be a compact Ricci shrinker, where $X$ is a vector field such that
\[\Ric(\bar g) +\tfrac12\LL_X\bar g=\tfrac12 \bar g,\]
and let $\xi_t$ be a 1-parameter family of diffeomorphisms generated by $(-t)^{-1}X$ with $\xi_{-1}=\id$. Then $\bar{g}_t:=\frac{1}{(-t)}\xi_{t}^*\bar{g}$, $t\in(-\infty,0)$ is the Ricci flow associated to the shrinker.
{Let $g_\tau$ be a Ricci flow, then $\widetilde g_{\tau}=e^{\tau}(\xi_{-e^{-\tau}}^{-1})^*g_{-e^{-\tau}}$ 
satisfies the \textbf{rescaled and modified Ricci flow equation} with respect to $\bar g$:
\begin{equation}\label{eq:shrinker stationary solution}
    \frac{\partial}{\partial_\tau}\widetilde g_\tau=-2\Ric(\widetilde g_\tau)-\LL_X\widetilde g_\tau+ \widetilde g_\tau.
\end{equation}
In particular, $\bar g$ is a stationary solution to \eqref{eq:shrinker stationary solution}.}
For any $h_t$ which solves the RDTF Perturbation equation \eqref{eq:rescaled nonlinear} with background $\bar g_t$, let 
$$\bar{h}_t=\xi_{t*}h_t=(\xi_{t}^{-1})^*h_t=(\xi_{t}^{-1})^*((\chi_t^{-1})^*g_t )-\tfrac{1}{(-t)}\bar{g}.$$
Then $\bar h_t$ satisfies the \textbf{modified RDTF Perturbation equation} with background $\bar g$,
\begin{equation*}
\begin{split}
\partial_t\bar{h}_t&=\partial_t((\xi_{t}^{-1})^*h_t)=(\xi_{t}^{-1})^*(\partial_th_t)+\mathcal{L}_{\xi_{(1+t)*}\circ\partial_t\xi^{-1}_{t}}(\xi_{t}^{-1})^*h_t\\
&=(\xi_{t}^{-1})^*(\partial_th_t)-\mathcal{L}_{\partial\xi_{t}\circ \xi_{t}^{-1}}(\xi_{t}^{-1})^*h_t\\
&=(\xi_{t}^{-1})^*(\partial_th_t)-\mathcal{L}_X\bar{h}_t\\
&=\Delta_{L,\bar{g}}\bar{h}_t+Q_{\bar{g}}(\bar{h}_t)-\mathcal{L}_X\bar{h}_t.
\end{split}
\end{equation*}
where we used the formula $\partial_t(\xi_t^*h)=\mathcal{L}_{\xi^{-1}_{t*}X}\xi_t^*h=\xi_t^*(\mathcal{L}_X h)$ for any flow $\xi_t$ of vector field $X$, and also the observation $\chi_t\circ \chi_t^{-1}=i.d.$ which implies 
$(\chi_t^{-1})^*\circ \partial_t\chi_t^{-1}+\partial_t\chi_t\circ \chi_t^{-1}=0$.
Moreover, $\widehat{h}_\tau:=e^{\tau}\bar h_{-e^{-\tau}}$ satisfies 
\begin{equation*}
\begin{split}
\partial_\tau\widehat{h}_\tau
&=\Delta_{L,\bar{g}}\widehat{h}_\tau+\widehat{h}_\tau+Q_{\bar{g}}(\widehat{h}_\tau)-\mathcal{L}_X\widehat{h}_\tau,
\end{split}
\end{equation*}
of which the linearization is
\begin{equation*}    \partial_{\tau}\widehat{h}=\Delta_{\bar{g}}\widehat{h}+2\Rm_{\bar{g}}*\widehat{h}-2\Ric_{\bar{g}}(\widehat{h})+\widehat{h}-\mathcal{L}_{X}\widehat{h}.
\end{equation*}
Now assume the shrinker is gradient, which means $X=\nabla f$ for some smooth function $f$, then using $\mathcal{L}_{\nabla f}h=\nabla_{\nabla f}h+2\nabla^2f*h$ and $\Ric+\nabla^2f=\tfrac12 \bar g$, we have
\begin{equation*}
    \mathcal{L}_{\nabla f}h+2\Ric(h)-h=\nabla_{\nabla f}h.
\end{equation*}
So for a gradient shrinker, $\widehat{h}_\tau$ satisfies the 
\textbf{rescaled modified RDTF Perturbation equation} with background $\bar g$:
\begin{equation}\label{eq:most important eq}
    \partial_{\tau}\widehat{h}=\Delta_{\bar{g},f}\widehat{h}+2\Rm_{\bar{g}}*\widehat{h}+Q_{\bar{g}}(\widehat{h})=:L\widehat{h}+Q_{\bar{g}}(\widehat{h}),
\end{equation}
where $Q_{\bar{g}}(\widehat{h})$ satisfies \eqref{eq:Q(h)} and the derivatives are taken with respect to $\bar g$. The flow $\bar g+\widehat{h}$ is called a \textbf{rescaled modified RDTF} with background $\bar g$.
Note by \eqref{eq:second derivative} that $L$ is equal to the second variation operator of the $\mu(\cdot,1)$-entropy restricted on divergence-free tensors.

\section{Harmonic map heat flow perturbation}\label{sec:HMHF}
In this section, we recall some facts of the harmonic map heat flow (HMHF). We introduce the HMHF Perturbation equation, and show that the \textit{difference} of two HMHFs satisfies the equation. For a given Ricci shrinker, we define the modified HMHF and its perturbation. This modification enables us to study HMHFs between a Ricci flow and the shrinker’s Ricci flow while fixing the background metric to be the shrinker, facilitating spectral analysis.
\subsection{Linearization of the map laplacian}
The Laplacian of a smooth map $f:(M,g)\to(N,h)$, under  two local coordinates $\{x^i\}$ and $\{x^{\gamma}\}$, respectively, is a vector field defined by $\Delta_{g,h}f=(\Delta_{g,h}f^{\gamma})\frac{\partial}{\partial x^{\gamma}}$. Denote the Christoffel symbols for $g_t$ and $h_t$ respectively by $\Gamma$ and $\hat\Gamma$, then we have
\begin{equation}\label{eq:Laplacian of map}
 \begin{split}
 (\Delta_{g,h}f)^{\gamma}&=g^{ij}\left(\frac{\partial^2f^{\gamma}}{\partial x^i\partial x^j}-\Gamma^k_{ij}\frac{\partial f^{\gamma}}{\partial x^k}+\hat\Gamma^{\gamma}_{\alpha\beta}(f(x))\frac{\partial f^{\alpha}}{\partial x^i}\frac{\partial f^{\beta}}{\partial x^j}\right)\\
 &=(\Delta_gf)^{\gamma}+g^{ij}\left(\hat\Gamma^{\gamma}_{\alpha\beta}(f(x))\frac{\partial f^{\alpha}}{\partial x^i}\frac{\partial f^{\beta}}{\partial x^j}\right).
 \end{split}
\end{equation}

Let $(M,g)$ be a Riemannian manifold, we consider a family of smooth maps $f^s:(M,g)\to(M,g)$, which is parametrized by $s\in[-\epsilon,\epsilon]$, with $f^0=\id$.
Let $X=\ds f$.
Let $\{x^i\}_{i=1}^n$ be a local coordinate, and assume the maps $f_s$ have the component $\{f^{\gamma}\}_{\gamma=1}^n$, where we omit the dependence of $f$ on $s$.
Then
$\ds f^{\gamma}=X^{\gamma}$, $f^{\gamma}|_{s=0}=x^{\gamma}$, and $\left.\frac{\partial f^{\gamma}}{\partial x^k}\right|_{s=0}=\delta^{\gamma}_k$.
Taking derivative on \eqref{eq:Laplacian of map} with respect to $s$ and restricting on $s=0$, we obtain
\begin{equation}\label{eq:X equation RHS}\begin{split}
\ds(\Delta_{g,g}f)^{\gamma}&=(\Delta_gX)^{\gamma}+g^{ij}\ds\left(\Gamma^{\gamma}_{\alpha\beta}(f(x))\frac{\partial f^{\alpha}}{\partial x^i}\frac{\partial f^{\beta}}{\partial x^j}\right)\\
&=(\Delta_gX)^{\gamma}+g^{ij}\left(\ds\Gamma^{\gamma}_{ij}(f(x))+\Gamma^{\gamma}_{\alpha j}(x)\frac{\partial X^{\alpha}}{\partial x^i}+\Gamma^{\gamma}_{i \beta}(x)\frac{\partial X^{\beta}}{\partial x^j}\right)\\
&=(\Delta_gX)^{\gamma}+g^{ij}\left(\frac{\partial\Gamma^{\gamma}_{ij}}{\partial x^u}X^u+\Gamma^{\gamma}_{\alpha j}\frac{\partial X^{\alpha}}{\partial x^i}+\Gamma^{\gamma}_{i \beta}\frac{\partial X^{\beta}}{\partial x^j}\right).
\end{split}
\end{equation}
On the other hand, we have
\begin{equation}\label{eq:Delta X}
    (\Delta_g X)^{\gamma}=\Delta_g X^{\gamma}+g^{ij}\left(\frac{\partial X^{k}}{\partial x^j}\Gamma^{\gamma}_{ik}-\frac{\partial X^{k}}{\partial x^i}\Gamma^{\gamma}_{jk}+X^k\Gamma^{p}_{jk}\Gamma^{\gamma}_{ip}-X^p\Gamma^{\gamma}_{kp}\Gamma^{k}_{ij}+\frac{\partial\Gamma^{\gamma}_{jk}}{\partial x^i}X^k\right).
\end{equation}
Using 
$R_{ijk}^{\;\;\;\ell}=\frac{\partial\Gamma^{\ell}_{ik}}{\partial x^j}-\frac{\partial\Gamma^{\ell}_{jk}}{\partial x^i}+\left(\Gamma^{p}_{ik}\Gamma^{\ell}_{jp}-\Gamma^{p}_{jk}\Gamma^{\ell}_{ip}\right)$,
we obtain
\begin{equation}\label{eq:Delta_Ric}
    (\Delta_{g}X+\Ric_{g}X)^\gamma=g^{ij} \left( \frac{\partial^2 X^\gamma}{\partial x^i \partial x^j} 
- \Gamma^k_{ij} \frac{\partial X^\gamma}{\partial x^k} + \Gamma^\gamma_{j\beta} \frac{\partial X^\beta}{\partial x^i}
+ \Gamma^\gamma_{i\alpha} \frac{\partial X^\alpha}{\partial x^j} 
+ \frac{\partial \Gamma^\gamma_{ij}}{\partial x^\nu} X^\nu \right),
\end{equation}
which implies by \eqref{eq:X equation RHS} that
\begin{equation}
    \ds(\Delta_{g,g}f)^{\gamma}=(\Delta_g X+\Ric(X))^{\gamma}.
\end{equation}

\subsection{Harmonic map heat flow perturbation}
We say $\{f_t\}$ is a HMHF between $(M,g_t)$ and $(N,h_t)$, if $f_t:(M,g_t)\rightarrow (N,h_t)$ is a time-dependent family of smooth maps solving 
\begin{equation}\label{eq:HMHF}
    \partial_tf_t=\Delta_{g_t,h_t}f_t.
\end{equation}
Suppose $M$ is compact, and
assume under any normal coordinates we have $\|g_t-h_t\|_{C^{2,\alpha}}+\|f_t-\id\|_{C^{2,\alpha}}<\tfrac{1}{100}\min\{\textnormal{inj}_{(M,h_t)},1\}$, where the norms and derivatives are measured with respect to $h_t$. 
Then we define the \textit{difference} between $f_t$ and $\id$ to be
the vector field 
$$X_t:=\exp_{h_t,\id}^{-1}(f_t),$$ which is a section of the pull-back tangent bundle of $N$. That is, $X_t(p)=\exp_{h_t,p}^{-1}(f_t(p))$.

Recall by \eqref{eq:phi_X diffeo} and Lemma \ref{l: key lemma for regularity} we have
\begin{equation*}
    P^i(x):=f^i(x)-x^i- X^i(x)=\int_0^1 \frac{\partial^2 \exp^{i}}{\partial v^j \partial v^k}(x,sX(x))X^j(x)X^k(x)\,ds, 
\end{equation*}
which satisfies $\|P^i\|_{C^{k}}\le C\|X^i\|^2_{C^{k}}$, and the exponential map is with respect to $h_t$.
Substituting $f^i$ by $x^i+X^i+P^i$ in the equation \eqref{eq:Laplacian of map},
we obtain
\begin{multline*}
    \partial_{t} f^\gamma = g^{ij} \Bigg(\bigg( \frac{\partial^2 X^\gamma}{\partial x^i \partial x^j} 
- \Gamma^k_{ij} \frac{\partial X^\gamma}{\partial x^k} + \Gamma^\gamma_{j\beta} \frac{\partial X^\beta}{\partial x^i}
+ \Gamma^\gamma_{i\alpha} \frac{\partial X^\alpha}{\partial x^j} 
+ \frac{\partial \Gamma^\gamma_{ij}}{\partial x^\nu} X^\nu \bigg)\\
+ \left(\hat\Gamma^{\gamma}_{j\beta}   - 
\Gamma^{\gamma}_{j\beta}\right)
\frac{\partial X^\beta}{\partial x^i}
+ \bigg( \hat\Gamma^\gamma_{i\alpha}  - \Gamma^\gamma_{i\alpha} \bigg) 
\frac{\partial X^\alpha}{\partial x^j}+ \frac{\partial P^\gamma}{\partial x^i \partial x^j}-\Gamma^k_{ij}\frac{\partial P^\gamma}{\partial x^k}+\hat\Gamma^\gamma_{\alpha\beta}\frac{\partial X^\alpha}{\partial x^i}\frac{\partial X^\beta}{\partial x^j}\\
+ \bigg( \hat\Gamma^\gamma_{ij} - \Gamma^\gamma_{ij} - \frac{\partial \Gamma^\gamma_{ij}}{\partial x^\nu} X^\nu \bigg)+\hat\Gamma^\gamma_{\alpha\beta}\frac{\partial P^\alpha}{\partial x^i}\frac{\partial f^\beta}{\partial x^j}+\hat\Gamma^\gamma_{\alpha\beta}\frac{\partial f^\alpha}{\partial x^i}\frac{\partial P^\beta}{\partial x^j}\Bigg),
\end{multline*}
where $\hat\Gamma$ and $\Gamma$ are evaluated at $f=x + X + P$ and $x$ respectively.
Note \eqref{eq:Delta_Ric} implies the first term is $(\Delta_{g}X+\Ric_{g}X)^\gamma$, 
and $\partial_t X^\gamma=\partial_t f^\gamma-\partial_t P^\gamma$, it follows that $X_t$ satisfies the following \textbf{HMHF Perturbation equation at $\id$}
\begin{equation}\label{eq:HMHF Perturbation}
    \partial_t X_t=(\Delta_{h_t}+\Ric_{h_t}) X_t+E_t(X_t,{g_t-h_t})+Q_t(X_t)+G_t,
\end{equation}
where $G_t$ is independent of $X$, and
$E_t(X+Y,g-h)=E_t(X,g-h)+E_t(Y,g-h)$, $Q_t(X)=\sum_{i+j\le 2}Q_{ij}(t,X,\nabla X,\nabla^2X)\nabla ^iX*\nabla^j X$, and there is $C(g_t,h_t)>0$ such that 
\begin{equation}\label{eq:Q and E}
\begin{split}
\|E_t(X)\|_{C^{k-2,\alpha}}&\le C\|g-h\|_{C^{k,\alpha}}\|X\|_{C^{k,\alpha}},\\
    \|G_t\|_{C^{k-2,\alpha}}&\le C\|g-h\|_{C^{k,\alpha}},\\
     \|Q_{ij}\|_{C^{2,\alpha}}&\le C.
\end{split}
\end{equation}

Next, let $ f_{i,t}:(M,g_t)\to(M,h_t)$ be two HMHFs between $g_t$ and $h_t$, $t\in[0,T]$, $i=1,2$, such that the vector field $X_{i,t}:=\exp_{h_t, \id}^{-1} f_{i,t}$ is well-defined and satisfies $\|X_{i,t}\|_{C^{2,\alpha}}\le\tfrac{1}{100}\textnormal{inj}_{h_t}$. Then \eqref{eq:HMHF Perturbation} implies that $X_t:=X_{2,t}-X_{1,t}$ satisfies the following \textbf{HMHF Perturbation equation at $ f_{1,t}$},
\begin{equation}\label{eq:HMHF perturbation sec 7}
    \partial_t X_t=\Delta_{h_t} X_t+\Ric_{h_t} X_t+E_t(X_t,g_t-h_t)+Q_t(X_t),
\end{equation}
where $E_t,Q_t$ satisfy the same properties above.

\subsection{Modified HMHF and its Perturbation}

We fix a shrinker $(M,\bar g,f)$, and assume $\xi_\tau$ satisfies $\partial_\tau\xi_\tau=\nabla f\circ\xi_\tau$ with $\xi_0=\id$, and $\xi_\tau^*\bar g$ is the rescaled Ricci flow of the shrinker. Let $g_\tau$, $\tau\in[0,T]$, be a rescaled Ricci flow.
Let $\chi_\tau$ be a HMHF between the two rescaled Ricci flows $g_\tau$ and $\xi_\tau^*\bar g$.
Then $\psi_{\tau}:=\xi_{\tau}\circ{\chi}_{\tau}\circ\xi_{\tau}^{-1}$ is the \textbf{modified HMHF between the modified rescaled Ricci flow $(\xi^{-1}_{\tau})^*g_{\tau}$ and $\bar g$}, and satisfies
\begin{equation}\label{eq:M_HMHF}
\begin{split}
\partial_\tau\psi_\tau
&=\partial_\tau(\xi_{\tau}\circ{\chi}_{\tau}\circ\xi_{\tau}^{-1})\\
&=\partial_\tau(\xi_\tau\circ{\chi}_\tau)\circ\xi_\tau^{-1}+(\xi_\tau\circ{\chi}_\tau)_*(\partial_{\tau}\xi_\tau^{-1})\\
 &=(\xi_{\tau})_*(\partial_{\tau}{\chi}_{\tau})\circ\xi_{\tau}^{-1}+(\partial_\tau\xi_\tau\circ{\chi}_\tau)\circ\xi_\tau^{-1}+(\xi_\tau\circ{\chi}_\tau)_*(\partial_{\tau}\xi_\tau^{-1})\\
 &=(\xi_{\tau})_*(\Delta_{g_{\tau},\xi^*_{\tau}\bar g}{\chi}_{\tau})\circ{\xi}_{\tau}^{-1}+(\partial_\tau\xi_\tau\circ{\chi}_\tau)\circ\xi_\tau^{-1}+(\xi_\tau\circ{\chi}_\tau)_*(\partial_{\tau}\xi_\tau^{-1})\\
&=\Delta_{(\xi^{-1}_{\tau})^*g_{\tau},\bar g}\psi_{\tau}+\nabla f\circ\psi_\tau-(\psi_\tau)_*(\nabla f),
\end{split}
\end{equation}
where in the last equation we used $\partial_\tau\xi_\tau=\nabla f\circ\xi_\tau$, and $\partial_\tau\xi_\tau^{-1}=-(\xi_\tau^{-1})_*(\nabla f)$. 
Suppose $X_\tau:=\exp_{\id}^{-1}\psi_\tau$ is well-defined and $\|X_\tau\|_{C^{2,\alpha}}\le\frac{1}{100}\textnormal{inj}_{\bar g}$. Then by the same computation as \eqref{eq:HMHF Perturbation}, we see that $X_\tau$ satisfies the \textbf{modified HMHF Perturbation equation at $\id$}:
\begin{equation}\label{eq:modified_HMHF_perturbation}
    \partial_\tau X_\tau=\mathfrak L_{\bar g} X_\tau+E_\tau(X_\tau,{(\xi^{-1}_{\tau})^*g_{\tau}-\bar g})+Q_\tau(X_\tau)+G_\tau,
\end{equation}
where $\mathfrak L_{\bar g} X=\Delta_{\bar g,f} X+\tfrac12 X=\Delta_{\bar g}+\Ric_{\bar g}-[\nabla f,\cdot]$, and $E_\tau,Q_\tau,G_\tau$ satisfy the same properties as in \eqref{eq:HMHF Perturbation}. 

Moreover, let $\psi^1_\tau,\psi^2_\tau:(M,(\xi^{-1}_{\tau})^*g_{\tau})\to(M,\bar g)$ be two modified HMHFs between $(\xi^{-1}_{\tau})^*g_{\tau}$ and $\bar g$, $t\in[0,T]$, such that the vector field $X_\tau:=\exp_{\bar g,\psi^1_\tau}^{-1}\psi^2_\tau$ is well-defined and $\|X_\tau\|_{C^{2,\alpha}}\le\tfrac{1}{100}\textnormal{inj}_{\bar g}$. Then by the same computation for \eqref{eq:HMHF perturbation sec 7} we see that $X_\tau$ satisfies the \textbf{modified HMHF Perturbation equation at $\psi^1_\tau$},
\begin{equation}\label{eq:HMHF perturbation sec 7_2}
    \partial_\tau X_\tau=\mathfrak L_{\bar g} X_\tau+E_\tau(X_\tau,(\xi^{-1}_{\tau})^*g_{\tau}-\bar g)+Q_\tau(X_\tau),
\end{equation}
where $E,Q$ satisfy the same properties as in \eqref{eq:HMHF perturbation sec 7}. In particular, the linearization of this equation is $\partial_\tau X_\tau=\mathfrak L_{\bar g} X_\tau$.

Lastly, we explain the relation between the modified HMHF, Ricci flow, and Ricci DeTurck flow:
Let $g_t$ be a Ricci flow, and denote by $\bar g_t$ the Ricci flow of the shrinker $\bar g$.
Let $\phi_{t}$ be a HMHF between $g_{t}$ and $\bar g_{t}$.
Since for any $\lambda,\mu>0$ the map laplacian scales by $\Delta_{\lambda^2g,\mu^2h}f=\lambda^{-2}\Delta_{g,h}f$, it follows that $\chi_\tau:=\phi_{-e^{-\tau}}$ is a HMHF between the two corresponding rescaled Ricci flows $e^{\tau}g_{-e^{-\tau}}$ and $e^{\tau}\bar g_{-e^{-\tau}}$. 
Then $\widetilde\chi_\tau= \xi_\tau\circ\chi_\tau\circ \xi^{-1}_\tau$ is a modified HMHF between $(\xi_\tau)_*e^{\tau}g_{-e^{-\tau}}$ and $(\xi_\tau)_*e^{\tau}\bar g_{-e^{-\tau}}=\bar g$.
In particular, we have the following commutative diagram:
\[\begin{tikzcd}
e^{\tau}g_{-e^{-\tau}} \arrow[r, "\chi_\tau"] \arrow[d, "\xi_\tau"]
& e^{\tau}\bar g_{-e^{-\tau}} \arrow[d, "\xi_\tau"] \\
(\xi_\tau)_*e^{\tau}g_{-e^{-\tau}}  \arrow[r, "\widetilde\chi_\tau"]
&  \bar g
\end{tikzcd}
\]
Moreover, $h_\tau:=(\chi_\tau)_*e^{\tau}g_{-e^{-\tau}}-e^{\tau}\bar g_{-e^{-\tau}}$ satisfies the rescaled Ricci DeTurck flow perturbation equation, and $\widetilde h_\tau:=(\xi_\tau)_*h_\tau=(\widetilde\chi_\tau)_*(\xi_\tau)_*e^{\tau}g_{-e^{-\tau}}-\bar g$ satisfies the modified rescaled Ricci DeTurck flow perturbation equation.

\section{Dynamics of ODEs}
\label{appendix:ODE}
In this section, let $(M,\bar{g},f)$ be a compact shrinker, and we consider \textbf{the rescaled and modified RDTF Perturbation equation}
\begin{equation}\label{eq:rescaled RDTF Perturbation equation}
    \pt h_t=L  h_t+Q(h_t),
\end{equation}
where $L=\Delta_f+2\Rm*$ and $Q(h)=\nabla h*\nabla h+\nabla^*(h*\nabla h)$. Note that instead of the $H^1$-norm, we will consider the $H_W$-norm defined in \eqref{def:H_W} as
\begin{equation}
    \|h\|_{H_W}^2=\|h\|_{L^2_f}^2-\langle \PP^-h, L \PP^- h\rangle :=\sum_{i=1}^\infty (1+ \lambda_i^-)\|\PP_i h\|_{L^2_f}^2,
\end{equation}
where $\lambda_i^-=\max\{0,-\lambda_i\}$. Moreover, for easy of notation, we frequently denote $\|\cdot \|_{L_f^2}$ and $\langle \cdot,\cdot\rangle_{L^2_f}$ by $\|\cdot\|$ and $\langle \cdot,\cdot\rangle$, respectively.

\begin{lem}\label{prop:error.est}
For any symmetric 2-tensor $\varphi$, we have
\begin{equation}\label{eq:error estimate L2}
    \big|\langle \varphi,Q(h) \rangle_{L^2_f}\big| \leq C \|\varphi\|_{L^2_f}\|h\|_{C^2}\| h\|_{H_W}.
\end{equation}
\end{lem}

\begin{proof}
Since $|\big(\nabla h*\nabla h +\nabla^*(h*\nabla h)| \leq \|h\|_{C^2}(|h|+|\nabla h|)$, we have
\begin{equation}
    \big|\langle \varphi,Q(h) \rangle\big| \leq C\|\varphi\|\|Q(h)\|\leq C \|\varphi\|\|h\|_{C^2}\| h\|_{H_W}.
\end{equation}
\end{proof}

\begin{lem}\label{lem:general_V}
Let $\bar g$ be a compact shrinker. Let $V$ be an invariant subspace. For any $\delta \in (0,1)$, there exist $C,\eps>0$ only depending on $\bar g,\delta$ such that the following holds:
Suppose $h_\tau$, $\tau\in[0,T]$, is a smooth solution to \eqref{eq:rescaled RDTF Perturbation equation}, and satisfies $\|h_\tau\|_{C^2}\leq \varepsilon$. Let $\lambda_{\min}\in [-\infty,\infty)$ and $\lambda_{\max}\in(-\infty,\infty)$ be the minimal and maximal eigenvalues of $V$. Then
    \begin{align}
      \tfrac{d}{d\tau}e^{-2[(1+\delta)\lambda_{\min}-\delta\lambda_{\min}^+]\tau}\|\PP_Vh_\tau\|^2_{H_W}&\geq  -C\varepsilon e^{-2[(1+\delta)\lambda_{\min}-\delta\lambda_{\min}^+]\tau}\|h_\tau\|^2_{H_W},\label{eq:first1}\\              
      \tfrac{d}{d\tau}e^{-2[(1-\delta) \lambda_{\max}+\delta\lambda_{\max}^+]\tau}\|\PP_Vh_\tau\|^2_{H_W}&\leq  C\varepsilon e^{-2[(1-\delta) \lambda_{\max}+\delta\lambda_{\max}^+]\tau}\|h_\tau\|^2_{H_W},\label{eq:second2}
    \end{align}
    where $\lambda_{\max}^+=\max\{0,\lambda_{\max}\}$ and $\lambda_{\min}^+=\max\{0,\lambda_{\min}\}$.  Note that the first inequality is vacuously true when  $\lambda_{\min}=-\infty$ and $\tau>0$. 
    \end{lem}

\begin{proof}
For simplicity, we let $\PP_V^-:=\PP_V\PP^-$.
Note that we have 
    \begin{align*}
        &\langle \PP_V h,Lh\rangle\le \lambda_{\max}\|\PP_V h\|^2, &&
        -\langle L\PP^-_V h,Lh\rangle\le (\lambda_{\max}-\lambda_{\max}^+)\langle -\PP^-_V h,L\PP^-_Vh\rangle.
    \end{align*}
Hence, 
\begin{equation*}
\langle \PP_V h-(1-\delta) L\PP^-_V h,Lh\rangle\le  \delta\lambda_{\max}\|\PP_Vh\|^2+(1-\delta)\lambda_{\max}\|\PP_V h\|^2_{H_W}.
\end{equation*}
Also, by Lemma \ref{prop:error.est} we have
\begin{align*}
    \langle \PP_V h-L\PP^-_V h,Q(h)\rangle&\le \eps(\|\PP_V h\|+\|L\PP^-_V h\|)\|h\|_{H_W}\le C\eps (\|L\PP^-_V h\|^2+\|h\|^2_{H_W}).
\end{align*}
Hence, combining the above estimates with $\eps\ll\delta \leq \frac{1}{2}$, we obtain
\begin{align*}
\tfrac{1}{2} \tfrac{d}{d\tau}\|\PP_Vh \|_{H_W}^2&=\langle \PP_Vh {-L\PP_V^- h }, \partial_\tau h \rangle=\langle \PP_V h -L\PP_V^- h , L h +Q(h )\rangle\\
    &\le \langle \PP_Vh -(1-\delta) L\PP_V^- h ,   L h \rangle-\delta\|L\PP_V^- h\|^2+\langle \PP_V h -L\PP_V^- h , Q(h )\rangle\\
    &\le \langle \PP_Vh -(1-\delta) L\PP_V^- h ,   L h \rangle-\delta\|L\PP_V^- h\|^2+C\eps (\|L\PP^-_V h\|^2+\|h\|^2_{H_W}) \\
    &\le \delta\lambda_{\max}\|\PP_Vh\|^2+(1-\delta) \lambda_{\max}\|\PP_Vh\|^2_{H_W} +C\eps\|h\|^2_{H_W}.
\end{align*}
This implies \eqref{eq:second2}. We can obtain \eqref{eq:first1} in the same manner by observing
\begin{equation*}
\langle \PP_V h-(1+\delta) L\PP^-_V h,Lh\rangle\ge -\delta\lambda_{\min}\|\PP_Vh\|^2+(1+ \delta)\lambda_{\min} \|\PP_V^-h\|_{H_W}^2.
\end{equation*}
\end{proof}

For a non-zero RDTF Perturbation $h_\tau$,
the following lemma shows that the ratio between the $H_W$-norms of the projections of $h_\tau$ into two invariant subspaces will be preserved almost as in the linear case, if $h_\tau$ is sufficiently small and the denominator has a certain portion in the overall $H_W$-norm.

\begin{lem}\label{lem:B5}
Let $\bar g$ be a compact shrinker. Let $h_\tau$, $\tau\in[0,T]$, be a smooth non-zero RDTF Perturbation with background metric $\bar g$.
Let $V_1,V_2$ be two invariant subspaces with $\dim V_2<\infty$. Then, given $A,C_0>0$ and $\delta \in (0,1)$, there exists $\eps(A,C_0,\delta,V_1,V_2,T)>0$ such that the following holds: Suppose that $\|h_{\tau}\|_{C^2}\leq \varepsilon$ holds for $\tau\in [0,T]$, and also
\begin{align*}
&\|h_0\|_{H_W}^2\le C_0\|\PP_{V_2}h_0\|_{H_W}^2, &&  \|\PP_{V_1}h_0\|_{H_W}^2\le A\|\PP_{V_2}h_0\|_{H_W}^2.  
\end{align*}
Then 
  \[\|\PP_{V_1}h_\tau\|_{H_W}^2\le Ae^{2[(1-\delta)\lambda_{V_1,\max}+\delta \lambda_{V_1,\max}^+-(1+\delta)\lambda_{V_2,\min}+\delta \lambda_{V_2,\min}^++\delta]\tau}\|\PP_{V_2}h_\tau\|_{H_W}^2,\]
where $\lambda_{V_1,\max},\lambda_{V_2,\min}$ are the largest and smallest eigenvalues of $V_1,V_2$, respectively, and $\lambda_{V_1,\max}^+=\max\{0,\lambda_{V_1,\max}\}$, and $\lambda_{V_2,\min}^+=\max\{0,\lambda_{V_2,\min}\}$. 
\end{lem}

\begin{proof}
For simplicity, we write $\PP_{V_i}=\PP_i$ for each $i=1,2$, and write $\lambda_{V_1,\min}=\Lambda_1$ and $\lambda_{V_2,\max}=\Lambda_2$. We also recall the first eigenvalue $\lambda_1>0$ of the entire $L^2_f$. We consider
\begin{equation*}
a(t):=C_0 e^{-2[(1+\delta)\Lambda_2-\delta \Lambda_2^+] \tau}\|\PP_2 h_\tau\|_{H_W}^2-e^{-2(\lambda_1+\delta)\tau}\|h_\tau\|_{H_W}^2.
\end{equation*}
Then, by applying Lemma \ref{lem:general_V} for $L^2_f$ and $V_2$ to get
\begin{equation*}
a'\geq 2\delta e^{-2(\lambda_1+\delta)\tau}\|h_\tau\|_{H_W}^2-C\varepsilon (C_0e^{-2[(1+\delta)\Lambda_2-\delta \Lambda_2^+] \tau}+e^{-2(\lambda_1+\delta) \tau}) \|h_\tau\|_{H_W}^2.
\end{equation*}
Hence, given $T,\delta,\Lambda_2,C_0$, there is small $\varepsilon>0$ satisfying $a'\geq 0$ for $\tau\in [0,T]$. Since  we know $a(0)\geq 0$ by assumption,  for $\tau \in [0,T]$ we have
\begin{equation}\label{eq:growth_control}
e^{-2(\lambda_1+\delta)\tau}\|h_\tau\|_{H_W}^2 \leq C_0 e^{-2[(1+\delta)\Lambda_2-\delta \Lambda_2^+] \tau}\|\PP_2 h_\tau\|_{H_W}^2.
\end{equation}
Next, we consider
\begin{equation*}
b(t):=A e^{-2\Lambda_2' \tau}\|\PP_2 h_\tau\|_{H_W}^2-e^{-2\Lambda_1'\tau}\|\PP_1 h_\tau\|_{H_W}^2,
\end{equation*}
where $\Lambda_1'=(1-\delta)\Lambda_1+\delta \Lambda_1^+$ and $\Lambda_2'=(1+\delta)\Lambda_2-\delta \Lambda_2^+-\delta$.
Then,
\begin{align*}
b'\geq 2A\delta   e^{-2\Lambda_2' \tau}\|\PP_2 h_\tau\|_{H_W}^2
-C\varepsilon(e^{-2\Lambda_2'\tau}+ e^{-2\Lambda_1'\tau})\|h_\tau\|_{H_W}^2.
\end{align*}
Thus, using \eqref{eq:growth_control},  we can obtain $b'\geq 0$ in $\tau\in[0,T]$ for small $\varepsilon$. Therefore, together with the assumption $b(0)\geq 0$ we finish the proof.
\end{proof}

\section{Existence of solutions with exponential decay}
\label{appendix:existence}
In this section, we prove several lemmas of showing the existence of solutions that decay exponentially to to ancient or immortal PDEs. Throughout this section, $L$ denotes a general elliptic operator which acts on the sections of $(r,s)$-tensors $T^r_s M$ on a Riemannian manifold $(M,\bar g)$, and $L$ is symmetric under the weighted $L^2_f$-norm for some smooth function $f$ on $M$. We use $X_\tau$ or $h_\tau$ to denote the unknown $(r,s)$-tensor fields.
We assume the spectrum $\Spec(L)$ of $L$ has the following decomposition:
$$\lambda_1 \geq \lambda_2 \geq \cdots\ge\lambda_I>0=\cdots= \lambda_{I+k}=0 > \lambda_{I+k+1}\ge\cdots,$$ for some $I,k \geq 0$, and $L Y_i+\lambda_i Y_i=0$ where $\|Y_i\|_{L^2_f}=1$. 
In this paper, we apply results in this section to a compact shrinker $(M,\bar g,f)$, and we choose $L$ either as $\Delta_f+\tfrac12$ acting on vector fields, or $L=\Delta_f+2\Rm*$ acting on symmetric 2-tensor fields.

The following lemma shows the existence of an ancient solution to an inhomogeneous parabolic equation $(\partial_\tau-L)X_\tau=G_\tau$, where $G_\tau$ decays exponentially fast as $\tau \to - \infty$.

\begin{lem}\label{bl: inhomogeneous solution existence with asymptotic}
We consider the equation $(\partial_\tau-L)X_\tau=G_\tau$ on $(-\infty,0]$ with
\begin{equation*}
\mathcal{G}:=\sup_{\tau\leq 0}e^{-\mu  \tau}\|G_\tau\|_{C^\alpha}<\infty,
\end{equation*}
for some $\mu>0$.
Then, there is an ancient solution $X_\tau$ for $\tau\in (-\infty,0]$ such that
\begin{align*}
  \|X_\tau\|_{C^{2,\alpha}}\leq C_\theta \mathcal{G}e^{\theta  \tau},
\end{align*}
holds for any $\theta\in (0,\mu)\setminus {\Spec}(L)$, where $C_\theta$ depends only on $\mu-\theta$ and $\min_{i\in \mathbb{N}}| \theta-\lambda_i|$.
\end{lem}

\begin{proof}
Let $J\in \{0,\cdots,I\}$ be the number satisfying  $ \theta \in (\lambda_{J+1},  \lambda_J)$ considering $\lambda_0=+\infty$. Then, we define $X_\tau$ by
\begin{equation}
\begin{cases}
  \langle  X_\tau,Y_i\rangle_{L^2_f}  = -\int_{\tau}^{0}  e^{-\lambda_i(s- \tau)} \langle  G_s,Y_i\rangle_{L^2_f}  ds\qquad &\textnormal{for}\quad  i =1,\cdots, J,\\
   \langle  X_\tau,Y_i\rangle_{L^2_f}  = \int_{-\infty}^\tau e^{-\lambda_i(s- \tau)} \langle  G_s,Y_i\rangle_{L^2_f} ds\qquad &\textnormal{for}\quad  i\geq J+1.
\end{cases}
\end{equation}
Since $0\leq \lambda_{J+1}< \theta <\min\{\mu,\lambda_{J}\}$, we have
 \begin{equation*}
    \sum_{i=1}^{J}\langle  X_\tau,Y_i\rangle_{L^2_f}^2 \leq \left(\int^{0}_\tau  e^{-2\lambda_J(s-\tau)}e^{2\theta  s} ds\right) \left(\int^{0}_\tau \|G_s\|_{L^2_f}^2e^{-2\theta  s} ds\right)\leq C \mathcal{G}^2e^{2\theta  \tau},
 \end{equation*}
 where here and below $C>0$ denotes a constant depending on $\mu-\theta ,\lambda_J-\theta,\theta-\lambda_{J+1}$. Similarly, remembering $\theta-\lambda_i > 0$ for $i\geq J+1$, we obtain
 \begin{equation*}
     \sum_{i\ge J+1} \langle  X_\tau,Y_i\rangle_{L^2_f}^2 \leq \left(\int^{\tau}_{-\infty} e^{-2\lambda_{J+1}(s-\tau)} e^{2\theta  s} ds\right) \left( \int^{\tau}_{-\infty} \|G_s\|_{L^2_f}^2e^{-2\theta  s} ds\right) \leq C\mathcal{G}^2 e^{2\mu   \tau}.
 \end{equation*}
Thus, we have $\|X_\tau\| _{L^2_f}\leq C\mathcal{G}e^{\theta  \tau}$, and the interior estimates yield $\|X_\tau\|_{C^{2,\alpha}}\leq C\mathcal{G} e^{\theta \tau}$.
\end{proof}

The next lemma shows the existence of such an immortal solution.

\begin{lem}\label{bl: inhomogeneous solution existence with asymptotic II}
 We consider the equation $(\partial_\tau-L)X_\tau=G_\tau$ with 
\begin{equation*}
\mathcal{G}:=\sup_{\tau\geq 0}e^{\mu \tau}\|G_\tau\|_{C^\alpha}<\infty.
\end{equation*}
for some $\mu >0$. Then, there is an immortal solution $X_\tau$ for $\tau \in [0,\infty)$ such that 
\begin{equation*}
\|X_\tau\|_{C^{2,\alpha}}\leq C_\theta \mathcal{G}e^{-\theta\tau},
\end{equation*}
holds for any $\theta \in (0,\mu)\setminus {\Spec}(L)$, where $C_\theta$ depends only on $\mu-\theta$, and $\min\{|\theta-\lambda_i|\}$.
\end{lem}

\begin{proof}
As the proof of the previous lemma \eqref{bl: inhomogeneous solution existence with asymptotic}, we define
\begin{equation*}
\begin{cases}
    \langle  X_\tau,Y_i\rangle_{L^2_f}  =   -\int_{\tau}^{\infty} e^{-\lambda_i(s- \tau)}\langle  G_s,Y_i\rangle_{L^2_f}  ds\qquad &\textnormal{for}\quad  i \leq  J,\\
    \langle  X_\tau,Y_i\rangle_{L^2_f}  = \int_{0}^{\tau}  e^{-\lambda_i(s- \tau)} \langle  G_s,Y_i\rangle_{L^2_f}  ds\qquad &\textnormal{for}\quad  i\geq J+1.
\end{cases}
\end{equation*}
Then, we can directly compute to get the desired result as in the proof of lemma \eqref{bl: inhomogeneous solution existence with asymptotic}.
\end{proof}

The next lemma gives the existence of an ancient or immortal solution to a more general PDE with quadratic terms and an inhomogenous term.

\begin{lem}\label{lem:existence_modified_HMHF}
We consider the equation for $\tau \in I_0$, 
    \begin{equation*}
    (\partial_\tau- L)X_\tau=E_\tau(X_\tau)+Q_\tau(X_\tau)+G_\tau,
\end{equation*}
where $I_a=(-\infty,a]$  $($resp. $[a,\infty)$$)$ for $a\in \mathbb{R}$, $E_\tau(X+Y)=E_\tau(X)+E_\tau(Y)$ and
\begin{align*}
&\|E_\tau(X)\|_{C^\alpha}\le  C_0 e^{\mu \tau}\|X\|_{C^{2,\alpha}}, \qquad
 \|G_\tau\|_{C^\alpha}\leq C_0 e^{\mu\tau},\\
    &\|Q_\tau(X)-Q_\tau(Y)\|_{C^\alpha}\le C_0(\|X\|_{C^{2,\alpha}}+\|Y\|_{C^{2,\alpha}})\|X-Y\|_{C^{2,\alpha}},
\end{align*}
for some $C_0\geq 1$ and $\mu>0$  $($resp. $\mu<0$$)$. Then, given $  \theta \in (0,\mu)$ $($resp. $\theta \in (\mu,0)$$)$,  there is $T \leq 0$ $($resp. $T\geq 0$$)$ such that there is a solution $X_\tau$ defined for $\tau \in I_T$ satisfying 
\begin{equation*}
\|X_\tau\|_{C^{2,\alpha}}\le C e^{\theta\tau}
\end{equation*}
 for $\tau \in I_T$, where $C$ depends only on $\mu,\theta,\bar g$.
\end{lem}

\begin{proof}
We provide a proof only for $\tau \leq 0$, since the proof for the other case $\tau \geq 0$ is identical by replacing Lemma \ref{bl: inhomogeneous solution existence with asymptotic} with Lemma \ref{bl: inhomogeneous solution existence with asymptotic II}.
By Lemma \ref{bl: inhomogeneous solution existence with asymptotic}, there is  $X_{\tau}^0$ solving $(\partial_{\tau}-L)X_{\tau}^0=G_{\tau}$ for $\tau\in (-\infty,0]$ 
\begin{equation}\label{eq:X0_esti}
    \|X_{\tau}^0\|_{C^{2,\alpha}}\leq  C_0C_1 e^{\theta'  \tau},
\end{equation}
where $\theta'\in [\theta,\mu)\setminus \text{Spec}(L)$ and $C_1=(\mu,\theta,\theta',\bar g)\geq 1$ . Then by the assumptions and \eqref{eq:X0_esti}  we have
\begin{equation*}
    \|E_\tau(X_{\tau}^0)\|_{C^\alpha}\leq  C_0^2C_1  e^{(\mu+\theta')  \tau}, \qquad \|Q_\tau(X^0_{\tau})\|_{C^\alpha}\leq   C_0^3C_1^2  e^{2\theta'  \tau}.
\end{equation*} 
Let $\theta''=\min\{\mu,\theta'\}$, and we consider $T<0$ satisfying
\begin{equation*}
 16C_0^2C_1^2   e^{\theta'' T} \leq 1.
\end{equation*}
Then, we apply Lemma \ref{bl: inhomogeneous solution existence with asymptotic}, with replacing $(-\infty,0]$ and $\mu$ by $(-\infty,T]$ and $\theta'+\theta''$ respectively, so that we have a solution $X^1_{\tau}$ to
\begin{equation}\label{eq:X^1_evol}
(\partial_\tau-L)X_\tau^1=E_\tau(X_{\tau}^0)+Q_\tau(X^0_\tau)
\end{equation}
for $\tau\in (-\infty,T]$ such that for $\tau\leq T$
\begin{equation*}
    \|X^1_\tau\|_{C^{2,\alpha}} \leq C_1 \big[\sup_{s\leq 0} 2C_0^3C_1^2 e^{-\theta' s }e^{(\theta'+\theta'')(s+T)}\big] e^{\theta'(\tau-T)}=2^{-1}C_0C_1e^{\theta' \tau}.
\end{equation*}
This implies $\|X^1_\tau+X^0_\tau\|_{C^{2,\alpha}}+\|X^0_\tau\|_{C^{2,\alpha}}\leq 4C_0C_1 e^{\theta'\tau}$. Then, by the assumptions we have
\begin{align*}
    \|E_\tau(X_{\tau}^1)\|_{C^\alpha} \leq 2^{-1}C_0^2C_1 e^{(\mu+\theta')  \tau}, \qquad \|Q_\tau(X^1_\tau+X^0_\tau)-Q_\tau(X^0_\tau)\|_{C^\alpha}\leq  2C_0^3C_1^2 e^{2\theta'  \tau}.
\end{align*}
Then, again we consider a solution $X^2_\tau$ to
\begin{equation*}
(\partial_\tau-L)X^2_\tau=E_\tau(X_{\tau}^1)+Q_\tau(X^1_\tau+X^0_\tau)-Q_\tau(X^0_\tau),
\end{equation*} 
which satisfies $\|X^2_\tau\|_{C^{2,\alpha}}\leq  2^{-2}C_0C_1 e^{\theta'\tau}$ for $\tau \leq T$. Then, we have
\begin{align*}
    \|E_\tau(X_{\tau}^2)\|_{C^\alpha} \leq 2^{-2}C_0^2C_1 e^{(\mu+\theta')  \tau}, \qquad \|Q_\tau(\bar X^2_\tau)-Q(\bar X^1_\tau)\|_{C^\alpha}\leq   C_0^3C_1^2 e^{2\theta'  \tau},
\end{align*}
where $\bar X^k_\tau:=X^0_\tau+\cdots X^k_\tau$, so that we get $\|X^3_\tau\|_{C^{2,\alpha}}\leq 2^{-3}C_0C_1e^{\theta'\tau}$ for $\tau \leq T$. By iterating this process, we obtain a sequence $\{X^k_\tau\}_{k=1}^\infty$ solving 
\begin{equation}\label{eq:X^k_evol}
(\partial_\tau-L)X^k_\tau=E_\tau(X^{k-1}_\tau)+Q_\tau(\bar X^{k-1}_\tau)-Q_\tau(\bar X^{k-2}_\tau),
\end{equation}
and $\|X^{k}_\tau\|_{C^{2,\alpha}} \leq 2^{-k}C_0C_1e^{\theta' \tau}$ for $\tau \leq T$. Then, $\bar X^k_\tau$ converges to $X_\tau$ in the $C^{2,\alpha}$ sense, which is a solution to $(\partial_\tau-L)X_\tau=E_\tau(X_\tau)+Q_\tau(X_\tau)+G_\tau$ satisfying the desired assertions on $I_T=(-\infty,T]$.
\end{proof}

The next lemma says that if the error terms decay arbitrarily fast, there is a solution that has the same convergence rate.
In this paper, we apply this lemma to the HMHF Perturbation equation \eqref{eq:HMHF Perturbation} between a fast decay Ricci flow and a shrinker, and it produces a fast convergent HMHF.

\begin{lem}\label{l: existence to time-dependent, inhomogeneous equation-2}
We consider the equation on $[0,\infty)$, 
    \begin{equation*}
    (\partial_\tau- L)X_\tau=E_\tau(X_\tau)+Q_\tau(X_\tau)+G_\tau.
\end{equation*}
Suppose $E_\tau(X+Y)=E_\tau(X)+E_\tau(Y)$, and for any $\lambda>0$ there exist $C_0(\lambda),T(\lambda)>0$ such that 
\begin{align*}
&\|E_\tau(X)\|_{C^\alpha}\le  C_0 e^{-\lambda \tau}\|X\|_{C^{2,\alpha}}, \qquad
 \|G_\tau\|_{C^\alpha}\leq C_0 e^{-\lambda\tau},\\
    &\|Q_\tau(X)-Q_\tau(Y)\|_{C^\alpha}\le C_0(\|X\|_{C^{2,\alpha}}+\|Y\|_{C^{2,\alpha}})\|X-Y\|_{C^{2,\alpha}},
\end{align*}
for all $\tau\ge T(\lambda)$. Then there exist $T_0>0$ and a solution $X_\tau$ on $[T_0,\infty)$ satisfying 
\begin{equation*}
     \|X_\tau\|_{C^{2,\alpha}}= O(e^{-\lambda \tau}).
 \end{equation*}
as $\tau \to \infty$ for every $\lambda>0$.
\end{lem}

\begin{proof} 
As in the proof of Lemma \ref{lem:existence_modified_HMHF} with $\mu=-2$ and $\theta'\in (-2,-1]\setminus \text{Spec}(L)$, there are $T_0 >0$ and a sequence $X^i_\tau$ such that 
\begin{equation*}
\|X^i_\tau\|_{C^{2,\alpha}}\leq 2^{-i}C_0C_1e^{\theta'\tau}
\end{equation*}
for $\tau \geq T_0$ where $C_1=(\theta,\bar g)$, and moreover $(\partial_\tau-L)X^0_\tau =G_\tau$, \eqref{eq:X^1_evol}, and \eqref{eq:X^k_evol} hold. Then, the solution $X_\tau=\sum_{k=0}^\infty \bar X^k_\tau$ exists for $\tau \geq T_0$. 

Here, we notice that each $X^k_\tau=\sum_i \langle X^k_\tau,Y_i\rangle_{L^2_f}Y_i$ is explicitly constructed in the proof of Lemma \ref{bl: inhomogeneous solution existence with asymptotic II}. In particular, the constructed solution $X^k_\tau$ are independent of the decay rates $\lambda$. Remembering this construction, given $\lambda >2$ we choose $\lambda' \in [\lambda-1,\lambda)\setminus \text{Spec}(L)$. Then, there is $T_{\lambda'}$ such that
\begin{equation*}
\|X^i_\tau\|_{C^{2,\alpha}}\leq 2^{-i}C_0C' e^{-\lambda'\tau}
\end{equation*}
holds for $\tau \geq T_{\lambda'}$, where $C'$ depends only on $\lambda',\bar g$. Hence, the solution $X_\tau=\lim X^i_\tau$ satisfies $\|X_\tau\|_{C^{2,\alpha}}=O(e^{-\lambda'\tau})$ as $\tau \to +\infty$. Since $\lambda \leq \lambda'-1$ is arbitrarily large, this completes the proof.
\end{proof}

The next lemma studies the equation $(\partial_\tau- L)X_\tau=E_\tau(X_\tau)+Q_\tau(X_\tau)$, such as \eqref{eq:most important eq} and \eqref{eq:HMHF perturbation sec 7}.
For any given solution $X^0_\tau$ to this equation, it finds another solution whose difference from $X^0_\tau$ is dominated by $a e^{-\lambda_i\tau}Y_i$ for some eigenvalue $\lambda_i$ and eigenvector $Y_i$. 
We use this lemma in Section \ref{sec: existence} to construct both ancient and immortal RDTFs and HMHFs with prescribed asymptotics.

\begin{theorem}\label{l: contraction mapping RDTF}
Consider the equation
\begin{equation}\label{eq:both_RDTF_HMHF}
    (\partial_\tau- L)X_\tau=E_\tau(X_\tau)+Q_\tau(X_\tau),
\end{equation}
where 
$E_\tau(X+Y)=E_\tau(X)+E_\tau(Y)$ and for some $C_0>0$ that 
\begin{align*}
\|E_\tau(X)\|_{C^\alpha}&\le  C_0 e^{\mu \tau}\|X\|_{C^{2,\alpha}},\\ 
    \|Q_\tau(X)-Q_\tau(Y)\|_{C^\alpha}&\le C_0(\|X\|_{C^{2,\alpha}}+\|Y\|_{C^{2,\alpha}})\|X-Y\|_{C^{2,\alpha}}.
\end{align*}
Suppose $X^0_\tau$ is a solution to \eqref{eq:both_RDTF_HMHF} on $I=(-\infty,0]$ $($resp. $ [0,\infty)$$)$ satisfying 
\begin{equation*}
 \|X_\tau^0\|_{C^{2,\alpha}}\leq C_1 e^{\delta\tau},
\end{equation*}
for some $C_1>0$ and  $\delta>0$ $($resp. $\delta<0$$)$.
Then, given $a\in \mathbb{R}$ and a unit eigentensor $Y_i$ of $L$ with eigenvalue $\lambda_i>0$ $($resp. $\lambda_i<0$$)$ and $\|Y_i\|_{L^2_f}=1$, there is a solution $X_\tau$ to \eqref{eq:both_RDTF_HMHF} such that  
    \begin{equation}\label{eq:contraction strata on base}
\|X_\tau-X^0_\tau -ae^{\lambda_i \tau} Y_i \|_{C^{2,\alpha}}=  O(e^{(\delta'+\lambda_i) \tau}),
    \end{equation}
holds as $\tau\to -\infty$ $($resp. $\tau\to \infty$$)$ for some $\delta'>0$ $($resp. $\delta'<0$$)$.
\end{theorem}

\begin{proof} We only prove for $\tau\leq 0$. Then, the other case $\tau\geq 0$ can be proven in the same manner by replacing Lemma \ref{bl: inhomogeneous solution existence with asymptotic} with Lemma \ref{bl: inhomogeneous solution existence with asymptotic II}.

Instead of $X_\tau$, we consider a solution $Z_\tau=X_\tau-X^0_\tau$ to the equation
\begin{equation*}
(\partial_\tau-L)Z_\tau=E_\tau(Z_\tau)+ Q'_\tau(Z_\tau),
\end{equation*}
where $Q'_\tau(X)=Q_\tau(X+X^0_\tau)-Q_\tau(X^0_\tau)$. Let $Z_{\tau}^0:=ae^{\lambda_i \tau} Y_i$. Then, 
\begin{equation*}
(\partial_{\tau}-L)Z_{\tau}^0=0, \qquad    \|Z_{\tau}^0\|_{C^{2,\alpha}}\leq C_2 e^{ \lambda_i  \tau},
\end{equation*}
for some $C_2\geq \max\{1,C_0,C_1\}$ depending on $|a|,Y_i$. Hence, by the assumptions
\begin{equation*}
\| E_\tau(Z_{\tau}^0)\|_{C^\alpha}\leq   C_2^2  e^{(\mu+\lambda_i)  \tau}, \qquad \|Q'_\tau(Z_{\tau})\|_{C^\alpha}\leq  C_2^3 (e^{2\lambda_i  \tau}+2e^{( \delta+\lambda_i)  \tau})
\end{equation*}
for $\tau \leq T$. Now, we consider $ \theta \in (\lambda_i, \theta'+\lambda_i)\setminus \text{Spec}(L)$, where $\theta':=\min\{\mu,\delta, \lambda_i\}$. Then, by Lemma \ref{bl: inhomogeneous solution existence with asymptotic}, given $T<0$ there is a solution $Z^1_\tau$ to
\begin{equation*}
(\partial_\tau-L)Z^1_\tau=E_\tau(Z^0_\tau)+Q'_\tau(Z^0_\tau)
\end{equation*}
for $\tau \leq T$, satisfying 
\begin{equation*}
\|Z^1_\tau\|_{C^{2,\alpha}}\leq C_\theta e^{\theta (\tau-T)}\sup_{s\leq T}e^{-(\theta'+\lambda_i)(s-T)}\|E_\tau(Z^0_s)+Q'_\tau(Z^0_s)\|_{C^\alpha},
\end{equation*}
where $C_\theta$ only depends on $\theta'+\lambda_i-\theta$ and $\min_{i\in \mathbb{N}}|\theta-\lambda_i|$. We assume 
\begin{equation*}
100C_2^2C_\theta e^{(\theta'+\lambda_i-\theta)T}\leq 1. 
\end{equation*}
Then, $\|Z^1_\tau\|_{C^{2,\alpha}} \leq 2^{-1}C_2e^{\theta\tau}$.

Next, as in the proof of Lemma \ref{lem:existence_modified_HMHF}, for $k\geq 2$ we inductively consider a solution $Z^k_\tau$ to
\begin{equation*}
(\partial_\tau-L)Z^k_\tau=E_\tau(Z^k_\tau)+Q'_\tau(\bar Z^{k-1}_\tau)-Q'_\tau(Z^{k-2}_\tau),
\end{equation*}
where $\bar Z^k_\tau=Z^0_\tau+\cdots+Z^k_\tau$. Then, we can obtain $\|Z^k_\tau\|_{C^{2,\alpha}}\leq 2^{-k}C_2e^{\theta\tau}$ so that we get a desired solution $X_\tau:=X^0_\tau+Z^0_\tau+\sum_{k=1}^\infty Z^k_\tau$ satisfying
\begin{equation*}
\|X_\tau-X^0_\tau-ae^{\lambda_i\tau}Y_i\|_{C^{2,\alpha}}=\|\sum_{k=1}^\infty Z^k_\tau\|_{C^{2,\alpha}}=O(e^{\theta\tau}).
\end{equation*}
This completes the proof.
\end{proof}

\bibliography{ref.bib}
\bibliographystyle{abbrv}

\end{document}